\documentclass[reqno,12pt]{amsart}

\makeatletter

\newlength{\myabovedisplayskip}
\newlength{\mybelowdisplayskip}
\newlength{\myabovedisplayshortskip}
\newlength{\mybelowdisplayshortskip}

\def\setdisplayskips{%
	\abovedisplayskip=\myabovedisplayskip
	\belowdisplayskip=\mybelowdisplayskip
	\abovedisplayshortskip=\myabovedisplayshortskip
	\belowdisplayshortskip=\mybelowdisplayshortskip
}

\usepackage{etoolbox}

\newcommand{\patchmathenv}[1]{%
	\AtBeginEnvironment{#1}{\setdisplayskips}%
}

\patchmathenv{equation}
\patchmathenv{equation*}
\patchmathenv{align}
\patchmathenv{align*}
\patchmathenv{gather}
\patchmathenv{gather*}
\patchmathenv{multline}
\patchmathenv{multline*}
\patchmathenv{flalign}
\patchmathenv{flalign*}
\patchmathenv{alignat}
\patchmathenv{alignat*}

\let\orig@itemize\itemize
\let\orig@enditemize\enditemize

\renewenvironment{itemize}{%
	\orig@itemize
	\setlength{\topsep}{10pt plus 5pt minus 5pt}
	\setlength{\partopsep}{0pt}
	\setlength{\parsep}{5pt plus 2.5pt minus 1pt}
	\setlength{\itemsep}{12pt plus 1pt minus 1pt}
	\setlength{\leftmargin}{25pt}
	\setlength{\labelwidth}{20pt}
	\setlength{\labelsep}{5pt}
}{%
	\orig@enditemize
}

\makeatother

\usepackage{mathtools}
\usepackage{amssymb}
\usepackage{bm}
\usepackage{amsthm}
\usepackage{cite}
\usepackage[breaklinks=true]{hyperref}
\hypersetup{colorlinks=true, linkcolor=blue!90!black, citecolor=orange!80!black, urlcolor=blue!90!black}

\usepackage{tikz, float, ifthen}
\usetikzlibrary{decorations.markings, arrows.meta}
\tikzset{->-/.style={line width=0.8pt,decoration={
			markings,
			mark=at position 0.56 with {\arrow{Triangle[length=1.7mm,width=1.3mm]}}},postaction={decorate}}}

\numberwithin{equation}{section}

\theoremstyle{definition}
\newtheorem{remark}{Remark}
\newtheorem{notation}{Notation}

\theoremstyle{plain}
\newtheorem{corollary}{Corollary}
\newtheorem{proposition}{Proposition}
\newtheorem{lemma}{Lemma}
\newtheorem{theorem}{Theorem}

\renewcommand{\epsilon}{\varepsilon}
\renewcommand{\phi}{\varphi}
\renewcommand{\imath}{\mathrm{i}}
\DeclareMathOperator{\sh}{sh}
\DeclareMathOperator{\ch}{ch}
\DeclareMathOperator{\sign}{sign}
\let\Re\relax
\let\Im\relax
\DeclareMathOperator{\Re}{Re}
\DeclareMathOperator{\Im}{Im}
\DeclareMathOperator{\Ln}{Ln}

\newcommand{\Id}{\mathcal{I}}
\newcommand{\Jd}{\mathcal{J}}
\newcommand{\Idc}{\mathcal{I}_c}
\newcommand{\Jdc}{\mathcal{J}_c}
\newcommand{\Gd}{\mathcal{G}}
\newcommand{\Gdc}{\mathcal{G}_c}
\newcommand{\dm}{\delta_{\mathrm{max}}}

\begin{document}

\title[From hyperbolic to complex Euler integrals]
{From hyperbolic to complex Euler integrals}

\author{N. \,M. Belousov}

\address{N. B.: Beijing Institute of Mathematical Sciences and Applications,
Huairou district, Beijing, 101408, China}

\author{G. \,A. Sarkissian}

\address{G. S.: Laboratory of Theoretical Physics,
JINR, Dubna, Moscow region, 141980 Russia and
Yerevan Physics Institute, Alikhanian Br. 2, 0036\, Yerevan, Armenia}

\author{V. \,P. Spiridonov}

\address{V. S.: Laboratory of Theoretical Physics,
JINR, Dubna, Moscow region, 141980 Russia and
National Research University Higher School of Economics, Moscow, Russia}

\begin{abstract} \noindent
Hyperbolic hypergeometric integrals are defined as Barnes-type integrals of products of hyperbolic gamma functions. Their reduction to ordinary hypergeometric functions is well known. We study in detail their degeneration to complex hypergeometric functions. Namely, using uniform bounds on the integrands, we prove that the univariate hyperbolic beta integral and the conical function degenerate to two-dimensional integrals over the complex plane.
\end{abstract}

\maketitle

\tableofcontents

\newpage

\section{Introduction}

The Euler beta integral
\begin{align} \label{Euler-beta}
	\int_0^1 t^{a - 1} (1 - t)^{b - 1} \, dt= \frac{\Gamma(a) \Gamma(b)}{\Gamma(a + b)}, \qquad \Re a, \Re b > 0,
\end{align}
is one of the simplest integrals that can be expressed in terms of the gamma function $\Gamma(x)$.
The gamma function admits various generalizations depending on additional parameters, see Figure~\ref{fig:gamma}. Correspondingly, there are variants of the beta integral, and more generally, of hypergeometric functions, for each type of the gamma function \cite{essays}.

Various reductions of elliptic hypergeometric functions were considered in ~\cite{BRS, Ra, SS}.
Most of them were rigorously proved using the uniform bounds established in \cite{Ra}.
The main goal of the present paper is to prove new limiting relations between hyperbolic
and complex rational Euler integrals, which require more refined uniform estimates. 

\begin{figure}[h]\centering
	
	\begin{tikzpicture}
		\def\h{1}
		\def\l{3}
		\def\a{0.6}
		\def\b{1}
		\def\s{1.3}
		
		\draw[->] (-0.1, 0) node[above] {$\;\;$elliptic} -- (-\l, -\h) node[yshift = -0.6cm] {trigonometric};
		\draw[ ->] (0.1, 0) -- (\l, -\h) node[yshift = -0.6cm] {hyperbolic};
		\draw[dash pattern=on 6pt off 4pt, ->] (-0.4*\l, -\h - \a) -- (0.4*\l, -\h - \a);
		
		\draw[->] (-\l, -\h - \b) -- (-\l, -2*\h - \b) node[yshift = -0.6cm] {rational (Euler)};
		\draw[->] (\l - \s, -\h - \b) -- (-\l + \s, -2*\h - \b);
		\draw[->] (\l, -\h - \b) -- (\l, -2*\h - \b) node[yshift = -0.6cm] {complex rational};
		\draw[dash pattern=on 6pt off 4pt, ->] (-0.35*\l, -2*\h - \a - \b) -- (0.35*\l, -2*\h - \a - \b);
		
	\end{tikzpicture}
	
	\caption{Types of gamma functions: solid arrows represent limits, dashed arrows correspond to algebraic connections} \label{fig:gamma}
\end{figure}
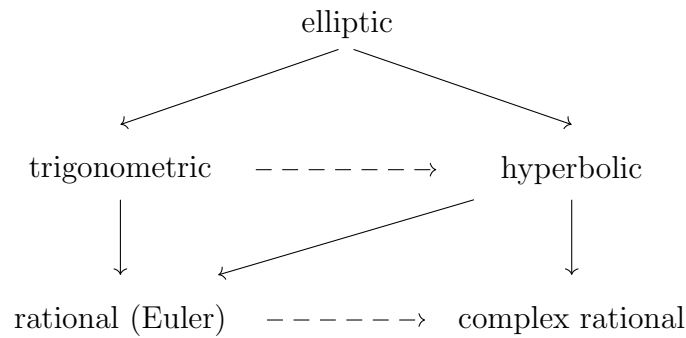

In Section~\ref{sec:gammas} we recall well-known properties of the complex rational and hyperbolic gamma functions, as well as the corresponding beta integrals. Next, in Section~\ref{sec:beta-lim} we describe the limiting procedure that relates these beta integrals. In Section~\ref{sec:conic} we show that the prescribed procedure also works for the Euler-type representation of the complex conical function. In Section~\ref{sec:hypergeometry} we outline how the same technique can be applied to the hyperbolic hypergeometric function of Ruijsenaars. In Section~\ref{sec:concl} we indicate several possible directions for further study. Finally, in the Appendices we prove some uniform bounds for the gamma functions and provide related technical statements.

\section{Properties of the gamma functions} \label{sec:gammas}

\subsection{Complex rational case} \label{sec:comp-gamma}

For $z \in \mathbb{C}$ and a pair $a, a' \in \mathbb{C}$ denote
\begin{align}
	[z]^a = z^a \bar{z}^{a'} = |z|^{a + a'} e^{\imath (a - a') \arg z}, \qquad \int_{\mathbb{C}} d^2z = \int_{\mathbb{R}^2} d\Re z \, d \Im z.
\end{align}
The double power $[z]^a$ is a single-valued function of $z$ if $a - a' \in \mathbb{Z}$. In all formulas where $a$ is an explicit integer we set $a' = a$, for example $[z]^1 = |z|^2$.

For a pair $a, a' \in \mathbb{C}$ such that $a - a' \in \mathbb{Z}$ the gamma function associated with the complex field is given by
\begin{align} \label{cgamma-def}
	\bm{\Gamma}(a|a') = \frac{1}{\pi} \int_{\mathbb{C}} [z]^{a - 1} e^{\bar{z} - z} \, d^2z= \frac{\Gamma(a)}{\Gamma(1 - a')},
\end{align}
see~\cite[Section 1.4]{GGR}. The integral conditionally converges if $0 <\Re(a+ a') <1$, while the right-hand side is clearly meromorphic in $a,a'\in \mathbb{C}$. It satisfies two difference equations
\begin{align}
	\bm{\Gamma}(a + 1|a') = a \bm{\Gamma}(a|a'), \qquad \bm{\Gamma}(a|a' + 1) =-a' \bm{\Gamma}(a|a')
\end{align}
and has the properties
\begin{align}
	\bm{\Gamma}(a|a') = (-1)^{a - a'} \bm{\Gamma}(a'|a), \qquad \bm{\Gamma}(a|a') \bm{\Gamma}(1 - a| 1 - a') = (-1)^{a - a'},
\end{align}
which follow from the reflection formula $\Gamma(a) \Gamma(1 - a) = \pi/\sin(\pi a)$.

The complex analogue of the Euler integral~\eqref{Euler-beta} has the form
\begin{align} \label{cgamma-beta}
	\frac{1}{\pi} \int_{\mathbb{C}} [z]^{a - 1} [1 - z]^{b - 1} d^2z = \frac{\bm{\Gamma}(a|a') \bm{\Gamma}(b|b')}{\bm{\Gamma}(a+b|a'+b')},
\end{align}
see~\cite[p. 2]{N} and references therein. Here we assume $a - a', b - b' \in \mathbb{Z}$ to have single-valued integrand and
\begin{align}
	\Re (a + a') > 0, \qquad \Re(b + b') > 0, \qquad \Re(a + a' + b + b') < 2
\end{align}
for the integral to be (absolutely) convergent.

\subsection{Hyperbolic case} \label{sec:hyp-gamma}

It is instructive to first recall the trigonometric gamma function due to Jackson
\begin{align}\label{qgamma}
	\Gamma_q(z) = \frac{(q; q)_{\infty}}{(q^z; q)_{\infty}} \, (1 - q)^{1 - z}, \qquad (x; q)_{\infty} = \prod_{k = 0}^\infty (1 - xq^{k}),
\end{align}
which satisfies  equations
\begin{align}
	\Gamma_q(z+1) =\frac{1-q^z}{1-q}\Gamma_q(z), \qquad \Gamma_q\biggl(z+\frac{2\pi \imath}{\ln q}\biggr) =e^{-2\pi \imath\frac{\ln (1-q)}{\ln q}} \Gamma_q(z).
\end{align}
It reduces to the Euler gamma function in the limit $\lim_{q \to 1} \Gamma_q(x) = \Gamma(x)$ \cite{K}. The infinite $q$-products in \eqref{qgamma} converge only if $|q| < 1$.

The hyperbolic gamma function \cite{R} is a variant of $q$-gamma function that remains well defined at $|q| = 1$.
With different conventions it is also known as the ``modular quantum dilogarithm'' \cite{F} and its reciprocal introduced in \cite{Sh} is called the ``double sine'' function. We define it as
\begin{align} \label{hgamma-def}
	\gamma^{(2)}(z) \equiv \gamma^{(2)}(z; \omega_1, \omega_2) = e^{- \frac{\pi \imath}{2} B_{2,2}(z; \omega_1, \omega_2)} \, \gamma(z; \omega_1, \omega_2)
\end{align}
where $B_{2,2}$ is the second order multiple Bernoulli polynomial
\begin{align} \label{B22-def}
	B_{2,2}(z;\omega_1, \omega_2) = \frac{1}{\omega_1\omega_2} \left( \Bigl(z-\frac{\omega_1+\omega_2}{2} \Bigr)^2-\frac{\omega_1^2+\omega_2^2}{12}\right)
\end{align}
and
\begin{align} \label{hgamma-int}
	\gamma(z; \omega_1, \omega_2) = \frac{( \tilde{q} e^{2\pi \imath \frac{z}{\omega_1}}; \tilde{q} )_{\infty}}{( e^{2\pi \imath \frac{z}{\omega_2}}; q )_{\infty}} = \exp\left(-\int_{\mathbb{R}+\imath 0}\frac{e^{zt}}{(1-e^{\omega_1 t})(1-e^{\omega_2 t})}\frac{dt}{t}\right)
\end{align}
with parameters
\begin{align}
	q = e^{2\pi \imath \frac{\omega_1}{\omega_2}}, \qquad \tilde{q} = e^{-2\pi \imath \frac{\omega_2}{\omega_1}}.
\end{align}
On the one hand, for $\Im \omega_1/\omega_2 > 0$
both $q$-products in \eqref{hgamma-int} converge and in this case the hyperbolic gamma function
is essentially equal to the ratio of two trigonometric ones.
This is similar to the relation between rational and complex rational gamma functions. On the other hand, the integral representation~\eqref{hgamma-int} is well defined under assumptions $0 < \Re z < \Re(\omega_1 + \omega_2)$ and $\Re \omega_1, \Re\omega_2 \geq 0$, which admits $|q| = |\tilde{q}| = 1$. In what follows we always assume $\Re \omega_1, \Re \omega_2 \geq 0$.

\begin{figure}[t]
	\centering
	\begin{tikzpicture}[line cap = round, thick]
		\def\xm{4.3} 
		\def\ym{3.4} 
		\def\z{0.4} 
		
		\def\ofx{1}
		\def\ofy{1}
		\def\osx{1.1}
		\def\osy{-0.6}
		\def\r{1pt}
		\def\s{0.7}
		\pgfmathsetmacro{\ofsx}{\ofx+\osx}
		\pgfmathsetmacro{\ofsy}{\ofy+\osy}
		
		\pgfmathsetmacro{\bx}{\s*(\ofsx + 3.3*\osx)}
		\pgfmathsetmacro{\by}{\s*(\ofsy + 3.3*\osy)}
		\pgfmathsetmacro{\mx}{\s*\ofsx}
		\pgfmathsetmacro{\my}{\s*\ofsy}		
		\pgfmathsetmacro{\tx}{\s*(\ofsx + 3.3*\ofx)}
		\pgfmathsetmacro{\ty}{\s*(\ofsy + 3.3*\ofy)}
		\draw[gray!30!white, dashed]  (\bx, \by) -- (\mx, \my) node[xshift = -0.6cm, yshift = 0.3cm] {\footnotesize \color{black} $\omega_1 + \omega_2$} -- ( \tx, \ty ); 		
		
		\pgfmathsetmacro{\bsx}{\s*(-3.3*\osx)}
		\pgfmathsetmacro{\bsy}{\s*(-3.3*\osy)}
		\pgfmathsetmacro{\msx}{0}
		\pgfmathsetmacro{\msy}{0}		
		\pgfmathsetmacro{\tsx}{\s*(- 3.3*\ofx)}
		\pgfmathsetmacro{\tsy}{\s*(- 3.3*\ofy)}
		\draw[gray!30!white, dashed]  (\bsx, \bsy) -- (\msx, \msy)-- ( \tsx, \tsy ); 		
		
		\draw[gray!40!white, ->] (-\xm + 1.7, 0) -- (\xm, 0);
		\draw[gray!40!white, ->] (0, -\ym + 0.7) -- (0, \ym);
		
		\draw[gray!50] (\xm - \z, \ym) -- (\xm - \z, \ym - \z) -- node[text = gray!90, yshift = 0.2cm] {\small $z$} (\xm, \ym - \z);
		
		\foreach \n in {0, 1, ..., 5}{
			\foreach \k in {0, 1, ..., 6}{
				\pgfmathparse{int(\n + \k)}
				\ifnum \pgfmathresult < 4
				\pgfmathsetmacro{\x}{\s*(\ofsx+\n*\ofx + \k*\osx)}
				\pgfmathsetmacro{\y}{\s*(\ofsy+\n*\ofy + \k*\osy)}
				\draw[thin, fill = white] (\x, \y) circle (\r);
				\fi
			}
		}
		
		\foreach \n in {0, 1, ..., 5}{
			\foreach \k in {0, 1, ..., 6}{
				\pgfmathparse{int(\n + \k)}
				\ifnum \pgfmathresult < 4
				\pgfmathsetmacro{\x}{\s*(-\n*\ofx - \k*\osx)}
				\pgfmathsetmacro{\y}{\s*(-\n*\ofy - \k*\osy)}
				\filldraw (\x, \y) circle (\r);
				\fi
			}
		}
		
	\end{tikzpicture}
	\caption{Poles (black) and zeros (white) of hyperbolic gamma function} \label{fig:poles-zeros}
\end{figure}
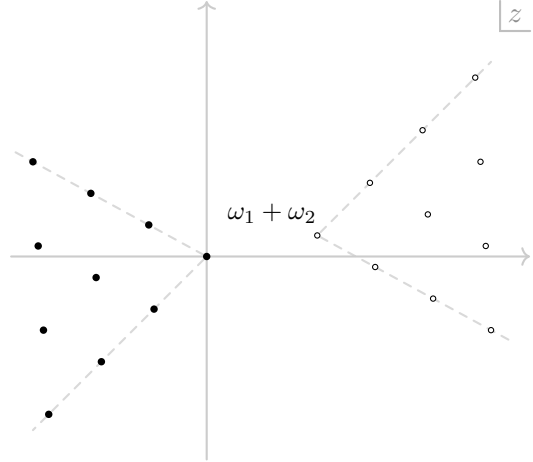

The hyperbolic gamma function satisfies two difference equations
\begin{align}
	\frac{\gamma^{(2)}(z + \omega_1) }{\gamma^{(2)}(z)}= 2\sin \frac{\pi z}{\omega_2}, \qquad \frac{\gamma^{(2)}(z + \omega_2)}{\gamma^{(2)}(z)} = 2\sin \frac{\pi z}{\omega_1}
\end{align}
and the reflection relation
\begin{align} \label{hgamma-refl}
	\gamma^{(2)}(z)  \gamma^{(2)}(\omega_1 + \omega_2 - z) = 1.
\end{align}
As clear from representation~\eqref{hgamma-int}, it also has properties
\begin{align}
	\gamma^{(2)}(z; \omega_1, \omega_2) = \gamma^{(2)}(z; \omega_2, \omega_1) , \quad \gamma^{(2)}(az; a\omega_1, a\omega_2) = \gamma^{(2)}(z; \omega_1, \omega_2), \quad a> 0
\end{align}
and it obeys the complex conjugation rule
$\overline{\gamma^{(2)}(z; \omega_1, \omega_2)}=\gamma^{(2)}(\bar z; \bar\omega_1, \bar\omega_2)$.
Besides, $\gamma^{(2)}(z)$ is a meromorphic function with the poles and zeros at the points
\begin{align}
	z_{\mathrm{poles}} = - m_1 \omega_1 - m_2 \omega_2, \qquad z_{\mathrm{zeros}} = \omega_1 + \omega_2 + m_1 \omega_1 + m_2 \omega_2, \qquad m_1, m_2 \in \mathbb{Z}_{\geq 0},
\end{align}
which form two wedges, see Figure~\ref{fig:poles-zeros}. Since $\gamma^{(2)}(z; \omega_1, \omega_2)$ is symmetric in $\omega_1, \omega_2$, without loss of generality we assume that $\arg \omega_1 \geq \arg\omega_2$. Then we have the following asymptotics outside of wedges
\begin{align} \label{hgamma-asymp}
	\begin{aligned}
		& \lim_{z\to \infty} e^{\frac{\pi\imath}{2} B_{2,2}(z; \omega_1,\omega_2)}\gamma^{(2)}(z)=1,
		&& \quad\arg \omega_1<\arg z<\arg \omega_2+\pi, \\[6pt]
		& \lim_{z\to \infty} e^{-\frac{\pi\imath}{2} B_{2,2}(z; \omega_1,\omega_2)} \gamma^{(2)}(z)=1,
		&& \quad \arg\omega_1-\pi<\arg z<\arg \omega_2.
	\end{aligned}
\end{align}
Proofs of the above properties can be found in~\cite[Section III.A]{R}, where the function
\begin{align} \label{G-Ruij}
	G(\omega_1, \omega_2; z) = \gamma^{(2)}\biggl( \frac{\omega_1 + \omega_2}{2} - \imath z; \omega_1, \omega_2 \biggr)
\end{align}
is used.

For brevity, denote
\begin{align}
	\gamma^{(2)}(a \pm b) = \gamma^{(2)}(a + b) \, \gamma^{(2)}(a - b).
\end{align}
The hyperbolic gamma function satisfies the Fourier transform identity~\cite[Proposition C.1]{R2}
\begin{align}\label{hgamma-beta}
	\int_{\imath \mathbb{R}} e^{\frac{2\pi \imath}{\omega_1 \omega_2} \lambda z} \, \gamma^{(2)}\biggl(\pm z + \frac{\omega_1 + \omega_2}{2} - g \biggr) \, \frac{dz}{\imath \sqrt{\omega_1 \omega_2}} = \frac{\gamma^{(2)} ( \pm \lambda + g )}{\gamma^{(2)}(2g)}.
\end{align}
Here we assume $| \Re \lambda| < \Re g < \Re (\omega_1 + \omega_2)/2$: the first inequality ensures convergence, while the second one guarantees that integration contour separates two series of integrand poles.

The above identity is the hyperbolic analogue of Euler beta integral~\eqref{Euler-beta}. Indeed, in the limit~$\omega_1 \to 0^+$ (that is $q \to 1$) the hyperbolic gamma function reduces to classical functions~\cite[Propositions III.6, III.7]{R}
\begin{align} \label{hgamma-cl-lim-1}
	& \gamma^{(2)}(z \omega_1) \underset{\omega_1 \to 0^+}{=} \frac{1}{\sqrt{2\pi}} \, \biggl( \frac{2\pi \omega_1}{\omega_2} \biggr)^{z - \frac{1}{2}} \, \Gamma(z), \\[10pt]  \label{hgamma-cl-lim-2}
	& \frac{\gamma^{(2)}(z + u \omega_1)}{\gamma^{(2)}(z)} \underset{\omega_1 \to 0^+}{=} \biggl( 2 \sin \frac{\pi z}{\omega_2} \biggr)^u.
\end{align}
So, assuming $\omega_1, \omega_2, g > 0$, rescaling the parameters $g = u \omega_1$, $\lambda = v \omega_1$ and using reflection formula~\eqref{hgamma-refl} we obtain
\begin{align}
	& \frac{\gamma^{(2)} ( \pm \lambda + g)}{\gamma^{(2)}(2g)} = \frac{\gamma^{(2)} ( (\pm v + u\bigr) \omega_1 )}{\gamma^{(2)}(2u \omega_1)} \underset{\omega_1 \to 0^+}{=} \frac{1}{2\pi} \sqrt{ \frac{\omega_2}{\omega_1} }\,  \frac{\Gamma(v + u) \Gamma(-v + u)}{\Gamma(2u)}, \\[10pt] \label{ratio-hgamma-cl-lim}
	& \gamma^{(2)}\biggl(\pm z + \frac{\omega_1 + \omega_2}{2} - g \biggr) = \frac{\gamma^{(2)}\bigl(z + \frac{\omega_2}{2} + \frac{1 - 2u}{2} \omega_1 \bigr) }{\gamma^{(2)}\bigl(z + \frac{\omega_2}{2} + \frac{1 + 2u}{2} \omega_1 \bigr) } \underset{\omega_1 \to 0^+}{=} \biggl( 2 \cos \frac{\pi z}{\omega_2} \biggr)^{-2u}.
\end{align}
It remains to justify that one can interchange the limit $\omega_1 \to 0^+$ and integration over $z$. By~\cite[Proposition 3.2]{BCDK} we have the bound
\begin{align} \label{ratio-hgamma-cl-bound}
	\Biggl|   \frac{\gamma^{(2)}\bigl(z + \frac{\omega_2}{2} + \frac{1 - 2u}{2} \omega_1 \bigr) }{\gamma^{(2)}\bigl(z + \frac{\omega_2}{2} + \frac{1 + 2u}{2} \omega_1 \bigr) }  \Biggr| \leq C \, e^{-\frac{\pi}{\omega_2} |z|}
\end{align}
uniform in $z \in \imath \mathbb{R}$ and $\omega_1 \in (0, \Omega]$ for sufficiently small $\Omega$. Hence, we can use dominated convergence theorem, and the integral identity~\eqref{hgamma-beta} reduces to
\begin{align} \label{Euler-beta-2}
	2\pi  \int_{\imath \mathbb{R}} e^{\frac{2\pi \imath}{\omega_2} v z} \, \biggl( 2 \cos \frac{\pi z}{\omega_2} \biggr)^{-2u} \, \frac{dz}{\imath \omega_2} =  \frac{\Gamma(v + u) \Gamma(-v + u)}{\Gamma(2u)}.
\end{align}
This transforms into the Euler integral~\eqref{Euler-beta} after the change of variable $t = 1/(1 + e^{2\pi \imath z/\omega_2})$.

\section{From hyperbolic to complex rational beta integral} \label{sec:beta-lim}

The reduction described at the end of the previous section is well known. Below we show that the integral~\eqref{hgamma-beta} can also be reduced to the complex rational beta integral~\eqref{cgamma-beta} in the limit $\omega_1/\omega_2 \to -1$ (so that $q, \tilde{q} \to 1$).

\subsection{Complex limits of the hyperbolic gamma function} \label{sec:hgamma-comp-lim}

The limit~\eqref{hgamma-cl-lim-1} has a complex analogue rigorously derived in~\cite[Section 2]{SS}. Namely,
\begin{align} \label{hgamma-comp-lim-1}
	\gamma^{(2)}(\imath \sqrt{\omega_1 \omega_2} [m + u \delta]) \underset{\delta \to 0^+}{=} e^{\frac{\pi \imath}{2} m^2} (4\pi \delta)^{\imath u - 1} \, \bm{\Gamma} \biggl( \frac{m + \imath u}{2} \bigg| \frac{-m + \imath u}{2} \biggr),
\end{align}
where it is assumed that
\begin{align}
	\sqrt{ \frac{\omega_1}{\omega_2} } = \imath + \delta + O(\delta^2), \qquad \delta > 0, \qquad m \in \mathbb{Z}, \qquad u \in \mathbb{C}.
\end{align}

One can  apply this limit to the hyperbolic beta integral~\eqref{hgamma-beta}. It can be done in
two different ways. In the first case one can transform the hyperbolic gamma functions in the integrand to
the complex gamma functions. In the second situation one takes parameter values such that
similar replacement takes place on the right-hand side of \eqref{hgamma-beta}.
The first limit was described in detail in our recent paper \cite[Section 4]{BSS} and we briefly
recall it here. Denote
\begin{align}
	g^* = \frac{\omega_1+\omega_2}{2}-g
\end{align}
and use the following parametrization
\begin{align} \label{g_bar}
	g^* = \imath  \sqrt{\omega_1 \omega_2}(r + h \delta), \qquad
	\lambda = \imath\sqrt{\omega_1 \omega_2} (N+\beta),
\qquad r, N \in \mathbb{Z}, \qquad h, \beta \in \mathbb{C}.
\end{align}
When $\delta\to0^+$ infinite number of poles start to pinch the integration contour
around the points $z \in \imath \mathbb{Z}$. Parametrization of $z$ in the form
\begin{align}
	z = \imath\sqrt{\omega_1 \omega_2}(m + u \delta), \qquad m \in \mathbb{Z}, \qquad u \in \mathbb{R}
\end{align}
removes pinching and converts the univariate integral over $z$ into bilateral infinite sum of integrals
over the variable $u$. If one takes $N = N(\delta)$ in such a way that $N\delta\to \alpha$ goes to some fixed number,
then the exponential factor in the integrand is preserved. The limit on the right-hand side is determined by the following identity established in \cite[Section~3]{BSS}
\begin{multline} \label{hgamma-comp-lim-2}
	e^{-\pi \imath Nm - \frac{\pi \imath}{2}m^2} \, \frac{ \gamma^{(2)} \bigl(\imath \sqrt{\omega_1 \omega_2} [N + \beta + m + u \delta + O(\delta^2)]\bigr) }{ \gamma^{(2)} \bigl(\imath \sqrt{\omega_1 \omega_2} [N + \beta]\bigr) } \\
	\underset{\substack{\delta \to 0^+\\[2pt] N\delta \to \alpha \;\,}}{=} \bigl(2\sh \pi(\alpha + \imath \beta)\bigr)^{\frac{m + \imath u}{2}} \, \bigl(2\sh \pi(\alpha - \imath \beta)\bigr)^{\frac{-m + \imath u}{2}} ,
\end{multline}
where it is assumed that\footnote{In~\cite{BSS} we required $|N| \to \infty$, but it is easy to see that the limit holds for fixed $N$ too (when $\alpha=0$). Besides, here we added $O(\delta^2)$ term, which also doesn't spoil the corresponding arguments.}
\begin{align}
	\sqrt{ \frac{\omega_1}{\omega_2}} = \imath + \delta + O(\delta^2), \qquad \delta > 0, \qquad N, m \in \mathbb{Z}, \qquad u, \beta \in \mathbb{C}.
\end{align}
This is a complex analogue of the limit~\eqref{hgamma-cl-lim-2}. As a result, after denoting $r=2\ell,\, h=2s$, one obtains the following identity
\begin{eqnarray} && \makebox[-2em]{}
\frac{1}{4\pi}
 \sum_{m \in \mathbb{Z}+\mu}\int_{\mathbb{R}}
 e^{-2\pi \imath(  \alpha u +  \beta m) } \; {\bf\Gamma}(s\pm u,\ell \pm m)\, du
\nonumber \\[3pt] &&
=  \bm{\Gamma}(2s, 2\ell)\bigl( 2 \ch \pi(\alpha + \imath\beta) \bigr)^{-\ell - \imath s}
\bigl( 2 \ch \pi(\alpha - \imath \beta) \bigr)^{\ell - \imath s},
\label{CBT}\end{eqnarray}
where $\mu \in \{0, 1/2\}$ is such that $\ell + m \in \mathbb{Z}$ and the sum of integrals converges for $\beta \in \mathbb{R}$, $\Im s>-1/2$. This is the complex binomial theorem
which, in turn, is the Fourier inverse of the complex beta integral \eqref{cgamma-beta} (cf. \eqref{cbeta-exp}, \eqref{Jd}).
For further details, see \cite{BSS}.

\begin{remark}
	By $\sh(w)$ and $\ch(w)$ we denote hyperbolic sine and cosine correspondingly. Besides, since the function 
	$$(\sh w)^{a} \, (\sh \bar{w})^{a'}, $$
	with $a - a' \in \mathbb{Z}$, is single-valued in $w \in \mathbb{C} \setminus \imath \pi \mathbb{Z}$, we can take any branch of the individual complex power $(\sh w)^{a} \equiv |\sh w|^a e^{\imath a \arg(\sh w)}$. To avoid ambiguity, here in what follows we choose the principal branch, that is $\arg(\sh w) \in (-\pi, \pi]$.
\end{remark} 

The first principal result of the present paper consists in a rigorous proof of the limiting formula in the second
possible case, when the univariate integral on the left-hand side of \eqref{hgamma-beta} is converted into
a two-dimensional integral over the complex plane. For this purpose choose a different parametrization of variables
\begin{align} \label{gl-param}
	g = \imath \sqrt{\omega_1 \omega_2}(m + u \delta), \qquad \lambda = \imath \sqrt{\omega_1 \omega_2}(k + v \delta), \qquad m,k\in \mathbb{Z}, \qquad u, v\in\mathbb{C},
\end{align}
and apply~\eqref{hgamma-comp-lim-1} to the right-hand side of relation~\eqref{hgamma-beta}, which yields
\begin{align}
	\frac{\gamma^{(2)} ( g \pm \lambda)}{\gamma^{(2)}(2g)} \underset{\delta \to 0^+}{=} \frac{e^{\pi \imath (m + k)}}{4\pi\delta} \, \frac{\bm{\Gamma} (a |a') \bm{\Gamma}(b|b')}{\bm{\Gamma}(a + b| a' + b')},
\end{align}
where
\begin{align} \label{ab-param}
	\begin{aligned}
		& a = \frac{m + k + \imath (u + v)}{2}, && \qquad a' = \frac{-m - k + \imath (u + v)}{2}, \\[6pt]
		& b =  \frac{m - k + \imath(u - v)}{2}, && \qquad b' =  \frac{- m + k + \imath(u - v)}{2}.
	\end{aligned}
\end{align}
Up to inessential factors the result coincides with the right-hand side of the complex beta integral~\eqref{cgamma-beta}.
Therefore, the same limit should be valid for the left-hand side, that is
\begin{align} \label{hbeta-lim-1}
	\lim_{\delta \to 0^+} \delta \, \int_{\imath \mathbb{R}} \Id(z) \, \frac{dz}{\imath \sqrt{\omega_1 \omega_2}} = \frac{e^{\pi \imath (m + k)}}{4\pi^2} \int_{\mathbb{C}} [t]^{a - 1} [1 - t]^{b - 1} d^2t ,
\end{align}
where
\begin{align} \label{Id}
	\Id(z) = e^{\frac{2\pi \imath}{\omega_1 \omega_2} \lambda z} \, \gamma^{(2)}\biggl(\pm z + \frac{\omega_1 + \omega_2}{2} - g \biggr).
\end{align}
A natural question arises whether this reduction can be performed directly at the level of integrals.

Formula~\eqref{hgamma-comp-lim-2} implies that the integrand \eqref{Id}
with $g, \lambda$ parametrized as in~\eqref{gl-param} and $\sqrt{\omega_1/\omega_2} = \imath + \delta + O(\delta^2)$ has the limit
\begin{multline} \label{Id-lim}
	\Id(\imath \sqrt{\omega_1 \omega_2}[N + \beta]) = e^{-2\pi \imath (\beta k + (N + \beta) v \delta )} \, \frac{\gamma^{(2)}\bigl(\imath\sqrt{\omega_1 \omega_2} [N + \beta - m - (u + \imath) \delta + O(\delta^2) ]\bigr)}{ \gamma^{(2)}\bigl(\imath\sqrt{\omega_1 \omega_2} [N + \beta + m + (u - \imath) \delta + O(\delta^2)] \bigr) }  \\[6pt]
	\underset{\substack{\delta \to 0^+\\[2pt] N\delta \to \alpha \; \,}}{=}
\frac{e^{-2\pi \imath (\alpha v + \beta k)}}{ (2\sh \pi (\alpha + \imath \beta))^{m + \imath u} \, (2\sh \pi (\alpha - \imath \beta))^{-m + \imath u}},
\end{multline}
where we assume $N \in \mathbb{Z}$, $\alpha, \beta \in \mathbb{R}$.

The complex beta integral can also be written in terms of the exponential functions. Changing the integration variable $t = 1/(1 - e^{2\pi(\alpha + \imath \beta)})$ and inserting parametrization~\eqref{ab-param} we have
\begin{align} \label{cbeta-exp}
	\frac{e^{\pi \imath (m + k)}}{4\pi^2} \int_{\mathbb{C}} [t]^{a - 1} [1 - t]^{b - 1} d^2t
	= \int_{\mathbb{R}} d\alpha \int_{-\frac{1}{2}}^{\frac{1}{2}} \Jd(\alpha, \beta) \, d\beta,
\end{align}
where
\begin{align}  \label{Jd}
 \Jd(\alpha, \beta) :=
\frac{e^{-2\pi \imath (\alpha v + \beta k)}}{ (2\sh \pi (\alpha + \imath \beta))^{m + \imath u} \, (2\sh \pi (\alpha - \imath \beta))^{-m + \imath u}}.
\end{align}
Notice that  $\Jd(\alpha, \beta)$ is $1$-periodic in $\beta$, $\Jd(\alpha, \beta+1) = \Jd(\alpha, \beta)$, i.e. one can use for integration over $\beta$
any interval $[a, a + 1]$. In the above notation the formula~\eqref{Id-lim} simply reads
\begin{align} \label{Id-lim-2}
	\Id(\imath \sqrt{\omega_1 \omega_2}[N + \beta]) \underset{\substack{\delta \to 0^+\\[2pt] N\delta \to \alpha \; \,}}{=} \Jd(\alpha, \beta),
\end{align}
while the reduction~\eqref{hbeta-lim-1} is equivalent to the following statement.

\begin{theorem} 
	Assume that the parameters of the function $\Id(z) \equiv \Id(z; \lambda, g, \omega_1, \omega_2)$ satisfy the conditions
	\begin{align} \label{o-param}
		\omega_1 = \imath + \delta, \qquad \omega_2 = - \imath + \delta, \qquad \frac{1}{\delta} \in \mathbb{Z}_{>0},
	\end{align}
	and
	\begin{align} \label{gl-param-2}
		\begin{aligned}
			& g = \imath \sqrt{\omega_1 \omega_2}(m + u \delta), && \hspace{1cm} m \in \mathbb{Z}, && \hspace{0.5cm} \Im u \in (-1, 0), \\[6pt]
			& \lambda = \imath \sqrt{\omega_1 \omega_2}(k + v \delta), && \hspace{1cm} k\in \mathbb{Z}, &&  \hspace{0.5cm}  v \in \mathbb{R}.
		\end{aligned}
	\end{align}
	Then the following limit holds
	\begin{align} \label{hbeta-lim-2}
		\lim_{\delta \to 0^+} \delta \int_{\imath \mathbb{R}} \Id(z) \, \frac{dz}{\imath \sqrt{\omega_1 \omega_2}} = \int_{\mathbb{R}} d\alpha \int_{-\frac{1}{2}}^{\frac{1}{2}} \Jd(\alpha, \beta) \, d\beta.
	\end{align}
\end{theorem}

The proof by direct transformation of one integral into another is given in Sections~\ref{sec:beta-outline}--\ref{sec:beta-lim-calc}. But before that let us make a few remarks about the assumptions on the parameters~\eqref{o-param},~\eqref{gl-param-2}. First, notice that $\sqrt{\omega_1/\omega_2} = \imath + \delta + O(\delta^2)$, as desired. The choice $\delta=1/n$, $n\in\mathbb{Z}_{>0}$, is done for a convenience of considerations below. 

Second, recall that the hyperbolic beta integral evaluation formula~\eqref{hgamma-beta} holds under the assumption $|\Re \lambda| < \Re g < \Re(\omega_1 + \omega_2)/2 = \delta$. To simplify matters, in what follows we take $\lambda \in \imath \mathbb{R}$, or equivalently $v \in \mathbb{R}$.\footnote{In practice, this is usually sufficient for analytic continuation of the hypergeometric type integrals in their parameters.} Hence, the condition $0 < \Re g < \delta$ forces the assumption $\Im u \in (-1, 0)$.

Finally, it is also possible to perform the reduction in the case $m,k \in \mathbb{Z} + 1/2$. This requires only slight modifications, and we omit it to ease the exposition.

\subsection{Outline} \label{sec:beta-outline}

Let us describe the main steps of the reduction~\eqref{hbeta-lim-2}. The integrand $\Id(z)$ \eqref{Id} has poles at the points
\begin{align}
	\begin{aligned}
		z_{\mathrm{poles}} & = \pm \biggl( m_1 \omega_1 + m_2 \omega_2 + \frac{\omega_1 + \omega_2}{2} - g \biggr) \\[6pt]
		& = \pm \bigl( \imath(m_1 - m_2) + (m_1 + m_2 + 1) \delta - \imath \sqrt{1 + \delta^2}(m + u \delta) \bigr),
	\end{aligned}
\end{align}
where $m_1, m_2 \in \mathbb{Z}_{\geq 0}$. In the limit $\delta \to 0^+$ these poles pinch integration contour $\imath \mathbb{R}$ at the points $\imath \mathbb{Z}$, as illustrated in Figure~\ref{fig:pinch}.

\begin{figure}[h]
	\centering
	\begin{tikzpicture}[line cap = round, thick]
		\def\xm{3.3} 
		\def\ym{3.2} 
		\def\z{0.4} 
		
		\def\m{-1}
		\def\ru{0.5}
		\def\iu{-0.9}
		\def\d{1}
		\def\r{1pt}
		\def\s{0.7}
		\pgfmathsetmacro{\gx}{-sqrt(1+\d*\d)*\iu*\d}  
		\pgfmathsetmacro{\gy}{sqrt(1+\d*\d)*(\m + \ru*\d)}  
		
		\pgfmathsetmacro{\mx}{\s*\gx}
		\pgfmathsetmacro{\my}{\s*\gy}
		\pgfmathsetmacro{\tx}{\s*(\gx + 3.3*\d)}
		\pgfmathsetmacro{\ty}{\s*(\gy + 3.3)}
		\pgfmathsetmacro{\bx}{\s*(\gx + 3.3*\d)}
		\pgfmathsetmacro{\by}{\s*(\gy - 3.3)}
		\draw[gray!30!white, dashed]  (\bx, \by) -- (\mx, \my) -- ( \tx, \ty ); 		
		\draw[gray!30!white, dashed]  (-\bx, -\by) -- (-\mx, -\my) -- ( -\tx, -\ty );		
		
		\draw[gray!40!white, ->] (-\xm, 0) -- (\xm, 0);
		\draw[gray!40!white, ->] (0, -\ym) -- (0, \ym);
		
		\draw[gray!50] (\xm - \z, \ym) -- (\xm - \z, \ym - \z) -- node[text = gray!90, yshift = 0.2cm] {\small $z$} (\xm, \ym - \z);
		
		\foreach \n in {0, 1, ..., 5}{
			\foreach \k in {0, 1, ..., 6}{
				\pgfmathparse{int(\n + \k)}
				\ifnum \pgfmathresult < 4
				\pgfmathparse{\s*(\gx + (\n + \k)*\d)}
				\pgfmathsetmacro{\x}{\pgfmathresult}
				\pgfmathparse{\s*(\gy + \n - \k)}
				\pgfmathsetmacro{\y}{\pgfmathresult}
				\filldraw (\x, \y) circle (\r);
				\filldraw (-\x, -\y) circle (\r);
				\fi
			}
		}
		
		\draw[red!80!black, very thick, ->-] (0, -\ym) -- (0, \ym - 0.2);

	\end{tikzpicture}
	\hspace{2.5cm}
	\begin{tikzpicture}[line cap = round, thick]
		\def\xm{2.3}
		\def\ym{3.2}
		\def\z{0.4}
		\def\m{-1}
		\def\ru{0.5}
		\def\iu{-0.9}
		\def\d{0.2}
		\def\r{1pt}
		\def\s{0.7}
		\pgfmathsetmacro{\gx}{-sqrt(1+\d*\d)*\iu*\d}
		\pgfmathsetmacro{\gy}{sqrt(1+\d*\d)*(\m + \ru*\d)}

		\pgfmathsetmacro{\mx}{\s*\gx}
		\pgfmathsetmacro{\my}{\s*\gy}
		\pgfmathsetmacro{\tx}{\s*(\gx + 5.3*\d)}
		\pgfmathsetmacro{\ty}{\s*(\gy + 5.3)}
		\pgfmathsetmacro{\bx}{\s*(\gx + 3.3*\d)}
		\pgfmathsetmacro{\by}{\s*(\gy - 3.3)}
		\pgfmathsetmacro{\rx}{\s*(\gx + 13.3*\d)}
		\draw[gray!30!white, dashed]  (\bx, \by) -- (\mx, \my) -- ( \tx, \ty ); 		
		\draw[gray!30!white, dashed]  (-\bx, -\by) -- (-\mx, -\my) -- ( -\tx, -\ty ); 		
		
		\draw[gray!40!white, ->] (-\xm, 0) -- (\xm, 0);
		\draw[gray!40!white, ->] (0, -\ym) -- (0, \ym);
		
		\draw[gray!50] (\xm - \z, \ym) -- (\xm - \z, \ym - \z) -- node[text = gray!90, yshift = 0.2cm] {\small $z$} (\xm, \ym - \z);
		
		\foreach \n in {0, 1, ..., 20}{
			\foreach \k in {0, 1, ..., 20}{
				\pgfmathparse{int(\n + \k) < 13}
				\ifthenelse{\pgfmathresult = 1}{
					\pgfmathparse{\s*(\gx + (\n + \k)*\d)}
					\pgfmathsetmacro{\x}{\pgfmathresult}
					\pgfmathparse{\s*(\gy + \n - \k)}
					\pgfmathsetmacro{\y}{\pgfmathresult}
					\pgfmathparse{(\y < \ym) && (\y > -\ym) ? 1 : 0}
					\ifthenelse{\pgfmathresult = 1}{
						\filldraw (\x, \y) circle (\r);
						\filldraw (-\x, -\y) circle (\r);}{};
				}{};
			}
		}
		
		\draw[red!80!black, very thick, ->-] (0, -\ym) -- (0, \ym - 0.2);
		
	\end{tikzpicture}

	\caption{Poles of $\Id(z)$ with $\delta = 1$ and $\delta = 0.2$} \label{fig:pinch}
\end{figure}
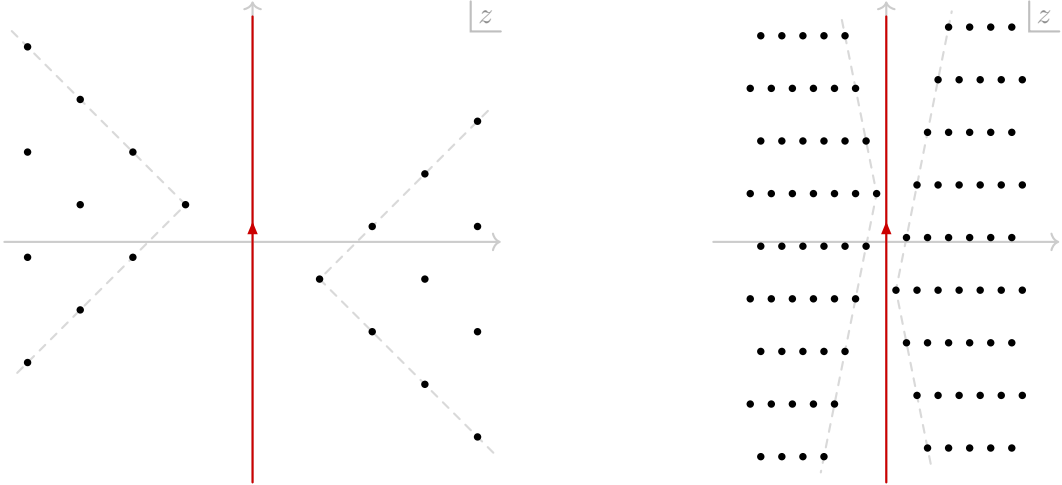

The first step is to convert the integral over $z$ into the sum of integrals around pinched points
\begin{align}
	\int_{\imath \mathbb{R}} \Id(z) \, \frac{dz}{\imath \sqrt{\omega_1 \omega_2}} = \sum_{N \in \mathbb{Z}} \int_{- \frac{1}{2}}^{\frac{1}{2}} \Id(\imath \sqrt{\omega_1 \omega_2} \, [N + \beta]) \, d\beta.
\end{align}
Notice that the argument of $\Id$-function on the right is the same, as in the limiting formula~\eqref{Id-lim-2}.

Second, we cut off the sum at $|N| = M/\delta$
\begin{align}
	\sum_{N \in \mathbb{Z}} \int_{- \frac{1}{2}}^{\frac{1}{2}} \Id(\imath \sqrt{\omega_1 \omega_2} \, [N + \beta]) \, d\beta = \lim_{M \to \infty} \sum_{|N| \leq M/\delta} \int_{- \frac{1}{2}}^{\frac{1}{2}} \Id(\imath \sqrt{\omega_1 \omega_2} \, [N + \beta]) \, d\beta
\end{align}
assuming $M \in \mathbb{Z}_{>0}$. Note that the wedge slopes in Figure~\ref{fig:pinch} are equal to $1/\delta$, so that the distance from the point $z = \imath \sqrt{\omega_1 \omega_2} M/ \delta$ to the poles is of order one, as opposed to (say) $z = 0$, for which the distance is of order $\delta$.

Since the points with $|N| \geq M/\delta$ experience no pinching, the tails of the above series (i.e. the sum over $|N| \geq M/\delta$) are expected to be bounded uniformly in $\delta$. In Section~\ref{sec:beta-lim-switch} we prove that this is indeed the case and, as a result, we can interchange two limits
\begin{multline}\label{change_lim}
	\lim_{\delta \to 0^+} \delta \lim_{M \to \infty} \sum_{|N| \leq M/\delta} \int_{- \frac{1}{2}}^{\frac{1}{2}} \Id(\imath \sqrt{\omega_1 \omega_2} \, [N + \beta]) \, d\beta \\[-6pt]
	= \lim_{M \to \infty} \, \lim_{\delta \to 0^+} \delta \sum_{|N| \leq M/\delta} \int_{- \frac{1}{2}}^{\frac{1}{2}} \Id(\imath \sqrt{\omega_1 \omega_2} \, [N + \beta]) \, d\beta.
\end{multline}
Mind that from the left we multiply the integral by $\delta$ before taking the limit $\delta \to 0^+$, as in the claimed formula~\eqref{hbeta-lim-2}. 

The final step is to calculate the limit
\begin{align}
	\lim_{\delta \to 0^+} \delta \sum_{|N| \leq M/\delta} \int_{- \frac{1}{2}}^{\frac{1}{2}} \Id(\imath \sqrt{\omega_1 \omega_2} \, [N + \beta]) \, d\beta = \int_{-M}^M d\alpha \int_{-\frac{1}{2}}^{\frac{1}{2}} \Jd(\alpha, \beta) \, d\beta,
\end{align}
which is done in Section~\ref{sec:beta-lim-calc}. The idea is that each term in the truncated sum suits the limiting formula~\eqref{Id-lim-2}. In other words, for fixed $N, \beta$ and small $\delta$ the integrand $\Id(\imath \sqrt{\omega_1 \omega_2} \, [N + \beta])$ is close to $\mathcal{J}(N\delta, \beta)$, so it is reasonable to expect that
\begin{align} \label{IJ-sums}
	\lim_{\delta \to 0^+} \delta \sum_{|N| \leq M/\delta} \int_{- \frac{1}{2}}^{\frac{1}{2}} \Id(\imath \sqrt{\omega_1 \omega_2} \, [N + \beta]) \, d\beta = \lim_{\delta \to 0^+} \delta \sum_{\substack{ |N| \leq M/\delta \\[2pt] N \neq 0}} \int_{- \frac{1}{2}}^{\frac{1}{2}} \Jd(N\delta, \beta) \, d\beta.
\end{align}
Note that we cannot add the term $N = 0$ on the right, since $\Jd(0, \beta)$ is not integrable for generic parameters. However, we show that the corresponding term on the left vanishes
\begin{align} \label{I0-lim}
	\lim_{\delta \to 0^+} \delta \int_{- \frac{1}{2}}^{\frac{1}{2}} \Id(\imath \sqrt{\omega_1 \omega_2} \, \beta) \, d\beta = 0.
\end{align}
The above formulas require uniform bounds on ratios of the hyperbolic gamma functions, which we derive in Appendices \ref{sec:ratios-q} and \ref{sec:ratios-gamma}.

At last, the sum on the right-hand side of \eqref{IJ-sums} is a Riemann sum of mesh $\delta$, so that
\begin{align}
	\lim_{\delta \to 0^+} \delta \sum_{\substack{ |N| \leq M/\delta \\[2pt] N \neq 0}} \int_{- \frac{1}{2}}^{\frac{1}{2}} \Jd(N\delta, \beta) \, d\beta = \int_{-M}^M d\alpha \int_{-\frac{1}{2}}^{\frac{1}{2}} \Jd(\alpha, \beta) \, d\beta.
\end{align}
Although for some values of $\Im u$ the integral over $\alpha$ is improper, the above approximation still holds in this case, see Appendix~\ref{sec:riem-sum-beta}. Taking $M \to \infty$ we finish the reduction~\eqref{hbeta-lim-2}.

\subsection{Interchanging limits} \label{sec:beta-lim-switch}

In this section we prove that the interchange of two limits \eqref{change_lim} is legitimate.
Notice that the separate limit $M\to\infty$ exists for any fixed $\delta > 0$ due to the convergence of the initial hyperbolic integral. Besides, in Section~\ref{sec:beta-lim-calc} we show that the separate limit $\delta \to 0^+$ also exists for any fixed $M > 0$ (by calculating it explicitly). Hence, to interchange two limits it is sufficient to show that the limit $M \to \infty$ is uniform for small enough $\delta \leq \dm$. Standardly, this is done by estimating the tails
\begin{align}\label{Tail}
	\Biggl| \, \delta  \sum_{| N| \geq M/\delta} \, \int_{- \frac{1}{2}}^{\frac{1}{2} }  \Id \bigl(\imath \sqrt{\omega_1 \omega_2} \, [N + \beta] \bigr) \, d\beta \, \Biggr| \leq C_M
\end{align}
by a sequence of constants $C_M$ independent of $\delta$ and such that
\begin{align}
	\lim_{M \to \infty} C_M = 0.
\end{align}
The main ingredient of the proof is Corollary~\ref{cor:g-ratio-bigN} derived in Appendix~\ref{sec:ratios-gamma}.

Inserting parametrizations~\eqref{o-param},~\eqref{gl-param-2} into definition~\eqref{Id} and using reflection formula we rewrite the integrand as
\begin{align} \label{Id-param}
	\Id(\imath \sqrt{\omega_1 \omega_2}[N + \beta]) = e^{-2\pi \imath (N \delta v + \beta k + \beta v \delta )} \, \frac{\gamma^{(2)}\bigl(\imath\sqrt{\omega_1 \omega_2} [N + \beta - m - (u + \imath) \delta + \epsilon(\delta) \delta^2 ]\bigr)}{ \gamma^{(2)}\bigl(\imath\sqrt{\omega_1 \omega_2} [N + \beta + m + (u - \imath) \delta + \epsilon(\delta)\delta^2] \bigr) } ,
\end{align}
where
\begin{align}
	\epsilon(\delta) = \frac{\imath}{\delta} \biggl(1 - \frac{1}{\sqrt{1 + \delta^2}} \biggr)= O(\delta).
\end{align}
By Corollary~\ref{cor:g-ratio-bigN} (with $\nu = 1$), the integrand satisfies the bound
\begin{align}\label{C}
	\bigl| \Id(\imath \sqrt{\omega_1 \omega_2}[N + \beta]) \bigr| \leq C \bigl| 2\sh \pi (N \delta + \imath \beta) \bigr|^{2\Im u}
\end{align}
uniformly for all $|N| \geq 1/\delta$, $0 < \delta \leq \dm$ and $|\beta| \leq 1$ with some constants $C, \dm$.

For large enough $|N|\delta \geq M$
\begin{align}
	| 2\sh \pi (N \delta + \imath \beta) | \geq 2\sh (\pi |N| \delta) \geq \frac{1}{2} e^{\pi |N| \delta}.
\end{align}
Furthermore, by assumption~\eqref{gl-param-2}, $\Im u \in (-1, 0)$. Combining the above estimates we arrive at
\begin{multline}
	\Biggl| \, \delta  \sum_{| N| \geq M/\delta} \, \int_{- \frac{1}{2}}^{\frac{1}{2} } \Id \bigl(\imath \sqrt{\omega_1 \omega_2} \, [N + \beta] \bigr) \, d\beta \, \Biggr| \leq C' \delta \sum_{| N| \geq M/\delta} e^{-2 \pi  |N| \delta |\Im u|} \\
	= 2C' \delta \, \frac{e^{-2 \pi M |\Im u|}}{1 - e^{-2\pi \delta |\Im u|}} \leq C'' \, e^{-2 \pi M |\Im u|},
\end{multline}
where on the last step we use the fact that $\delta / (1 - e^{-2\pi \delta |\Im u|})$ is continuous for $\delta \in [0, \dm]$. Thus, the tails tend to zero uniformly in $\delta$.

\subsection{Calculating the $\delta$-limit} \label{sec:beta-lim-calc}

In this section we calculate the limit
\begin{align} \label{IJ-calc-lim}
	\lim_{\delta \to 0^+} \delta \sum_{|N| \leq M/\delta} \int_{- \frac{1}{2}}^{\frac{1}{2}} \Id(\imath \sqrt{\omega_1 \omega_2} \, [N + \beta]) \, d\beta = \int_{-M}^M d\alpha \int_{-\frac{1}{2}}^{\frac{1}{2}} \Jd(\alpha, \beta) \, d\beta.
\end{align}
The main ingredients for this are the bounds derived in Proposition~\ref{prop:g-ratio} and Corollary~\ref{cor:g-ratio-smallN}, as well as Lemma~\ref{lem:F-prod}, which are proven in Appendix~\ref{sec:ratios-gamma}. Let us proceed with the following steps.
\medskip

\noindent \textit{Step 1.} Taking $N_0 = |m| + 3$ let us prove that
\begin{align} \label{sum-IJ}
	\lim_{\delta \to 0^+} \delta \sum_{N_0 + 1 \leq |N| \leq M/\delta} \int_{- \frac{1}{2}}^{\frac{1}{2}} \Bigl [ \Id(\imath \sqrt{\omega_1 \omega_2} \, [N + \beta]) - \Jd(N \delta, \beta) \Bigr] \, d\beta = 0.
\end{align}
If we factorise the expression~\eqref{Id-param}
\begin{multline}
	\Id(\imath \sqrt{\omega_1 \omega_2}[N + \beta]) = e^{-2\pi \imath (N \delta v + \beta k + \beta v \delta )} \, \frac{\gamma^{(2)}\bigl(\imath\sqrt{\omega_1 \omega_2} [N + \beta - m - (u + \imath) \delta + \epsilon(\delta) \delta^2 ]\bigr)}{ \gamma^{(2)}\bigl(\imath\sqrt{\omega_1 \omega_2} [N + \beta ] \bigr) } \\[8pt]
	\times  \frac{\gamma^{(2)}\bigl(\imath\sqrt{\omega_1 \omega_2} [N + \beta ]\bigr)}{ \gamma^{(2)}\bigl(\imath\sqrt{\omega_1 \omega_2} [N + \beta + m + (u - \imath) \delta + \epsilon(\delta)\delta^2] \bigr) } ,
\end{multline}
then for each ratio of gamma functions we can apply Proposition~\ref{prop:g-ratio}\textit{(ii)} (here we use the condition $|N| \geq N_0 + 1$). If in addition we invoke Lemma~\ref{lem:F-prod}, then for the function in square brackets we obtain the bound
\begin{align}
	\bigl| \Id(\imath \sqrt{\omega_1 \omega_2} \, [N + \beta]) - \Jd(N \delta, \beta) \bigr| \leq \frac{C_1 \delta}{1 - e^{-C_2 |N| \delta}} \,  \bigl| \Jd(N \delta, \beta) \bigr|
\end{align}
with some constants $C_1, C_2 > 0$ uniform in $N ,\delta, \beta$ inside considered intervals.

Furthermore, by definition~\eqref{Jd}, for $|\beta| \leq 1/2$ and $N \ne 0$ we have
\begin{multline} \label{Jd-abs}
	\bigl| \Jd(N\delta, \beta) \bigr| = \bigl| 2\sh \pi (N \delta + \imath \beta) \bigr|^{2\Im u} = \bigl( 4\sh^2 (\pi N \delta) + 4 \sin^2(\pi \beta) \bigr)^{\Im u}  \\[6pt]
	\leq C \, (\delta^2 + \beta^2)^{\Im u},
\end{multline}
since $\Im u \in (-1, 0)$. Using this and the fact that $s/(1 - e^{-C_2 s})$ is bounded for $s \in [0, M]$, we arrive at a simpler bound
\begin{align}
	\bigl| \Id(\imath \sqrt{\omega_1 \omega_2} \, [N + \beta]) - \Jd(N \delta, \beta) \bigr| \leq \frac{C}{|N|} \, (\delta^2 + \beta^2)^{\Im u}.
\end{align}
Consequently, the whole sum is estimated as
\begin{multline}
	\Biggl| \, \delta \sum_{N_0 + 1 \leq |N| \leq M/\delta} \int_{- \frac{1}{2}}^{\frac{1}{2}} \Bigl [ \Id(\imath \sqrt{\omega_1 \omega_2} \, [N + \beta]) - \Jd(N \delta, \beta) \Bigr] \, d\beta \, \Biggr| \\
	\leq C  \sum_{N_0 + 1 \leq |N| \leq M/\delta} \frac{1}{|N|} \; \; \delta \int_{- \frac{1}{2}}^{\frac{1}{2}} (\delta^2 + \beta^2)^{\Im u}  \, d\beta .
\end{multline}
The sum over $N$ is bounded by the harmonic numbers
\begin{align}
	\sum_{N_0 + 1 \leq |N| \leq M/\delta} \frac{1}{|N|} \leq 2\sum_{N = 2}^{M/\delta} \frac{1}{N} \leq 2\ln \frac{M}{\delta}.
\end{align}
Besides, it is easy to show that the integral over $\beta$ satisfies the bound
\begin{align} \label{dbeta-int}
	\delta \int_{- \frac{1}{2}}^{\frac{1}{2}} (\delta^2 + \beta^2)^{\Im u}  \, d\beta  \leq C_1 \delta^{2(1 + \Im u)} + C_2 \delta \ln \frac{1}{\delta},
\end{align}
see Lemma~\ref{lem:I-bound}. The latter estimates imply the claim~\eqref{sum-IJ}.
\medskip

\noindent \textit{Step 2.} As one can see, formula~\eqref{sum-IJ} excludes the terms $|N| \leq N_0$. Let us show that for these terms we have
\begin{align}
	\lim_{\delta \to 0^+} \delta \int_{- \frac{1}{2}}^{\frac{1}{2}} \Id(\imath \sqrt{\omega_1 \omega_2} \, [N + \beta]) \, d\beta = 0.
\end{align}
By Corollary~\ref{cor:g-ratio-smallN}, for any $N_0 \in \mathbb{Z}_{>0}$ we have the  following
bound for the integrand~\eqref{Id-param}
\begin{align}
	\bigl| \Id(\imath \sqrt{\omega_1 \omega_2}[N + \beta]) \bigr| \leq A \, \bigl| 2\sh \pi (N \delta + 2k \sign(N) \delta + \imath \beta) \bigr|^{2\Im u}
\end{align}
uniformly for all $|N| \leq N_0$, $0 < \delta \leq \dm$ and $|\beta| \leq 1$ with some constants $A, \dm >0$ and $k \in \mathbb{Z}_{>0}$ (here $\sign(0) = 1$). Moreover, since in the above integral $|\beta| \leq 1/2$ and $\Im u \in (-1, 0)$, we have
\begin{align} \label{sh-ineq}
	\bigl| \sh \pi (N \delta + 2k \sign(N) \delta + \imath \beta) \bigr|^{2\Im u} \leq B \, (\delta^2 + \beta^2)^{\Im u}
\end{align}
with some constant $B$. This leads to the bound for the integral
\begin{align}
	\Biggl| \, \delta \int_{- \frac{1}{2}}^{\frac{1}{2}} \Id(\imath \sqrt{\omega_1 \omega_2} \, [N + \beta]) \, d\beta \, \Biggr| \leq C \delta \int_{- \frac{1}{2}}^{\frac{1}{2}} (\delta^2 + \beta^2)^{\Im u} \, d\beta.
\end{align}
Due to the inequality~\eqref{dbeta-int} the right-hand side tends to zero as $\delta \to 0^+$.
\medskip

\noindent \textit{Step 3.} Similarly, for all $|N| \leq N_0$, $N \ne 0$, we have
\begin{align}
	\lim_{\delta \to 0^+} \delta \int_{- \frac{1}{2}}^{\frac{1}{2}} \Jd(N\delta, \beta) \, d\beta = 0.
\end{align}
To show this one should  use the bound~\eqref{Jd-abs}. The rest of the arguments are the same, as in the previous step.
\medskip

\noindent \textit{Step 4.} Combining the results of all previous steps we arrive at the relation
\begin{align}
	\lim_{\delta \to 0^+} \delta \sum_{|N| \leq M/\delta } \int_{- \frac{1}{2}}^{\frac{1}{2}} \Id(\imath \sqrt{\omega_1 \omega_2} \, [N + \beta]) \, d\beta = \lim_{\delta \to 0^+} \delta \sum_{\substack{ |N| \leq M/\delta \\[2pt] N \neq 0}} \int_{- \frac{1}{2}}^{\frac{1}{2}} \Jd(N\delta, \beta) \, d\beta.
\end{align}
The right-hand side represents two Riemann sums (for $N > 0$ and $N < 0$) approximating two integrals
\begin{align} \label{Gd-int}
	\int_{-M}^0 \Gd(\alpha) \, d\alpha + \int_{0}^M \Gd(\alpha) \, d\alpha = \int_{-M}^M \Gd(\alpha) \, d\alpha, \qquad\quad \Gd(\alpha) = \int_{- \frac{1}{2}}^{\frac{1}{2}} \Jd(\alpha, \beta) \, d\beta.
\end{align}
To complete the calculation~\eqref{IJ-calc-lim} we only need the statement
\begin{align} \label{G-approx}
	\lim_{\delta \to 0^+} \delta \sum_{\substack{ |N| \leq M/\delta \\[2pt] N \neq 0}} \Gd(N \delta) = \int_{-M}^M \Gd(\alpha) \, d\alpha.
\end{align}
For proper Riemann integrals this is a general fact, however, such approximation may fail for improper ones depending on the function $\Gd(\alpha)$. In our case
\begin{align}
	|\Gd(\alpha)| \leq \int_{- \frac{1}{2}}^{\frac{1}{2}} \bigl| 2\sh \pi (\alpha + \imath \beta) \bigr|^{2\Im u}  \, d\beta \leq \bigl| 2\sh (\pi \alpha) \bigr|^{2\Im u}.
\end{align}
Hence, the integral over $\beta$ is absolutely convergent for all $\alpha \in \mathbb{R}$ only if $\Im u \in (-1/2, 0)$.

So, depending on $\Im u$, we have either a proper or improper Riemann integral over $\alpha$. To conclude the proof, we check in Appendix~\ref{sec:riem-sum-beta} that the approximation~\eqref{G-approx} holds in the improper case as well.

\section{From hyperbolic to complex rational conical function} \label{sec:conic}

The procedure described above can be also applied to the integrals that cannot be evaluated explicitly. One of such examples is given by the hyperbolic conical function, which we consider in this section.

\subsection{Conical functions}

The classical conical function (up to some inessential factors) represents a special case of the Gauss hypergeometric function ${}_2 F_1(a, b, c; z)$ with some restriction\footnote{There are multiple choices of this restriction due to various transformation formulas for the hypergeometric function.} on the parameters $a, b, c$, see~\cite[\href{http://dlmf.nist.gov/14}{Chapter 14}]{DLMF}.

In~\cite{R2} Ruijsenaars introduced hyperbolic variant of the conical function and derived several integral representations for it. In particular, the representation~\cite[(3.51)]{R2} in our notation reads (use~\eqref{hgamma-refl} and~\eqref{G-Ruij})
\begin{multline}
	\mathcal{R}(\omega_1, \omega_2, 2g; \imath x, \imath \lambda) = \frac{\gamma^{(2)}(4g)}{\gamma^{(2)}(2g \pm \lambda)} \int_{\imath \mathbb{R}} e^{ \frac{2\pi \imath}{\omega_1 \omega_2} \lambda z } \, \gamma^{(2)} \biggl( z \pm \frac{ x}{2} + \frac{\omega_1 + \omega_2}{2} - g\biggr) \\
	\times  \gamma^{(2)} \biggl( - z \pm \frac{ x}{2} + \frac{\omega_1 + \omega_2}{2} - g \biggr) \, \frac{dz}{\imath \sqrt{\omega_1 \omega_2}}.
\end{multline}
Due to the asymptotics~\eqref{hgamma-asymp}, this integral is well defined and the integration contour separates series of poles of gamma functions for $x, \lambda \in \imath \mathbb{R}$ and $0 < \Re g < \Re (\omega_1 + \omega_2)/2$.
In what follows it will be more convenient for us to work with the function
\begin{multline} \label{Psi}
	\Psi_\lambda(x; g) \equiv \Psi_\lambda(x; g , \omega_1, \omega_2) = \int_{\imath \mathbb{R}} e^{ \frac{2\pi \imath}{\omega_1 \omega_2} \lambda z } \, \gamma^{(2)} \biggl( \pm  z + \frac{\omega_1 + \omega_2}{2} - g \biggr) \\
	\times \gamma^{(2)} \biggl( \pm  (z - x) + \frac{\omega_1 + \omega_2}{2} - g \biggr) \, \frac{dz}{\imath \sqrt{\omega_1 \omega_2}},
\end{multline}
which transforms into the previous one after the shift of integration variable $z \to z + x/2$ modulo 
the integral prefactors.

The above function reduces to the Gauss hypergeometric function in the limit $\omega_1 \to 0^+$. Namely, assuming $\omega_1, \omega_2, g> 0$ and using the limiting formula~\eqref{ratio-hgamma-cl-lim} we have
\begin{align} \label{con-cl-lim}
	\lim_{\omega_1 \to 0^+} \sqrt{ \frac{\omega_1}{\omega_2} } \, \Psi_{v \omega_1}(x; u \omega_1) = \int_{\imath \mathbb{R}} e^{\frac{2\pi \imath}{\omega_2} v z} \, \biggl( 2 \cos \frac{\pi z}{\omega_2} \biggr)^{-2u} \, \biggl( 2 \cos \frac{\pi (z - x)}{\omega_2} \biggr)^{-2u}  \, \frac{dz}{\imath \omega_2} ,
\end{align}
where the interchange of the limit and integration can be justified using the uniform bound~\eqref{ratio-hgamma-cl-bound}. The last integral turns into the Euler representation of hypergeometric function
\begin{align}
	{}_2 F_1(a, b, c; w) = \frac{\Gamma(c)}{\Gamma(b)\Gamma(c - b)} \int_0^1 t^{b - 1} (1 - t)^{c - b - 1} (1 - wt)^{-a} \, dt
\end{align}
after the change of variable $t = 1/(1 + e^{2\pi \imath z/\omega_2})$ (in this way we obtain the special case of hypergeometric function with restriction $c = 2a$).

Notice that the function~\eqref{Psi} is similar to the hyperbolic beta integral~\eqref{hgamma-beta}. The difference is that the integrand
\begin{align} \label{Idc}
	\Idc(z) = e^{ \frac{2\pi \imath}{\omega_1 \omega_2} \lambda z } \, \gamma^{(2)} \biggl( \pm  z + \frac{\omega_1 + \omega_2}{2} - g \biggr) \, \gamma^{(2)} \biggl( \pm  (z - x) + \frac{\omega_1 + \omega_2}{2} - g \biggr)
\end{align}
contains two more gamma functions and the additional parameter $x$, which appears in the position similar to
the integration variable $z$. This suggests that the hyperbolic conical function also has a complex rational reduction. To describe it we parametrise, as before, $\omega_1 = \bar{\omega}_2 = \imath + \delta$ and take
\begin{align} \label{glx-param}
	\begin{aligned}
		& g = \imath \sqrt{\omega_1 \omega_2}(m + u \delta), && \hspace{1cm} m \in \mathbb{Z}, && \hspace{0.5cm} \Im u \in (-1, 0), \\[6pt]
		& \lambda = \imath \sqrt{\omega_1 \omega_2}(k + v \delta), && \hspace{1cm} k\in \mathbb{Z}, &&  \hspace{0.5cm}  v \in \mathbb{R}, \\[2pt]
		& x = \imath \sqrt{\omega_1 \omega_2} (K(\delta) + \sigma), && \hspace{1cm} K(\delta) \in \mathbb{Z}, && \hspace{0.5cm} |\sigma| \leq \frac{1}{2},
	\end{aligned}
\end{align}
such that $K(\delta) \delta \to \rho \in \mathbb{R}$ as $\delta \to 0^+$. Then due to~\eqref{hgamma-comp-lim-2} the integrand has the limit
\begin{align}
	\Idc(\imath \sqrt{\omega_1 \omega_2}[N + \beta]) \underset{\substack{\delta \to 0^+\\[2pt] N\delta \to \alpha \; \,}}{=} \Jdc(\alpha, \beta),
\end{align}
where
\begin{multline} \label{Jdc}
	\Jdc(\alpha, \beta) = e^{-2\pi \imath (\alpha v + \beta k)} \, 	\bigl( 2\sh \pi (\alpha + \imath \beta) \; 2\sh \pi (\alpha - \rho + \imath \beta - \imath \sigma) \bigr)^{-m - \imath u} \\[6pt]
	\times \bigl(2\sh \pi (\alpha - \imath \beta)  \; 2\sh \pi (\alpha - \rho - \imath \beta + \imath \sigma) \bigr)^{m - \imath u} .
\end{multline}
Hence, analogously to the case of beta integral~\eqref{hbeta-lim-2}, it is reasonable to expect the following statement.

\begin{theorem}
	Assume that the parameters of the function $\Idc(z) \equiv \Idc(z; x, \lambda, g, \omega_1, \omega_2)$ satisfy the conditions~\eqref{glx-param} and
	\begin{align}
		\omega_1 = \bar{\omega}_2 = \imath + \delta, \qquad \frac{1}{\delta} \in \mathbb{Z}_{>0}, 
		\qquad K(\delta) = \biggl\lfloor \frac{\rho}{\delta} \biggr\rfloor, \qquad \rho \ne 0.
	\end{align}
	Then the following limit holds
	\begin{align} \label{con-comp-lim}
		\lim_{\delta \to 0^+} \delta \int_{\imath \mathbb{R}} \Idc(z) \, \frac{dz}{\imath \sqrt{\omega_1 \omega_2}} = \int_{\mathbb{R}} d\alpha \int_{-\frac{1}{2}}^{\frac{1}{2}} \Jdc(\alpha, \beta) \, d\beta.
	\end{align}
\end{theorem}

The proof is given in Sections~\ref{sec:con-outline}--\ref{sec:con-lim-calc}, but before that let us make a few remarks. First, the resulting integral on the right
\begin{multline} \label{Phi}
	\Phi_{v, k}(\rho, \sigma; u, m) = \int_{\mathbb{R}} d\alpha \int_{-\frac{1}{2}}^{\frac{1}{2}} e^{-2\pi \imath (\alpha v + \beta k)} \, 	\bigl( 2\sh \pi (\alpha + \imath \beta) \; 2\sh \pi (\alpha - \rho + \imath \beta - \imath \sigma) \bigr)^{-m - \imath u} \\
	\times \bigl(2\sh \pi (\alpha - \imath \beta)  \; 2\sh \pi (\alpha - \rho - \imath \beta + \imath \sigma) \bigr)^{m - \imath u} \, d\beta
\end{multline}
is essentially the conical function of complex rational type. After the change of variable $z = 1/(1 - e^{2\pi(\alpha + \imath \beta)})$ it transforms into the special case of the hypergeometric function associated with the complex field (modulo some integral prefactors)
\begin{align}
	{}_2F_1^{\mathbb{C}}(a|a', b|b', c|c'; w,\bar{w}) = \frac{\bm{\Gamma}(c|c')}{\pi \bm{\Gamma}(b|b') \bm{\Gamma}(c - b| c' - b')} \int_{\mathbb{C}} [z]^{b - 1} [1 - z]^{c - b - 1} [1 - w z]^{-a} \, d^2z,
\end{align}
as introduced in~\cite[Section 6.6]{GGR}. Here we assume that $a - a', b - b', c - c' \in \mathbb{Z}$, and the integral converges in some domain of parameters, see~\cite[Section 3]{MN}.

Secondly, we remark that formula~\eqref{con-comp-lim} in fact follows from the results of our previous paper~\cite{BSS2}. The hyperbolic conical function can be interpreted as the eigenfunction of hyperbolic Ruijsenaars system 
Hamiltonians in the case of two particles (see~\cite[Section~4]{HR}), and in~\cite{BSS2} we have considered complex rational limit of this system. Correspondingly, the claim~\eqref{con-comp-lim} follows from the combination of formulas~\cite[(3.34), (4.30), (5.8)]{BSS2} and~\eqref{hgamma-comp-lim-1}.

However, this reasoning is non-direct: to establish the limit~\eqref{con-comp-lim} in~\cite{BSS2} we pass to the ``dual'' Mellin--Barnes integral representations of the corresponding functions. So, it is desirable to give a more straightforward proof, which is done in the next sections.

Finally, let us comment about the assumptions on the parameters. The conditions on $\delta, v, u$ have the same reasons, as in the situation of beta integral. The case of arbitrary $K(\delta)$, such that $K(\delta) \delta \to \rho$, requires only inessential modifications, which complicate the exposition, so for brevity we fix $K(\delta)$. Besides, if $\rho = 0$, then nothing changes in the calculation of the limit in comparison with the beta integral, however, if additionally $\sigma = 0$ one must assume a stronger condition $\Im u \in (-1/2, 0)$ for the limiting integral to be absolutely convergent.

\subsection{Outline} \label{sec:con-outline}

Let us describe the main steps of reduction~\eqref{con-comp-lim}. In what follows we assume $\rho > 0$ without loss of generality due to the symmetry $\Psi_\lambda(x; g) = \Psi_{-\lambda}(-x; g)$.

As before, the first step is to transform the integral~\eqref{Psi} into the series
\begin{align}
	\int_{\imath \mathbb{R}} \Idc(z) \, \frac{dz}{\imath \sqrt{\omega_1 \omega_2}} = \sum_{N \in \mathbb{Z}} \int_{- \frac{1}{2}}^{\frac{1}{2}} \Idc(\imath \sqrt{\omega_1 \omega_2} \, [N + \beta]) \, d\beta.
\end{align}
In the limit $\delta \to 0^+$ poles of the integrand~\eqref{Idc} pinch the integration contour at the points $\imath \mathbb{Z}$ and $\imath (\mathbb{Z} + \sigma)$, however, there is no pinching for the points sufficiently far away (e.g., $|z| \geq \sqrt{\omega_1 \omega_2} (\rho + 1)/\delta$). Hence, we cut off the series at $|N| = M/\delta$ ($M \in \mathbb{Z}_{>0}$) and prove that one can interchange two limits
\begin{align} 
\label{con-lim-switch}
\Bigl(\lim_{\delta \to 0^+} \,  \lim_{M \to \infty}- \lim_{M \to \infty} \, \lim_{\delta \to 0^+} \Bigr) \; \delta
 \sum_{|N| \leq M/\delta} \int_{- \frac{1}{2}}^{\frac{1}{2}} \Idc(\imath \sqrt{\omega_1 \omega_2} 
 \, [N + \beta]) \, d\beta =0,
\end{align} 
see Section~\ref{sec:con-lim-switch}.

Next, in Section~\ref{sec:con-lim-calc} we calculate the $\delta$-limit
\begin{align}\label{limit_con}
	\lim_{\delta \to 0^+} \delta \sum_{|N| \leq M/\delta} \int_{- \frac{1}{2}}^{\frac{1}{2}} \Idc(\imath \sqrt{\omega_1 \omega_2} \, [N + \beta]) \, d\beta = \int_{-M}^M d\alpha \int_{-\frac{1}{2}}^{\frac{1}{2}} \Jdc(\alpha, \beta) \, d\beta
\end{align}
for large enough $M$. For this we first show that
\begin{multline}
	\lim_{\delta \to 0^+} \delta \sum_{|N| \leq M/\delta} \int_{- \frac{1}{2}}^{\frac{1}{2}} \Idc(\imath \sqrt{\omega_1 \omega_2} \, [N + \beta]) \, d\beta \\
	= \lim_{\delta \to 0^+} \delta \Biggl( \sum_{N = -M/\delta}^{-1} + \sum_{N = 1}^{\lfloor \rho/\delta \rfloor - 1} + \sum_{N = \lfloor \rho/\delta \rfloor + 2}^{M/\delta} \Biggr) \int_{- \frac{1}{2}}^{\frac{1}{2}} \Jdc(N\delta, \beta) \, d\beta.
\end{multline}
Notice the difference with the beta integral case~\eqref{IJ-sums}. If $\Im u \leq -1/2$ the function $\Jdc(\alpha, \beta)$ is not integrable in $\beta$ for $\alpha = 0$ and $\alpha = \rho$, see~\eqref{Jdc}. To avoid these singular points we exclude the terms $N = 0$, $N = \lfloor \rho/\delta\rfloor$ and $N = \lfloor \rho/\delta\rfloor + 1$ on the right-hand side (both $\lfloor \rho/\delta\rfloor$ and $\lfloor \rho/\delta\rfloor + 1$ can be arbitrarily close to $\rho/\delta$ as $\delta \to 0^+$).

The last step is to prove that the right-hand side approximates the Riemann integral in $\alpha$
\begin{align}
	\lim_{\delta \to 0^+} \delta \Biggl( \sum_{N = -M/\delta}^{-1} + \sum_{N = 1}^{\lfloor \rho/\delta \rfloor - 1} + \sum_{N = \lfloor \rho/\delta \rfloor + 2}^{M/\delta} \Biggr) \int_{- \frac{1}{2}}^{\frac{1}{2}} \Jdc(N\delta, \beta) \, d\beta = \int_{-M}^M d\alpha \int_{-\frac{1}{2}}^{\frac{1}{2}} \Jdc(\alpha, \beta) \, d\beta.
\end{align}
This statement is intuitively clear, while the detailed analysis is given in Appendix~\ref{sec:riem-sum-con}.
Taking $M \to \infty$ we complete the reduction~\eqref{con-comp-lim}.

\subsection{Interchanging limits} \label{sec:con-lim-switch}
In this part there is almost nothing new in comparison to the beta integral case. To interchange two limits~\eqref{con-lim-switch} it is sufficient to estimate the tails
\begin{align}
	\Biggl| \, \delta  \sum_{| N| \geq M/\delta} \, \int_{- \frac{1}{2}}^{\frac{1}{2} }  \Idc \bigl(\imath \sqrt{\omega_1 \omega_2} \, [N + \beta] \bigr) \, d\beta \, \Biggr| \leq C_M
\end{align}
by constants $C_M$ independent of $\delta$ and vanishing as $M \to \infty$.

With the parametrisation~\eqref{glx-param} the integrand~\eqref{Idc} has the following form (after application 
of the reflection relation~\eqref{hgamma-refl})
\begin{multline}
	\Idc(\imath \sqrt{\omega_1 \omega_2}[N + \beta]) = e^{-2\pi \imath (N \delta v + \beta k + \beta v \delta )} \, \frac{\gamma^{(2)}\bigl(\imath\sqrt{\omega_1 \omega_2} [N + \beta - m - (u + \imath) \delta + \epsilon(\delta) \delta^2 ]\bigr)}{ \gamma^{(2)}\bigl(\imath\sqrt{\omega_1 \omega_2} [N + \beta + m + (u - \imath) \delta + \epsilon(\delta)\delta^2] \bigr) } \\[6pt]
	\times \frac{\gamma^{(2)}\bigl(\imath\sqrt{\omega_1 \omega_2} [N - \lfloor \rho/\delta \rfloor + \beta - \sigma - m - (u + \imath) \delta + \epsilon(\delta) \delta^2 ]\bigr)}{ \gamma^{(2)}\bigl(\imath\sqrt{\omega_1 \omega_2} [N - \lfloor \rho/\delta \rfloor + \beta - \sigma + m + (u - \imath) \delta + \epsilon(\delta)\delta^2] \bigr) } ,
\end{multline}
where $\epsilon(\delta) = \imath (1 - 1/\sqrt{1 + \delta^2} ) / \delta$. For $|N| \geq M/\delta \geq (\rho + 1)/\delta$ we have
\begin{align}
	\biggl| N - \biggl\lfloor \frac{\rho}{\delta} \biggr\rfloor \biggr| \geq \frac{M}{\delta} - \biggl\lfloor \frac{\rho}{\delta} \biggr\rfloor  \geq \frac{1}{\delta}.
\end{align}
Hence, we can use Corollary~\ref{cor:g-ratio-bigN} (with $\nu = 1$) to obtain the bound
\begin{align}
	\bigl| \Idc(\imath \sqrt{\omega_1 \omega_2}[N + \beta]) \bigr| \leq C \, \bigl| 2\sh \pi (N \delta + \imath \beta) \; 2\sh \pi (N - \lfloor \rho/\delta \rfloor \delta + \imath \beta - \imath \sigma) \bigr|^{2\Im u}
\end{align}
uniform in $N, \delta, \beta$ inside the considered intervals. Furthermore, for sufficiently large $|N| \delta \geq M$
\begin{align}
	\bigl| 2\sh \pi (N \delta + \imath \beta) \; 2\sh \pi (N - \lfloor \rho/\delta \rfloor \delta + \imath \beta - \imath \sigma) \bigr| \geq \frac{1}{2} e^{\pi (|N| + |N - \lfloor \rho/\delta \rfloor|)\delta} \geq \frac{1}{2} e^{2\pi |N| \delta - \pi \rho}.
\end{align}
Since $\Im u \in (-1, 0)$ we therefore have
\begin{align}
	\Biggl| \, \delta  \sum_{| N| \geq M/\delta} \, \int_{- \frac{1}{2}}^{\frac{1}{2} }  \Idc \bigl(\imath \sqrt{\omega_1 \omega_2} \, [N + \beta] \bigr) \, d\beta \, \Biggr| \leq C\delta \sum_{|N| \geq M/\delta} e^{-4\pi |N| \delta | \Im u |} \leq C' e^{-4\pi | \Im u | M},
\end{align}
which implies the claim.

\subsection{Calculating the $\delta$-limit} \label{sec:con-lim-calc}
In this section we calculate the limit
\eqref{limit_con} 
assuming sufficiently large $M$. Overall, the calculation is quite similar to that for the beta integral. It does, however, require a slightly more intricate manipulation of the bounds from Appendix~\ref{sec:ratios-gamma}, so we provide the details.

As shown in Appendix~\ref{sec:riem-sum-con}, due to the singularities at $\alpha = 0$ and $\alpha = \rho$, in general case the right-hand side of  
\eqref{limit_con} 
is approximated by three (Riemann-type) sums
\begin{align} \label{J-int-riem}
	\lim_{\delta \to 0^+} \delta \Biggl( \sum_{N = -M/\delta}^{-1} + \sum_{N = 1}^{\lfloor \rho/\delta \rfloor - 1} + \sum_{N = \lfloor \rho/\delta \rfloor + 2}^{M/\delta} \Biggr) \int_{- \frac{1}{2}}^{\frac{1}{2}} \Jdc(N\delta, \beta) \, d\beta = \int_{-M}^M d\alpha \int_{-\frac{1}{2}}^{\frac{1}{2}} \Jdc(\alpha, \beta) \, d\beta.
\end{align}
It remains to analyse the difference between the left-hand sides 
of \eqref{limit_con} and~\eqref{J-int-riem}.
\medskip

\noindent \textit{Step 1.} Take $N_0 = |m| + 3$ and consider the terms away from $N = 0$ and $N = \lfloor \rho/\delta \rfloor$ by at least $N_0 + 1$. For them we prove that
\begin{multline} \label{IJ-sum-N0}
	\lim_{\delta \to 0^+} \delta \Biggl( \sum_{N = -M/\delta}^{-N_0 - 1} + \sum_{N = N_0 + 1}^{\lfloor \rho/\delta \rfloor - 1 - N_0} + \sum_{N = \lfloor \rho/\delta \rfloor + 1 + N_0}^{M/\delta} \Biggr) \\[6pt]
	\times \int_{- \frac{1}{2}}^{\frac{1}{2}} \bigl[ \Idc(\imath \sqrt{\omega_1 \omega_2} \, [N + \beta]) - \Jdc(N\delta, \beta) \bigr] \, d\beta = 0.
\end{multline}
Let us consider the last sum, the arguments for the remaining two are analogous.

First, we show that
\begin{align} \label{Idc-Jdc-ratio}
	\Biggl| \frac{ \Idc(\imath \sqrt{\omega_1 \omega_2} \, [N + \beta]) }{ \Jdc(N\delta, \beta) } - 1 \Biggr| \leq \frac{C_1 \delta}{1 - e^{-C_2 (N \delta - \rho)}}
\end{align}
with some constants $C_1, C_2 > 0$ uniform in $N, \delta, \beta$ in the considered intervals. Let us write $\Jdc(\alpha, \beta) \equiv \Jdc(\alpha, \beta; \rho)$ to emphasize the dependence on the parameter $\rho$, see~\eqref{Jdc}. Factorise the expression in question
\begin{align}
	\frac{ \Idc(\imath \sqrt{\omega_1 \omega_2} \, [N + \beta]) }{ \Jdc(N\delta, \beta; \rho) } = \frac{ \Idc(\imath \sqrt{\omega_1 \omega_2} \, [N + \beta]) }{ \Jdc(N\delta, \beta; \lfloor \rho/\delta \rfloor \delta) } \; \frac{\Jdc(N\delta, \beta; \lfloor \rho/\delta \rfloor \delta) }{ \Jdc(N\delta, \beta; \rho) }.
\end{align}
Since in the sum $N - \lfloor \rho/\delta \rfloor \geq N_0 + 1$ and $|\beta - \sigma| \leq 1$, we can use Proposition~\ref{prop:g-ratio}\textit{(ii)} together with Lemma~\ref{lem:F-prod} for the first ratio
\begin{align} \label{Idc-Jdc-ratio-2}
	\Biggl| \frac{ \Idc(\imath \sqrt{\omega_1 \omega_2} \, [N + \beta]) }{ \Jdc(N\delta, \beta; \lfloor \rho/\delta \rfloor \delta) } - 1 \Biggr| \leq \frac{C_1 \delta}{1 - e^{-C_2 (N - \lfloor \rho/\delta \rfloor)\delta}} \leq \frac{C_1 \delta}{1 - e^{-C_2 (N \delta - \rho)}}.
\end{align}
The second ratio explicitly reads
\begin{align} \label{Jdc-ratio}
	\begin{aligned}
		\frac{\Jdc(N\delta, \beta; \lfloor \rho/\delta \rfloor \delta) }{ \Jdc(N\delta, \beta; \rho) } & = \biggl( \frac{2\sh \pi (N \delta - \rho + \imath \beta - \imath \sigma)}{2\sh \pi (N \delta - \lfloor\rho/\delta\rfloor \delta + \imath \beta - \imath \sigma)} \biggr)^{m + \imath u} \\[6pt]
		& \times  \biggl( \frac{2\sh \pi (N \delta - \rho - \imath \beta + \imath \sigma)}{2\sh \pi (N \delta - \lfloor\rho/\delta\rfloor \delta - \imath \beta + \imath \sigma)} \biggr)^{-m + \imath u}.
	\end{aligned}
\end{align}
Rewrite the first ratio of sines
\begin{align}
	\frac{2\sh \pi (N \delta - \rho + \imath \beta - \imath \sigma)}{2\sh \pi (N \delta - \lfloor\rho/\delta\rfloor \delta + \imath \beta - \imath \sigma)} = e^{\pi (\lfloor \rho/\delta \rfloor \delta - \rho)} \, \frac{ 1 - e^{-2\pi (N \delta - \rho + \imath \beta - \imath \sigma)} }{ 1 - e^{-2\pi (N \delta - \lfloor\rho/\delta\rfloor \delta + \imath \beta - \imath \sigma)} },
\end{align}
and analogously for the second one. By Lemma~\ref{lem:ln} (with $\mu = \rho/\delta - \lfloor\rho/\delta\rfloor$),
\begin{multline}
	\Bigl| \ln \bigl( 1 - e^{-2\pi (N \delta - \rho + \imath \tau)} \bigr) - \ln \bigl( 1 - e^{-2\pi (N \delta - \lfloor\rho/\delta\rfloor \delta + \imath \tau)} \bigr) \Bigr| \\[6pt]
	\leq \frac{2\pi(\rho/\delta - \lfloor\rho/\delta\rfloor) \delta}{1 - e^{-2\pi (N \delta- \rho)}} \leq \frac{2\pi \delta}{1 - e^{-2\pi (N \delta - \rho)}}.
\end{multline}
Using this inequality together with Lemmas~\ref{lem:exp},~\ref{lem:F-prod} we conclude that the second ratio given by~\eqref{Jdc-ratio} admits the same bound, as the first one,
\begin{align} \label{Jdc-ratio-2}
	\Biggl| \frac{\Jdc(N\delta, \beta; \lfloor \rho/\delta \rfloor \delta) }{ \Jdc(N\delta, \beta; \rho) } - 1 \Biggr| \leq \frac{D_1 \delta}{1 - e^{-D_2 (N \delta - \rho)}}
\end{align}
with some $D_1, D_2>0$. Thus, from the bounds~\eqref{Idc-Jdc-ratio-2} and~\eqref{Jdc-ratio-2}, using again Lemma~\ref{lem:F-prod}, we obtain the estimate~\eqref{Idc-Jdc-ratio}.

As a result, we have the bound
\begin{multline}
	\Biggl| \delta \sum_{N = \lfloor \rho/\delta \rfloor + 1 + N_0}^{M/\delta} \int_{- \frac{1}{2}}^{\frac{1}{2}} \bigl[ \Idc(\imath \sqrt{\omega_1 \omega_2} \, [N + \beta]) - \Jdc(N\delta, \beta) \bigr] \, d\beta \Biggr| \\
	\leq   \sum_{N = \lfloor \rho/\delta \rfloor + 1 + N_0}^{M/\delta} \frac{C_1 \delta}{1 - e^{-C_2 (N\delta - \rho)}} \; \delta \int_{- \frac{1}{2}}^{\frac{1}{2}}  \bigl| \Jdc(N\delta, \beta)\bigr| d\beta.
\end{multline}
The function $\Jdc(N\delta, \beta)$ is $1$-periodic in $\beta$, so in the above integral we can change 
the integration domain to $[\sigma - 1/2, \sigma + 1/2]$. Furthermore, from the definition~\eqref{Jdc} we have
\begin{align} \label{Jdc-abs}
	\bigl| \Jdc(N\delta, \beta) \bigr| = \bigl| 2\sh\pi(N\delta + \imath \beta) \; 2\sh \pi (N\delta - \rho + \imath\beta - \imath \sigma) \bigr|^{2 \Im u}.
\end{align}
The inequalities
\begin{align}
	& \bigl| 2\sh\pi(N\delta + \imath \beta) \bigr| \geq A, \qquad \bigl| 2\sh \pi (N\delta - \rho + \imath\beta - \imath \sigma)  \bigr|^2 \geq B(\delta^2 + (\beta - \sigma)^2)
\end{align}
with some constants $A, B$ hold for all $N \geq \lfloor \rho/\delta \rfloor + 1 + N_0$ and $\beta \in [\sigma - 1/2, \sigma + 1/2]$. Hence, the sum and integral can be separated
\begin{multline}
	\Biggl| \delta \sum_{N = \lfloor \rho/\delta \rfloor + 1 + N_0}^{M/\delta} \int_{- \frac{1}{2}}^{\frac{1}{2}} \bigl[ \Idc(\imath \sqrt{\omega_1 \omega_2} \, [N + \beta]) - \Jdc(N\delta, \beta) \bigr] \, d\beta \Biggr| \\
	\leq  C \sum_{N = \lfloor \rho/\delta \rfloor + 1 + N_0}^{M/\delta} \frac{\delta}{1 - e^{-C_2 (N\delta - \rho)}} \; \delta \int_{- \frac{1}{2}}^{\frac{1}{2}}  (\delta^2 + \beta^2)^{\Im u} d\beta.
\end{multline}
The sum is bounded by the harmonic numbers
\begin{multline}
	\sum_{N = \lfloor \rho/\delta \rfloor + 1 + N_0}^{M/\delta} \frac{\delta}{1 - e^{-C_2 (N\delta - \rho)}} = \sum_{N = \lfloor \rho/\delta \rfloor + 1 + N_0}^{M/\delta} \frac{N\delta - \rho}{1 - e^{-C_2 (N\delta - \rho)}} \, \frac{1}{N - \rho/\delta} \\
	\leq C \sum_{N = \lfloor \rho/\delta \rfloor + 1 + N_0}^{M/\delta}  \frac{1}{N - \rho/\delta} \leq C \sum_{n = 2}^{M/\delta} \frac{1}{n} \leq C \ln \frac{M}{\delta}.
\end{multline}
Combining this with the bound on the $\beta$-integral given by Lemma~\ref{lem:I-bound} we come to the claim.
\medskip

\noindent \textit{Step 2.} Formula~\eqref{IJ-sum-N0} excludes the terms $|N| \leq N_0$ and $|N - \lfloor \rho/\delta \rfloor | \leq N_0$. Let us prove that for them we have
\begin{align}
	\lim_{\delta \to 0^+} \delta \int_{- \frac{1}{2}}^{\frac{1}{2}} \Idc(\imath \sqrt{\omega_1 \omega_2} \, [N + \beta]) \, d\beta = 0.
\end{align}
Consider the case $|N| \leq N_0$, arguments for the remaining values $|N - \lfloor \rho/\delta \rfloor | \leq N_0$ are the same. In this case the integrand
\begin{multline}
	\bigl| \Idc(\imath \sqrt{\omega_1 \omega_2} \, [N + \beta]) \bigr| = \Biggl|  \frac{\gamma^{(2)}\bigl(\imath\sqrt{\omega_1 \omega_2} [N + \beta - m - (u + \imath) \delta + O(\delta^2) ]\bigr)}{ \gamma^{(2)}\bigl(\imath\sqrt{\omega_1 \omega_2} [N + \beta + m + (u - \imath) \delta + O(\delta^2)] \bigr) } \\[6pt]
	\times \frac{\gamma^{(2)}\bigl(\imath\sqrt{\omega_1 \omega_2} [N - \lfloor \rho/\delta \rfloor + \beta - \sigma - m - (u + \imath) \delta + O(\delta^2) ]\bigr)}{ \gamma^{(2)}\bigl(\imath\sqrt{\omega_1 \omega_2} [N - \lfloor \rho/\delta \rfloor + \beta - \sigma + m + (u - \imath) \delta + O(\delta^2) ] \bigr) }  \Biggr|,
\end{multline}
can be bounded using Corollaries~\ref{cor:g-ratio-bigN} and~\ref{cor:g-ratio-smallN} for the second and the first ratios correspondingly. Namely, from them we have
\begin{multline}
	\bigl| \Idc(\imath \sqrt{\omega_1 \omega_2} \, [N + \beta]) \bigr| \leq C \bigl| 2\sh \pi(N\delta + 2k \sign(N) \delta + \imath \beta) \bigr|^{2\Im u} \\[6pt]
	\times \bigl| 2\sh \pi(N \delta - \lfloor \rho/\delta \rfloor \delta + \imath \beta - \imath \sigma ) \bigr|^{2\Im u}
\end{multline}
with some $C$ and $k \in \mathbb{Z}_{>0}$ for all $N, \delta, \beta$ (here $\sign(0) = 1$). 
Under taken restrictions on the parameters and for a sufficiently small $\delta$ the hyperbolic sines satisfy the bounds
\begin{align}
	& \bigl| 2\sh \pi(N\delta + 2k \sign(N) \delta + \imath \beta) \bigr|^2 \geq A (\delta^2 + \beta^2), \\[6pt]
	& \bigl| 2\sh \pi(N \delta - \lfloor \rho/\delta \rfloor \delta + \imath \beta - \imath \sigma ) \bigr| \geq \bigl| 2 \sh \pi (N - \lfloor \rho/\delta \rfloor) \delta \bigr| \geq B
\end{align}
with some constants $A, B$. Therefore,
\begin{align}
	\biggl|  \delta \int_{- \frac{1}{2}}^{\frac{1}{2}} \Idc(\imath \sqrt{\omega_1 \omega_2} \, [N + \beta]) \, d\beta  \biggr| \leq C \delta \int_{- \frac{1}{2}}^{\frac{1}{2}} (\delta^2 + \beta^2)^{\Im u}.
\end{align}
The right-hand side tends to zero as $\delta \to 0^+$ due to Lemma~\ref{lem:I-bound}.
\medskip

\noindent \textit{Step 3.} Analogously, for $|N| \leq N_0$ or $|N - \lfloor \rho/\delta \rfloor | \leq N_0$ with $N \ne 0, \lfloor \rho/\delta \rfloor, \lfloor \rho/\delta \rfloor + 1$ we have
\begin{align}
	\lim_{\delta \to 0^+} \delta \int_{- \frac{1}{2}}^{\frac{1}{2}} \Jdc(N\delta, \beta) \, d\beta = 0.
\end{align}
To prove it use~\eqref{Jdc-abs}, the rest of arguments are the same, as at the previous step. Let us remark that we exclude $N = 0, \lfloor \rho/\delta \rfloor, \lfloor \rho/\delta \rfloor + 1$ to have $|N| \geq 1$ and $|N - \rho/\delta| \geq 1$, that is to avoid singular points $\alpha = 0, \rho$ of the function $\Jdc(\alpha, \beta)$.
\medskip

\noindent \textit{Step 4.} Combining the results of all the previous steps we arrive at the equality
\begin{multline}
	\lim_{\delta \to 0^+} \delta \sum_{|N| \leq M/\delta} \int_{- \frac{1}{2}}^{\frac{1}{2}} \Idc(\imath \sqrt{\omega_1 \omega_2} \, [N + \beta]) \, d\beta \\
	= \lim_{\delta \to 0^+} \delta \Biggl( \sum_{N = -M/\delta}^{-1} + \sum_{N = 1}^{\lfloor \rho/\delta \rfloor - 1} + \sum_{N = \lfloor \rho/\delta \rfloor + 2}^{M/\delta} \Biggr) \int_{- \frac{1}{2}}^{\frac{1}{2}} \Jdc(N\delta, \beta) \, d\beta.
\end{multline}
Together with the approximation identity~\eqref{J-int-riem} this yields the statement~\eqref{limit_con}.

\section{From hyperbolic to complex rational hypergeometric function} \label{sec:hypergeometry}

Denote $\bm{c} = (c_0, c_1, c_2, c_3)^t$. The hyperbolic conical function considered in the previous section is a specialization of the Ruijsenaars' hyperbolic hypergeometric function $R(\omega_1, \omega_2, \bm{c}; x, \lambda)$ depending on eight parameters~\cite{R3}. The latter has several integral representations, in particular the one given in~\cite[Theorem 4.21]{BRS}
\begin{align} \label{hyp-hypergeom}
	\begin{aligned}
		R(\omega_1, \omega_2, \bm{c}; \imath x, \imath \lambda) & = \frac{\gamma^{(2)} \bigl( \lambda + \frac{\omega_1 + \omega_2}{2} - \hat{c}_0 \bigr)}{ 2 \, \prod_{j = 1}^3  \gamma^{(2)} \bigl( \lambda + \frac{\omega_1 + \omega_2}{2} + \hat{c}_j \bigr) \gamma^{(2)} (-c_0 - c_j) } \\[6pt]
		& \times \int_{\imath \mathbb{R}} \frac{ \prod_{j = 1}^6 \gamma^{(2)} \bigl( \pm z + \frac{\omega_1 + \omega_2}{2} - u_j \bigr) }{ \gamma^{(2)} (\pm 2z) } \; \frac{d z}{\imath \sqrt{\omega_1 \omega_2}},
	\end{aligned}
\end{align}
where 
\begin{align}
	\hat{\bm{c}} = \frac{1}{2}
	\begin{pmatrix}
		1 & 1 & 1 & 1 \\
		1 & 1 & -1 & -1 \\
		1 & -1 & 1 & -1\\
		1 & -1 & -1 & 1
	\end{pmatrix} \bm{c}.
\end{align}
and
\begin{align}
	& u_j = \frac{\omega_1 + \omega_2}{4} + c_{j - 1} - \frac{\hat{c}_0}{2} - \frac{\lambda}{2},  \qquad (j = 1, \dots, 4), \\[8pt]
	&  u_{5/6} = \frac{\omega_1 + \omega_2}{4} \pm  x + \frac{\hat{c}_0}{2} + \frac{\lambda}{2}.
\end{align}
The above integral converges and the integrand poles are separated by the 
integration contour under assumptions $x, \lambda \in \imath \mathbb{R}$ and
\begin{align} \label{c-cond}
	| \Re \hat{c}_0 | < \Re \frac{\omega_1 + \omega_2}{2}, \qquad 2 \Re c_{j} < \Re \biggl( \frac{\omega_1 + \omega_2}{2} + \hat{c}_0 \biggr), \qquad j = 0, \dots, 3.
 \end{align}
In~\cite[Section 8]{R3} Ruijsenaars showed on a formal level of rigour that the above function reduces to the Gauss hypergeometric function in the limit $\omega \to 0^+$
\begin{align}
	\lim_{\omega \to 0^+} R(\omega_1, \pi, \omega_1 \bm{c}; x, \omega_1 \lambda) = {}_2 F_1 \biggl( \hat{c}_0 + \imath \lambda, \, \hat{c}_0 - \imath \lambda, \, c_0 + c_2 + \frac{1}{2}; \, -\sh^2 x \biggr).
\end{align}
In what follows we show (omitting details) that in the limit $\omega_1/\omega_2 \to -1$ it can be also reduced to the hypergeometric function over the complex field.

Let us drop integral prefactors and consider the function
\begin{align}
	\int_{\imath \mathbb{R}} \mathcal{I}_h(z) \; \frac{d z}{\imath \sqrt{\omega_1 \omega_2}}
\end{align}
with
\begin{align}
	\begin{aligned}
		\mathcal{I}_h(z) & = \frac{1}{\gamma^{(2)}(\pm 2z)} \; \prod_{j = 0}^3 \gamma^{(2)} \biggl( \pm z + \frac{\omega_1 + \omega_2}{4} - c_j + \frac{\hat{c}_0 + \lambda}{2} \biggr) \\[6pt]
		& \times  \gamma^{(2)} \biggl( \pm (z - x) + \frac{\omega_1 + \omega_2}{4} - \frac{\hat{c}_0 + \lambda}{2} \biggr) \,  \gamma^{(2)} \biggl( \pm (z + x) + \frac{\omega_1 + \omega_2}{4} - \frac{\hat{c}_0 + \lambda}{2} \biggr).
	\end{aligned}
\end{align}
Notice that due to the difference equations and reflection formula for the hyperbolic gamma function we have
\begin{align}
	\frac{1}{\gamma^{(2)}(\pm 2z)} = -4 \sin \frac{2\pi z}{\omega_1} \, \sin \frac{2\pi z}{\omega_2}.
\end{align}
For the complex reduction, as before, parametrise $\omega_1 = \bar{\omega}_2 = \imath + \delta$ and take
\begin{align}
	& c_0 = c_1 = \imath \sqrt{\omega_1 \omega_2} (m_1 + u_1 \delta), && m_1 \in \mathbb{Z}, && u_1 \in \mathbb{C}\\[6pt]
	& c_{2/3} = \imath \sqrt{\omega_1 \omega_2} (m_2 \pm 1/2 + u_2 \delta), && m_2 \in \mathbb{Z}, && u_2 \in \mathbb{C}\\[6pt]
	& \lambda = \imath \sqrt{\omega_1 \omega_2} (k + v \delta), && k \in \mathbb{Z}, && v \in \mathbb{R} \\[0pt]
	& x = \imath \sqrt{\omega_1 \omega_2} (K(\delta) + \sigma), && K(\delta) \in \mathbb{Z}, && |\sigma| \leq \frac{1}{2},
\end{align}
such that $K(\delta) \delta \to \rho \in \mathbb{R}$ as $\delta \to 0^+$ and also
\begin{align}
	| \Im (u_1 + u_2) | < 1, \qquad | \Im (u_1 - u_2) | < 1, \qquad \frac{m_1 + m_2 + k}{2} \in \mathbb{Z}.
\end{align}
Note that the conditions on $\Im u_j$ ensure limitations~\eqref{c-cond}, while the last restriction is needed to perform the reduction.

Assuming the above parametrisation, the integrand has the pointwise limit
\begin{align}
	\mathcal{I}_h(\imath \sqrt{\omega_1 \omega_2} [N + \beta]) \underset{\substack{\delta \to 0^+ \\[2pt] N\delta \to \alpha \;}}{=} \mathcal{J}_h(\alpha, \beta),
\end{align}
where
\begin{align}
	\begin{aligned}
		\mathcal{J}_h(\alpha, \beta) & = \bigl[ 2\sh \pi(\alpha + \imath \beta) \bigr]^{ m_2 - m_1 + k + \imath (u_2 - u_1 + v) } \, \bigl[ 2\ch \pi(\alpha + \imath \beta) \bigr]^{ m_1 - m_2 + k + \imath (u_1 - u_2 + v) } \\[6pt]
		& \times \bigl[ 2\sh \pi (\alpha + \rho + \imath \beta + \imath \sigma)  \;  2\sh \pi (\alpha - \rho + \imath \beta - \imath \sigma) \bigr]^{- \frac{m_1 + m_2 + k + \imath (u_1 + u_2 + v - \imath)}{2}  }.
	\end{aligned}
\end{align}
Recall that we use the shorthand notation for complex powers $[w]^{\frac{m + \imath u}{2}} = w^{\frac{m + \imath u}{2}} \bar{w}^{\frac{-m + \imath u}{2}}$ with $m \in \mathbb{Z}$.
The integral over imaginary line transforms into the integral over the cylinder in the same way, as in the previous sections, that is
\begin{align} \label{hhyp-lim}
	\lim_{\delta \to 0^+} \delta \int_{\imath \mathbb{R}} \mathcal{I}_h(z) \, \frac{dz}{\imath \sqrt{\omega_1 \omega_2} } = \int_{\mathbb{R}} d\alpha \int_{- \frac{1}{2}}^{\frac{1}{2}} \mathcal{J}_{h}(\alpha, \beta) \, d\beta.
\end{align}
Note only that now we have three singular points corresponding to $\alpha = 0, \pm \rho$, so the right-hand side should be approximated by four (Riemann-type) sums.

Finally, after the change of variable $z = -1/\sh^2\pi(\alpha + \imath \beta)$ the obtained integral from the right~\eqref{hhyp-lim} transforms into the Euler representation of hypergeometric function over the complex field (modulo integral prefactors)
\begin{align} \label{comp-2F1}
	{}_2F_1^{\mathbb{C}}(a|a', b|b', c|c'; w,\bar{w}) = \frac{\bm{\Gamma}(c|c')}{\pi \bm{\Gamma}(b|b') \bm{\Gamma}(c - b| c' - b')} \int_{\mathbb{C}} [z]^{b - 1} [1 - z]^{c - b - 1} [1 - w z]^{-a} \, d^2z
\end{align}
with $w = - \sh^2 \pi (\rho + \imath \sigma)$ and parameters
\begin{align}
	& a = \frac{m_1 + m_2 + k + \imath(u_1 + u_2 + v - \imath)}{2}, && \; a' = \frac{-m_1 - m_2 - k + \imath(u_1 + u_2 + v - \imath)}{2}, \\[6pt]
	& b = \frac{m_1 + m_2 - k + \imath(u_1 + u_2 - v - \imath)}{2}, && \; b' = \frac{- m_1 - m_2 + k + \imath(u_1 + u_2 - v - \imath)}{2}, \\[6pt]
	& c = m_1 + \imath (u_1 - \imath), && \; c' = - m_1 + \imath (u_1 - \imath).
\end{align}
Thus, one of the representations of Ruijsenaars' hypergeometric function can be reduced to the complex Euler integral. Let us remark that using another limit~\eqref{hgamma-comp-lim-1} and the technique of~\cite{SS} it should be also possible to obtain the Barnes representation of ${}_2F_1^{\mathbb{C}}$ given in~\cite{N}.

\section{Further directions} \label{sec:concl}

In this section we indicate several directions in which the present work may be further developed. First, it is natural to generalize the described technique to multidimensional hypergeometric integrals. Among these are the eigenfunctions of the hyperbolic Ruijsenaars system constructed by Halln\"as and Ruijsenaars~\cite{HR}. Their rational reduction yields well-studied Heckman--Opdam $\mathfrak{gl}_n$ hypergeometric functions~\cite{BCDK}. Using the proposed technique one can also reduce Halln\"as--Ruijsenaars functions to the complex counterparts of Heckman--Opdam functions. The latter (modulo inessential factors) would be the joint eigenfunctions of the complex hyperbolic Calogero--Sutherland Hamiltonians
\begin{align}
	H = - \sum_{j = 1}^n \partial_{z_j}^2 + \sum_{1 \leq j < k \leq n} \frac{2 g (g - 1)}{ \sh^2(z_j - z_k) }, \qquad H' = - \sum_{j = 1}^n \partial_{\bar{z}_j}^2 + \sum_{1 \leq j < k \leq n} \frac{2 g' (g' - 1)}{ \sh^2(\bar{z}_j - \bar{z}_k) },
\end{align}
where $z$ is a coordinate on the complex cylinder and $g - g' \in \mathbb{Z}$. The corresponding two-particle model, for which the eigenfunctions coincide with conical functions, is treated in~\cite{BSS2}. 

Second, note that all hyperbolic-type integrals considered in this paper contain hyperbolic gamma functions combined in pairs $\gamma^{(2)}(\pm z + a)$ (with $z$ being the integration variable). However, there are important types of integrals that do not belong to this class; the simplest one is~\cite[Corollary C.2]{R2}
\begin{align} \label{hgamma-fourier}
	\int_{\imath \mathbb{R} + 0} e^{\frac{2\pi \imath}{\omega_1 \omega_2} \lambda z + \frac{\pi \imath}{2} (B_{2,2}(z) - B_{2,2}(0)) } \gamma^{(2)}(z) \frac{dz}{\imath \sqrt{ \omega_1 \omega_2}} = e^{- \frac{\pi \imath}{2} B_{2,2}(\lambda)} \gamma^{(2)}(\lambda),
\end{align}
where $B_{2,2}(z)$ is given by~\eqref{B22-def}. There are two ways to take complex limit of this integral: one can convert either the hyperbolic function on the left or the one on the right into the complex gamma function. In the first case, which is described in detail in our paper~\cite[Section~6]{BSS}, one obtains the two-dimensional Fourier transform of the complex gamma function. In the second case one can use pointwise limit~\cite[(6.15)]{BSS} to reduce the above integral to the integral over complex plane defining the complex gamma function~\eqref{cgamma-def}. However, a rigorous proof of the latter limit requires establishing the corresponding estimates.
More involved integrals similar to~\eqref{hgamma-fourier} appear in the study of $b$-Whittaker functions~\cite{SSh}. 

Finally, it is known that univariate hyperbolic hypergeometric functions appear in the representation theory of modular double of $U_q(\mathfrak{sl}_2)$~\cite{B, I, PT}, while complex rational hypergeometric functions come up in representation theory of the group $SL(2,\mathbb{C})$~\cite{DSS, MN, N}. The limiting relation between these functions suggests a connection between (principal series) representations of the two algebraic objects, as well as their higher-rank counterparts.

\smallskip

\textbf{Acknowledgements.}
The authors are indebted to S. E. Derkachov for useful discussions and to the anonymous referee for helpful comments that improved the exposition. The work of V. P. Spiridonov has been partially funded within the framework of the HSE University Basic Research Program.

\appendix

\addtocontents{toc}{\protect\setcounter{tocdepth}{-1}}
\section*{Appendices}
\addtocontents{toc}{\protect\setcounter{tocdepth}{1}}
The following appendices are rather technical, so we begin with a brief overview of their contents and explain how they fit into the main text.
In the main text, our goal is to calculate the limit
\begin{align*} 
	\lim_{\delta \to 0^+} \delta \sum_{N \in \mathbb{Z}} \int_{- \frac{1}{2}}^{\frac{1}{2}} \mathcal{F}_\delta(N + \beta) \, d\beta,
\end{align*}
where $\mathcal{F}_\delta$ is a product of exponentials and ratios of hyperbolic gamma functions depending on the parameter $\delta$.
The pointwise limit
\begin{align*}
	\mathcal{F}_\delta(N + \beta) \underset{\substack{\delta \to 0^+ \\[3pt] N \delta \to \alpha \; \,}}{=} \tilde{\mathcal{F}}(\alpha, \beta)
\end{align*}
is well defined. In the appendices, we prove estimates justifying the identity
\begin{align*}
	\lim_{\delta \to 0^+} \delta \sum_{N \in \mathbb{Z}} \int_{- \frac{1}{2}}^{\frac{1}{2}} \mathcal{F}_\delta(N + \beta) \, d\beta = \int_{\mathbb{R}} d\alpha \int_{- \frac{1}{2}}^{\frac{1}{2}} \tilde{\mathcal{F}}(\alpha, \beta) \, d\beta.
\end{align*}
More precisely, the main result of Appendices~\ref{sec:ratios-q} and~\ref{sec:ratios-gamma} is Proposition~\ref{prop:g-ratio}, which, together with elementary Lemma~\ref{lem:F-prod}, yields a bound for the ratio $\mathcal{F}_\delta(N + \beta)/\tilde{\mathcal{F}}(N\delta, \beta)$ uniform in $N, \beta, \delta$, except for finitely many values of $N$.

To prove Proposition~\ref{prop:g-ratio}, we begin with ratios of $q$-products, into which the hyperbolic gamma function factorizes by definition. Under several assumptions on the parameters, we establish the key estimate for such ratios in Lemma~\ref{lem:q-pr}. The bound for $q$-products then implies a bound for hyperbolic gamma functions (Lemma~\ref{lem:gamma-ratio}). Most of the assumptions are subsequently removed using the auxiliary Lemmas~\ref{lem:ln}--\ref{lem:F-prod}, leading to Proposition~\ref{prop:g-ratio}.

However, Proposition~\ref{prop:g-ratio} does not apply to the finitely many values of $N$ satisfying $| N - \lfloor \alpha_s/\delta \rfloor| \leq N_0$, where $\alpha_s$ is defined by the property that $\tilde{\mathcal{F}}(\alpha_s, \beta)$ is singular at some $\beta$. To handle these exceptional cases, we prove Corollary~\ref{cor:g-ratio-smallN}. This result follows from Proposition~\ref{prop:g-ratio} and difference equations for the hyperbolic gamma function.

Combining these estimates allows us to replace $\mathcal{F}_\delta(N + \beta)$ by $\tilde{\mathcal{F}}(N\delta, \beta)$ in the above sum of integrals (after introducing cutoff $|N| \leq \lfloor M/\delta\rfloor$). The result is a Riemann-type sum, but, due to singularities of $\tilde{\mathcal{F}}(\alpha, \beta)$, the limiting Riemann integral (with cutoff $|\alpha| \leq M$) may be improper. As a final step, Appendix~\ref{sec:riem} shows that the two Riemann sums considered in the main text indeed approximate the corresponding improper integrals.

\section{Ratios of the $q$-products} \label{sec:ratios-q}

For $\Im \omega_1/\omega_2 > 0$ the hyperbolic gamma function essentially equals to the ratio of two infinite $q$-products with parameters $q = e^{2\pi \imath \omega_1/ \omega_2}$ and $\tilde{q} = e^{-2\pi \imath \omega_2/\omega_1}$, see~\eqref{hgamma-int}. So, to obtain bounds for the ratios of hyperbolic gamma functions we first consider ratios of the $q$-products. In this and other Appendices we assume that $\delta$ is a continuous real variable. Besides, by $\ln z$ with $z \in \mathbb{C}$ we denote principal branch of a logarithm with branch cut along the negative real axis.

\begin{notation} \label{not:ozy}
	Parametrize
	\begin{align} \label{param}
		\begin{aligned}
			& \omega_1 = \bar{\omega}_2 = \imath + \delta, && \qquad\quad  \delta > 0, \\[6pt]
			& z = \imath \sqrt{\omega_1 \omega_2} (N + \beta), && \qquad\quad N \in \mathbb{Z}, &&  \quad \beta \in \mathbb{R}, \\[6pt]
			& y = \imath \sqrt{\omega_1 \omega_2} (m + u \delta + \epsilon(\delta) \delta^2), && \qquad\quad m \in \mathbb{Z}, &&  \quad u \in \mathbb{C},
		\end{aligned}
	\end{align}
	where $\epsilon(\delta) \in C^1[0, \dm]$ with some $\dm > 0$.
\end{notation}

Recall the well-known limit for the ratio of $q$-products~\cite[(1.3.19)]{GR}
\begin{align} \label{q-prod-lim}
	\frac{(w;q)_{\infty}}{(q^c w; q)_{\infty}} \underset{q\to 1^-}{=} (1-w)^{c}.
\end{align}
With the taken notation we therefore have
\begin{align}
	\frac{\Bigl( e^{ 2\pi \imath \frac{z}{\omega_2}} ; \, e^{ 2\pi \imath \frac{\omega_1}{\omega_2} } \Bigr)_\infty}{\Bigl( e^{ 2\pi \imath \frac{z + y}{\omega_2}} ; \, e^{ 2\pi \imath \frac{\omega_1}{\omega_2} } \Bigr)_\infty} \;\; \underset{\substack{\delta \to 0^+ \\[3pt] N \delta \to \alpha \; \,}}{=} \;\; \Bigl( 1 - e^{-2\pi (\alpha + \imath \beta)} \Bigr)^{ \frac{m + \imath u}{2} },
\end{align}
where $\alpha \in \mathbb{R}$. The following lemma gives an estimate for the error in this limiting formula under some
constraints on parameters.

\begin{lemma}\label{lem:q-pr}
	Using Notation~\ref{not:ozy}, let $m > \Im u$. Then there exist $C_1, C_2> 0$ and $\dm > 0$ such that
	\begin{multline}\label{lemma-ineq}
		\Biggl| \, \ln \Bigl( e^{ 2\pi \imath \frac{z}{\omega_2}} ; \, e^{ 2\pi \imath \frac{\omega_1}{\omega_2} } \Bigr)_\infty - \ln \Bigl( e^{ 2\pi \imath \frac{z + y}{\omega_2}} ; \, e^{ 2\pi \imath \frac{\omega_1}{\omega_2} } \Bigr)_\infty \\
		- \frac{m + \imath u}{2} \, \ln \Bigl( 1 - e^{-2\pi (N \delta + \imath \beta)} \Bigr) \Biggr| \leq \frac{C_1 \delta}{1 - e^{-C_2 N \delta}}
	\end{multline}
	for all $N, \delta, \beta$ satisfying
	\begin{align}
		N \geq 1, \qquad 0 < \delta \leq \dm, \qquad | \beta | \leq \frac{1}{2}.
	\end{align}
\end{lemma}

\begin{remark}
	The bound~\eqref{lemma-ineq} essentially captures two regimes
	\begin{align}
		\frac{C_1 \delta}{1 - e^{-C_2 N \delta}} \leq
		\left\{ \begin{aligned}
			& C \delta, && \quad\text{if $N\delta$ is large (e.g., $N\delta \geq 1$)}, \\[6pt]
			& \frac{\tilde{C}}{N}, && \quad \text{if $N\delta$ is small}.
		\end{aligned}
		\right.
	\end{align}
	The uniform bound for the first regime has previously appeared in~\cite[Lemma 2.17]{Ra}. Besides, uniformity of the limit~\eqref{q-prod-lim} is established in~\cite[Proposition A.2]{K} for $c \in \mathbb{R}$. However, these results seem insufficient for our purposes.
\end{remark}

\begin{proof} To prove the desired estimate we use Taylor expansions
	\begin{align}
		\ln (w; q)_{\infty} = -\sum_{k = 1}^\infty \frac{w^k}{k (1 - q^k)},  \qquad \ln(1 - w) = - \sum_{k =1}^\infty \frac{w^k}{k}
	\end{align}
	absolutely convergent if $|w| < 1$, $|q| < 1$. Let us proceed with the following steps.
	\bigskip
	
	\noindent \textit{Step 1.}
	First, check that we can use these expansions. From~\eqref{param} we have
	\begin{align}
		\Im \frac{\omega_1}{\omega_2} = \frac{2\delta}{1 + \delta^2}, \qquad \Re \sqrt{ \frac{\omega_1}{\omega_2} }= \frac{\delta}{\sqrt{1 + \delta^2}}, \qquad \Im \sqrt{ \frac{\omega_1}{\omega_2} }= \frac{1}{\sqrt{1 + \delta^2}}.
	\end{align}
	Hence, under the lemma assumptions
	\begin{align}
		& \bigl| e^{-2\pi (N \delta + \imath \beta)} \bigr| = e^{- 2\pi N \delta} < 1, && \Bigl| e^{ 2\pi \imath \frac{\omega_1}{\omega_2} } \Bigr| = e^{- \frac{4\pi \delta}{1 + \delta^2} } < 1, \\[6pt]
		& \Bigl| e^{ 2\pi \imath \frac{z}{\omega_2} } \Bigr| = e^{- \frac{2\pi \delta}{\sqrt{1 + \delta^2}} (N + \beta)  } < 1, && \Bigl| e^{ 2\pi \imath \frac{y}{\omega_2} } \Bigr| = e^{- \frac{2\pi \delta}{\sqrt{1 + \delta^2}} (m - \Im u + O(\delta)) } < 1,
	\end{align}
	where the last inequality holds for sufficiently small $\dm$.
	Consequently, the following series
	\begin{align}\label{series}
		\ln \frac{\Bigl( e^{ 2\pi \imath \frac{z}{\omega_2}} ; \, e^{ 2\pi \imath \frac{\omega_1}{\omega_2} } \Bigr)_\infty}
		{\Bigl( e^{ 2\pi \imath \frac{z + y}{\omega_2}} ; \, e^{ 2\pi \imath \frac{\omega_1}{\omega_2} } \Bigr)_\infty}
		- \frac{m + \imath u}{2} \, \ln \Bigl( 1 - e^{-2\pi (N \delta + \imath \beta)} \Bigr)
		= \sum_{k = 1}^\infty \frac{e^{2\pi \imath \frac{z}{\omega_2} k}}{k} F_k,
	\end{align}
	where
	\begin{align}
		F_k=  \frac{m + \imath u}{2} \, e^{-2\pi \bigl(N\delta + \imath \beta + \frac{\imath z}{\omega_2} \bigr) k } + \, \frac{ e^{2\pi \imath  \frac{y}{\omega_2} k } - 1 }{1 - e^{ 2\pi \imath \frac{\omega_1}{\omega_2} k }},
	\end{align}
	converges absolutely. To bound the whole series in the desired way we split it into two parts
	\begin{align}\label{sum2}
		\sum_{k = 1}^\infty = \sum_{k = 1}^{ \lfloor 1/\delta \rfloor } + \sum_{ k = \lfloor 1/\delta\rfloor + 1 }^\infty
	\end{align}
	and estimate them separately.
	\bigskip
	
	\noindent \textit{Step 2.}
	Consider the first sum in~\eqref{sum2}. For brevity, denote
	\begin{align}\label{c12}
		\sqrt{ \frac{\omega_1}{\omega_2} } = \imath + \delta + c_1(\delta) \, \delta^2, \qquad \frac{\omega_1}{\omega_2} = - 1 + 2\imath \delta + 2 \imath \, c_2(\delta) \, \delta^2,
	\end{align}
	where $c_i(\delta) = O(1)$. Let us insert these expressions together with the parametrizations of $z$ and $y$~\eqref{param} into the function in brackets in~\eqref{series}.
	Then we arrive at the formula
	\begin{align}\label{f}
		F_k = \frac{m + \imath u}{2} \, e^{2\pi \bigl( c_1  N\delta + (1 + c_1 \delta) \beta \bigr) k \delta} +  \frac{ e^{-2\pi \bigl(m + \imath u + u \delta + c_1 \delta (m + u \delta + \epsilon \delta^2) + \epsilon \delta (\imath + \delta) \bigr) k \delta} - 1 }{1 - e^{ -4\pi ( 1  + c_2 \delta) k\delta }}.
	\end{align}
	Clearly, it is a function of $\delta$ and combinations of parameters~$k \delta$, $N \delta$
	\begin{align}\label{Ff}
		F_k = F(\delta, k\delta, N\delta).
	\end{align}
	Due to the lemma assumptions and the summation range $k = 1, \dots, \lfloor1/\delta \rfloor$, the arguments  of the function~$F(\delta, x, s)$ belong to the set
	\begin{align}\label{D}
		D = \bigl\{ (\delta, x, s) \in \mathbb{R}^3 \;\; | \;\; 0 < \delta \leq \dm, \;\; \delta \leq x \leq 1, \;\; \delta \leq s \bigr\}.
	\end{align}
	The function itself
	\begin{align}\label{Fexpl}
		F(\delta, x, s) = \frac{m + \imath u}{2} \, e^{2\pi \bigl( c_1  s + (1 + c_1 \delta) \beta \bigr) x} +  \frac{ e^{-2\pi \bigl(m + \imath u + u \delta + c_1 \delta (m + u \delta + \epsilon \delta^2) + \epsilon \delta (\imath + \delta) \bigr) x} - 1 }{1 - e^{ -4\pi ( 1  + c_2 \delta) x }}
	\end{align}
	is continuous on the larger domain
	\begin{align}
		D_F = \bigl\{ (\delta, x, s) \in \mathbb{R}^3 \;\; | \;\; 0 \leq \delta \leq \dm, \;\; 0 \leq x \leq 1, \;\; 0 \leq s\bigr\}
	\end{align}
	for small enough $\dm$, such that $1 + \Re [c_2(\delta)] \, \delta > 0$ holds true. Notice also that
	\begin{align}\label{Flim}
		F(0, 0, s) = \lim_{\substack{ \delta \to 0^+ \\ x \to 0^+} } F(\delta, x ,s) = 0.
	\end{align}

	Using explicit formula~\eqref{Fexpl} one can check that partial derivatives of $F$ with respect to $\delta$ and $x$ have the form
	\begin{align} \label{Fpd}
		\begin{aligned}
			& \frac{\partial F}{\partial \delta} = (m + \imath u) \pi \, \frac{dc_1}{d\delta} \, x \, s \, e^{2\pi \bigl( c_1  s + (1 + c_1 \delta) \beta \bigr) x} + f_1(\delta, x), \\[6pt]
			& \frac{\partial F}{\partial x} = (m + \imath u) \pi \, c_1 \, s \, e^{2\pi \bigl( c_1  s + (1 + c_1 \delta) \beta \bigr) x} + f_2(\delta, x),
		\end{aligned}
	\end{align}
where we only explicitly indicate the terms linear in $s$, and the functions $f_1, f_2$ are continuous on $D_F$
(recall the assumption that $\epsilon(\delta) \in C^1[0, \dm]$). Let us argue that these partial derivatives
admit the bounds
	\begin{align}\label{Fpd-b}
		\biggl| \frac{\partial F}{\partial \delta} \biggr| \leq C_1 s + C_2, \qquad \biggl| \frac{\partial F}{\partial x} \biggr| \leq C_1 s + C_2
	\end{align}
	with some $C_1,C_2 > 0$ uniform in $(\delta, x, s) \in D_F$.
	First, by definition~\eqref{c12}
	\begin{align}
		\Re c_1 = \frac{1}{\delta^2} \biggl( \Re \sqrt{\frac{\omega_1}{\omega_2}} - \delta \biggr) = \frac{1}{\delta} \biggl( \frac{1}{\sqrt{1 + \delta^2}} - 1 \biggr) < 0.
	\end{align}
	Hence, the exponential functions in \eqref{Fpd} are uniformly bounded on $D_F$
	\begin{align}
		\Bigl| e^{2\pi \bigl( c_1  s + (1 + c_1 \delta) \beta \bigr) x} \Bigr| = e^{2\pi \bigl( \Re c_1  s + (1 + \Re c_1 \delta) \beta \bigr) x} \leq C.
	\end{align}
	Second, since the functions $c_1(\delta)$, $c_1'(\delta)$, $f_j(\delta, x)$ are continuous and $\delta, x$ vary over compact sets, these functions are also uniformly bounded. Thus, we arrive at the estimates~\eqref{Fpd-b}
which yield	the standard estimate for the remainder in the Taylor series expansion
	\begin{align}
\bigl| F(\delta, x, s) \bigr| = \bigl| F(\delta, x, s) - F(0, 0, s) \bigr| \leq \frac{C_1 s + C_2}{2} ( x + \delta ).
	\end{align}
	It follows that on the subdomain $D$~\eqref{D} we have
	\begin{align}
		\bigl| F(\delta, x, s) \bigr| \leq (C_1 s + C_2) \, x,
	\end{align}
	which in the original notation~\eqref{Ff} is equivalent to
	\begin{align}
		\bigl| F(\delta, k\delta, N \delta) \bigr| \leq  (C_1 N\delta + C_2) \, k \delta.
	\end{align}
Note that the constants $C_1, C_2$ can be made uniform in $\beta$, since $F$ and its partial derivatives are clearly continuous with respect to $\beta$ and $|\beta| \leq 1/2$.
	
	Therefore,
	\begin{align}
		\left| \sum_{k = 1}^{ \lfloor 1/\delta \rfloor } \frac{e^{2\pi \imath \frac{z}{\omega_2} k}}{k} \, F(\delta, k\delta, N \delta) \right| \leq \delta \, (C_1 N\delta + C_2) \sum_{k = 1}^{ \lfloor 1/\delta \rfloor } e^{- \frac{2\pi \delta}{\sqrt{1 + \delta^2}} (N + \beta) k }.
	\end{align}
Since $N \geq 1$ and $\beta \geq -1/2$ we have $N + \beta \geq N/2$, so that
	\begin{align} \label{Nb-ineq}
		e^{- \frac{2\pi \delta}{\sqrt{1 + \delta^2}} (N + \beta)k } \leq e^{- C_3 N\delta k }, \qquad C_3 = \frac{\pi}{\sqrt{1 + \dm}}.
	\end{align}
	It remains to estimate the geometric progression sum
	\begin{align}
		\sum_{k = 1}^{ \lfloor 1/\delta \rfloor } e^{- C_3 N \delta k} \leq \sum_{k = 1}^{ \infty } e^{- C_3 N \delta k} = \frac{e^{- C_3 N \delta}}{1 - e^{- C_3 N \delta} }.
	\end{align}
	Collecting all together we arrive at the inequality
	\begin{align} \label{sum-p1}
		\left| \sum_{k = 1}^{ \lfloor 1/\delta \rfloor } \frac{e^{2\pi \imath \frac{z}{\omega_2} k}}{k} \, F(\delta, k\delta, N \delta) \right| \leq \frac{\delta \, (C_1 N\delta + C_2) \,  e^{- C_3 N \delta}}{ 1 - e^{- C_3 N \delta} } \leq \frac{C \delta}{1 - e^{- C_3 N \delta} },
	\end{align}
	where on the last step we use the fact that $(C_1 s + C_2) e^{-C_3 s}$ is uniformly bounded for $s \geq 0$.
	\bigskip
	
	\noindent \textit{Step 3.}
	It is left to estimate the second sum in~\eqref{sum2}
	\begin{align}
		\sum_{k = \lfloor 1/\delta \rfloor + 1}^\infty \frac{e^{2\pi \imath \frac{z}{\omega_2} k}}{k} \Biggl( \frac{m + \imath u}{2} \, e^{-2\pi \bigl(N\delta + \imath \beta + \frac{\imath z}{\omega_2} \bigr) k } + \, \frac{ e^{2\pi \imath  \frac{y}{\omega_2} k } - 1 }{1 - e^{ 2\pi \imath \frac{\omega_1}{\omega_2} k }} \Biggr).
	\end{align}
	There are two terms in the brackets and their sums can be separately bounded by the geometric series.
Indeed, for the first term we have
	\begin{align} \label{sum-p2}
		\left| \sum_{k = \lfloor 1/\delta \rfloor + 1}^\infty \frac{e^{-2\pi (N \delta + \imath \beta) k}}{k}  \right| \leq \, \delta \sum_{k = \lfloor 1/\delta \rfloor + 1}^\infty e^{-2\pi N \delta k} \leq \frac{\delta}{1 - e^{-2\pi N \delta}},
	\end{align}
	where we used the fact that $k \geq \lfloor 1/\delta \rfloor + 1 \geq 1/\delta$. To estimate the second term
	\begin{align}\label{2term}
		\left| \sum_{k = \lfloor 1/\delta \rfloor + 1}^\infty \frac{ e^{2\pi \imath \frac{z}{\omega_2} k } }{k} \, \frac{ e^{2\pi \imath  \frac{y}{\omega_2} k } - 1 }{1 - e^{ 2\pi \imath \frac{\omega_1}{\omega_2} k }}  \right| \leq \, \delta \sum_{k = \lfloor 1/\delta \rfloor + 1}^\infty e^{- \frac{2\pi \delta}{\sqrt{1 + \delta^2}} (N + \beta) k } \, \left| \frac{ e^{2\pi \imath  \frac{y}{\omega_2} k } - 1 }{1 - e^{ 2\pi \imath \frac{\omega_1}{\omega_2} k }}  \right|
	\end{align}
we apply the triangle inequalities
	\begin{align}
		& \Bigl| e^{2\pi \imath  \frac{y}{\omega_2} k } - 1 \Bigr| \leq  1 + e^{- \frac{2\pi k \delta}{\sqrt{1 + \delta^2}} (m - \Im u + O(\delta)) }  \leq 2, \\[8pt]
		& \Bigl| 1 - e^{ 2\pi \imath \frac{\omega_1}{\omega_2} k } \Bigr| \geq  1 - e^{- \frac{4\pi k \delta}{1 + \delta^2} } \geq 1 - e^{ - \frac{4\pi}{1 + \dm^2 } }.
	\end{align}
	Collecting all together and using the inequality~\eqref{Nb-ineq}, we can bound the right-hand side~\eqref{2term}
by the geometric series
	\begin{align} \label{sum-p3}
		\left| \sum_{k = \lfloor 1/\delta \rfloor + 1}^\infty \frac{ e^{2\pi \imath \frac{z}{\omega_2} k } }{k} \, \frac{ e^{2\pi \imath  \frac{y}{\omega_2} k } - 1 }{1 - e^{ 2\pi \imath \frac{\omega_1}{\omega_2} k }}  \right| & \leq \, \tilde{C} \delta \sum_{k = \lfloor 1/\delta \rfloor + 1}^\infty e^{- C_3 N\delta k }  \leq \frac{\tilde{C} \delta}{1 - e^{-C_3 N \delta}},
	\end{align}
	where the constants $\tilde{C}$ and $C_3 > 0$ do not depend on $\delta, N, \beta$.	
	\smallskip
	
	\noindent \textit{Step 4.}
	To conclude the proof combine the estimates~\eqref{sum-p1},~\eqref{sum-p2} and~\eqref{sum-p3}.
\end{proof}

\begin{corollary}\label{cor:q-pr}
	Using Notation~\ref{not:ozy}, assume that $m > \Im u$. For any $\ell \in \mathbb{Z}_{>0}$ there exist $C_1, C_2> 0$ and $\dm > 0$ such that
	\begin{multline}\label{cor-ineq}
		\Biggl| \, \ln \Bigl( e^{ 2\pi \imath \frac{z + \ell \omega_1}{\omega_2}} ; \, e^{ 2\pi \imath \frac{\omega_1}{\omega_2} } \Bigr)_\infty - \ln \Bigl( e^{ 2\pi \imath \frac{z + y  + \ell \omega_1}{\omega_2}} ; \, e^{ 2\pi \imath \frac{\omega_1}{\omega_2} } \Bigr)_\infty \\[6pt]
		- \frac{m + \imath u}{2} \, \ln \Bigl( 1 - e^{-2\pi (N \delta + \imath \beta)} \Bigr) \Biggr| \leq \frac{C_1 \delta}{1 - e^{-C_2 N \delta}}
	\end{multline}
	for all $N, \delta, \beta$ satisfying the constraints
	\begin{align}
		N \geq 1, \qquad 0 < \delta \leq \dm, \qquad | \beta | \leq \frac{1}{2}.
	\end{align}
\end{corollary}

\begin{proof}
	Split the left-hand side of the inequality~\eqref{cor-ineq} into two parts
	\begin{align*}
		& \mathrm{LHS} \leq \Biggl| \,  \ln \Bigl( e^{ 2\pi \imath \frac{z + \ell \omega_1}{\omega_2}} ; \, e^{ 2\pi \imath \frac{\omega_1}{\omega_2} } \Bigr)_\infty - \ln \Bigl( e^{ 2\pi \imath \frac{z}{\omega_2}} ; \, e^{ 2\pi \imath \frac{\omega_1}{\omega_2} } \Bigr)_\infty + \frac{\ell + \imath(-\imath\ell)}{2} \, \ln \Bigl( 1 - e^{-2\pi (N \delta + \imath \beta)} \Bigr) \Biggr|  \\[8pt]
		& + \Biggl| \, \ln \Bigl( e^{ 2\pi \imath \frac{z}{\omega_2}} ; \, e^{ 2\pi \imath \frac{\omega_1}{\omega_2} } \Bigr)_\infty - \ln \Bigl( e^{ 2\pi \imath \frac{z + y  + \ell \omega_1}{\omega_2}} ; \, e^{ 2\pi \imath \frac{\omega_1}{\omega_2} } \Bigr)_\infty - \frac{m + \ell + \imath(u-\imath\ell)}{2} \, \ln \Bigl( 1 - e^{-2\pi (N \delta + \imath \beta)} \Bigr) \Biggr|.
	\end{align*}
	Notice that
	\begin{align}
		\ell \omega_1 = \imath \sqrt{\omega_1 \omega_2} \biggl( -\imath \ell \sqrt{ \frac{\omega_1}{\omega_2} } \biggr) = \imath \sqrt{\omega_1 \omega_2} \bigl( \ell - \imath \ell \delta + c(\delta) \, \delta^2 \bigr),
	\end{align}
	where $c(\delta) = O(1)$. Hence,
	\begin{align}
		& z + \ell \omega_1 = \imath \sqrt{\omega_1 \omega_2} \bigl( N + \beta + \ell - \imath \ell \delta + c(\delta) \, \delta^2 \bigr), \\[6pt]
		& z + y + \ell \omega_1 = \imath \sqrt{\omega_1 \omega_2} \bigl( N + \beta + m + \ell + (u - \imath \ell) \delta + [c(\delta) + \epsilon(\delta)] \, \delta^2 \bigr).
	\end{align}
	Since $m > \Im u$ and $\ell \in \mathbb{Z}_{>0}$, we also have
		\begin{align}
			\ell > \Im(- \imath\ell) = -\ell, \qquad m + \ell > \Im(u - \imath \ell) = \Im u - \ell.
		\end{align}
	Thus, both parts are estimated using Lemma~\ref{lem:q-pr}.
\end{proof}

\section{Ratios of the hyperbolic gamma functions} \label{sec:ratios-gamma}

By definition~\eqref{hgamma-def}, ratios of hyperbolic gamma functions can be written in terms of infinite $q$-products. We consider
\begin{align}\label{g-ratio}
	\frac{ \gamma^{(2)}(z + y) }{ \gamma^{(2)}(z) } = e^{ \frac{\pi \imath}{2} \bigl( B_{2,2}(z) - B_{2,2}(z + y) \bigr) } \, \frac{\Bigl( e^{ 2\pi \imath \frac{z}{\omega_2}} ; \, e^{ 2\pi \imath \frac{\omega_1}{\omega_2} } \Bigr)_\infty}{\Bigl( e^{ 2\pi \imath \frac{z + y}{\omega_2}} ; \, e^{ 2\pi \imath \frac{\omega_1}{\omega_2} } \Bigr)_\infty} \; \frac{\Bigl( e^{ 2\pi \imath \frac{z + y - \omega_2}{\omega_1}} ; \, e^{ -2\pi \imath \frac{\omega_2}{\omega_1} } \Bigr)_\infty}{\Bigl( e^{ 2\pi \imath \frac{z - \omega_2}{\omega_1}} ; \, e^{ -2\pi \imath \frac{\omega_2}{\omega_1} } \Bigr)_\infty}.
\end{align}
As discussed in Section~\ref{sec:hgamma-comp-lim}, under the parametrization~\eqref{param}
one has the limit \cite{BSS}
\begin{align}\label{ratio-lim}
	e^{-\pi \imath Nm - \frac{\pi \imath}{2} m^2} \; \frac{ \gamma^{(2)}(z + y) }{ \gamma^{(2)}(z) } \underset{\substack{\delta \to 0^+\\[3pt] N \delta \to \alpha \; \,}}{=} \bigl(2\sh \pi(\alpha + \imath \beta) \bigr)^{\frac{m + \imath u}{2}} \, \bigl(2\sh \pi(\alpha - \imath \beta) \bigr)^{\frac{-m + \imath u}{2}}.
\end{align}
The following lemma gives an estimate of the error term under some restrictions on the parameters. Define
\begin{multline} \label{lng-ratio}
	\Ln  \frac{\gamma^{(2)}(z + y) }{ \gamma^{(2)}(z) } : =  \frac{\pi \imath}{2} \bigl( B_{2,2}(z) - B_{2,2}(z + y) \bigr) + \ln \Bigl( e^{ 2\pi \imath \frac{z}{\omega_2}} ; \, e^{ 2\pi \imath \frac{\omega_1}{\omega_2} } \Bigr)_\infty  - \ln \Bigl( e^{ 2\pi \imath \frac{z + y}{\omega_2}} ; \, e^{ 2\pi \imath \frac{\omega_1}{\omega_2} } \Bigr)_\infty \\[8pt]
	+ \ln \Bigl( e^{ 2\pi \imath \frac{z + y - \omega_2}{\omega_1}} ; \, e^{ -2\pi \imath \frac{\omega_2}{\omega_1} } \Bigr)_\infty - \ln \Bigl( e^{ 2\pi \imath \frac{z - \omega_2}{\omega_1}} ; \, e^{ -2\pi \imath \frac{\omega_2}{\omega_1} } \Bigr)_\infty.
\end{multline}
As before, by $\ln z$ we denote principal branch of a logarithm, so that $\Im (\ln z) \in (-\pi, \pi]$. It is not necessarily true that $\ln \gamma^{(2)}(z + y) / \gamma^{(2)}(z)$ equals the right-hand side of the last formula, because the imaginary part of the right-hand side may not be in $(-\pi, \pi]$. However, the ambiguities disappear after exponentiation
\begin{align}
	\exp \biggl( \Ln  \frac{\gamma^{(2)}(z + y) }{ \gamma^{(2)}(z) } \biggr) = \exp \biggl( \ln  \frac{\gamma^{(2)}(z + y) }{ \gamma^{(2)}(z) } \biggr) ,
\end{align}
which is what we need, since at the end we are interested in ratios of gamma functions themselves, not their logarithms (see Proposition~\ref{prop:g-ratio}). 

\begin{lemma} \label{lem:gamma-ratio}
	Using Notation~\ref{not:ozy}, assume $m > |\Im u|$. Then there exist constants \mbox{$C_1, C_2, \dm > 0$} such that
	\begin{multline}\label{lem2-ineq}
		\Biggl| \, \Ln  \frac{\gamma^{(2)}(z + y) }{ \gamma^{(2)}(z) } - \pi \imath N m - \frac{\pi \imath}{2} m^2 - \pi \imath N \epsilon(\delta) \delta^2 \\[6pt]
		- \frac{m + \imath u}{2} \ln \Bigl( 2 \sh \pi (N \delta + \imath \beta) \Bigr) - \frac{-m + \imath u}{2}  \ln \Bigl( 2 \sh \pi (N \delta - \imath \beta) \Bigr)  \Biggr| \leq \frac{C_1 \delta}{1 - e^{-C_2 N \delta}}
	\end{multline}
	for all $N, \delta, \beta$ satisfying the constraints
	\begin{align}
		N \geq 1, \qquad 0 < \delta \leq \dm, \qquad | \beta | \leq \frac{1}{2}.
	\end{align}
\end{lemma}

\begin{proof}
	The proof relies on Lemma~\ref{lem:q-pr} and Corollary~\ref{cor:q-pr}. First, rewrite hyperbolic sines in~\eqref{lem2-ineq}
	\begin{multline}
		\frac{m + \imath u}{2} \ln \Bigl( 2 \sh \pi (N \delta + \imath \beta)  \Bigr) + \frac{-m + \imath u}{2}  \ln \Bigl( 2 \sh \pi (N \delta - \imath \beta) \Bigr) \\[6pt]
		= \frac{m + \imath u}{2} \ln \Bigl( 1 - e^{-2\pi (N \delta + \imath \beta)} \Bigr) + \frac{-m + \imath u}{2} \ln \Bigl( 1 - e^{-2\pi (N \delta - \imath \beta)} \Bigr) + \pi \imath(N \delta u + \beta m) .
	\end{multline}
Second, using this formula we split the whole left-hand side~\eqref{lem2-ineq} into three parts
\begin{align}
\text{LHS of~\eqref{lem2-ineq} } \leq | F_1 | + |F_2| + |F_3|,
\end{align}
where
	\begin{align}
		& F_1 = \frac{\pi \imath}{2} \bigl( B_{2,2}(z) - B_{2,2}(z + y) \bigr) - \pi \imath N m - \frac{\pi \imath}{2} m^2 - \pi \imath N \epsilon \delta^2 - \pi \imath(N \delta u + \beta m), \\[10pt]
		& F_2 = \ln \Bigl( e^{ 2\pi \imath \frac{z}{\omega_2}} ; \, e^{ 2\pi \imath \frac{\omega_1}{\omega_2} } \Bigr)_\infty - \ln \Bigl( e^{ 2\pi \imath \frac{z + y}{\omega_2}} ; \, e^{ 2\pi \imath \frac{\omega_1}{\omega_2} } \Bigr)_\infty - \frac{m + \imath u}{2} \ln \Bigl( 1 - e^{-2\pi (N \delta + \imath \beta)} \Bigr), \\[10pt]
		& F_3 = \ln \Bigl( e^{ 2\pi \imath \frac{z + y - \omega_2}{\omega_1}} ; \, e^{ -2\pi \imath \frac{\omega_2}{\omega_1} } \Bigr)_\infty - \ln \Bigl( e^{ 2\pi \imath \frac{z - \omega_2}{\omega_1}} ; \, e^{ -2\pi \imath \frac{\omega_2}{\omega_1} } \Bigr)_\infty - \frac{-m + \imath u}{2}  \ln \Bigl( 1 - e^{-2\pi (N \delta - \imath \beta)} \Bigr).
	\end{align}
	It is easy to estimate the first term $F_1$. By definition~\eqref{B22-def},
	\begin{align}
		B_{2,2}(z) - B_{2,2}(z + y) = -\frac{1}{\omega_1 \omega_2} (2z + y - \omega_1 - \omega_2) y.
	\end{align}
	Inserting parametrization of $z,y$~\eqref{param} into $F_1$, we arrive at the expression
	\begin{align}
		\frac{1}{\pi \imath} \, F_1 = \frac{(2\beta + 2m + u \delta + \epsilon \delta^2) (u  + \epsilon \delta)\delta}{2} + \frac{\imath \delta}{\sqrt{1 + \delta^2}} (m + u \delta + \epsilon \delta^2).
	\end{align}
	Since $\epsilon(\delta) = O(1)$, we have uniform in $N, \delta, \beta$ bound
	\begin{align} \label{F1}
		|F_1 | \leq C \delta
	\end{align}
	for sufficiently small $\delta$.
	
	Next, we estimate the second term $F_2$. Since $N> 0$ and $m > | \Im u | \geq \Im u$, we are able to apply Lemma~\ref{lem:q-pr}
	\begin{align} \label{F2}
		|F_2| \leq \frac{C_1 \delta}{1 - e^{-C_2 N \delta}}.
	\end{align}
	Finally, if we complex conjugate the third term
	\begin{align}
		\bar{F}_3 = \ln \Bigl( e^{ 2\pi \imath \frac{- \bar{z} - \bar{y} + \omega_1}{\omega_2}} ; \, e^{ 2\pi \imath \frac{\omega_1}{\omega_2} } \Bigr)_\infty - \ln \Bigl( e^{ 2\pi \imath \frac{- \bar{z} + \omega_1}{\omega_2}} ; \, e^{ 2\pi \imath \frac{\omega_1}{\omega_2} } \Bigr)_\infty + \frac{m + \imath u}{2} \ln \Bigl( 1 - e^{-2\pi (N \delta + \imath \beta)} \Bigr),
	\end{align}
	then we can again apply Corollary~\ref{cor:q-pr} (with $\ell = 1$) and obtain
	\begin{align} \label{F3}
		|F_3| \leq \frac{C_3 \delta}{1 - e^{-C_4 N \delta}}.
	\end{align}
	Here we notice
	\begin{align}\label{com_conj}
		- \bar{z} = \imath \sqrt{\omega_1 \omega_2} (N + \beta) = z, \qquad - \bar{y} = \imath \sqrt{\omega_1 \omega_2} \bigl( m + \bar{u} \delta + \bar{\epsilon}(\delta) \, \delta^2 \bigr)
	\end{align}
	and recall that by assumption $m > | \Im u | \geq \Im \bar{u}$. Then the statement of lemma~\eqref{lem2-ineq} follows from the bounds~\eqref{F1},~\eqref{F2} and~\eqref{F3}.
\end{proof}

To use the above lemma in practice we need to remove restrictions on the parameters $m, u$ and allow $N < 0$. This is achieved in Proposition~\ref{prop:g-ratio}, but before that we prove several auxiliary lemmas.

\begin{lemma} \label{lem:ln}
	The inequality
	\begin{align}
		\Bigl| \, \ln \Bigl( 1 - e^{ -2\pi (N \delta + \imath \beta + \mu \delta)} \Bigr) - \ln \Bigl( 1 - e^{ -2\pi (N \delta + \imath \beta)} \Bigr) \Bigr| \leq \frac{2\pi \mu \delta }{1 - e^{-2\pi N \delta}}
	\end{align}
	holds for
	\begin{align}
		N \geq 1, \qquad \delta > 0, \qquad \beta \in \mathbb{R}, \qquad \mu \geq 0.
	\end{align}
\end{lemma}

\begin{proof}	
	Under the lemma assumptions we have
	\begin{align}
		\bigl| e^{-2\pi (N \delta + \imath \beta)} \bigr| = e^{-2\pi N \delta} < 1, \qquad \bigl| e^{-2\pi (N \delta + \imath \beta + \mu  \delta )} \bigr| = e^{-2\pi (N + \mu)\delta} < 1.
	\end{align}
Therefore one can expand logarithmic functions in the Taylor series
	\begin{align}
		\ln \Bigl( 1 - e^{ -2\pi (N \delta + \imath \beta + \mu \delta)} \Bigr) - \ln \Bigl( 1 - e^{ -2\pi (N \delta + \imath \beta)} \Bigr) = \sum_{k = 1}^\infty \frac{e^{-2\pi (N \delta + \imath \beta) k}}{k} \, \bigl( 1 - e^{-2\pi \mu \delta k} \bigr)
	\end{align}
and estimate on the right-hand side
	\begin{align}
		0 \leq 1 - e^{-2\pi \mu \delta k} \leq 2\pi \mu \delta k.
	\end{align}
	Consequently, the whole sum is bounded by the geometric series
	\begin{align}
		\Biggl| \, \sum_{k = 1}^\infty \frac{e^{-2\pi (N \delta + \imath \beta) k}}{k} \, \bigl( 1 - e^{-2\pi \mu \delta k} \bigr) \Biggr| \leq 2\pi \mu \delta \, \sum_{k = 1}^\infty e^{- 2\pi N\delta k} \leq \frac{2\pi \mu \delta}{1 - e^{-2\pi N \delta}},
	\end{align}
	which concludes the lemma proof.
\end{proof}

\begin{lemma} \label{lem:exp}
	Suppose that some function $f(N, \delta)$ satisfies the bound
	\begin{align}
		|f| \leq \frac{C \delta}{1 - e^{-D N \delta}}
	\end{align}
	uniformly for all integers $N \geq 1$ and $0 < \delta \leq \dm$ with some constants
$C \in (0, \infty)$, $D \in (0, \infty]$ and $\dm>0$. Then the functions $e^{\pm f} - 1$
have similar uniform bounds
	\begin{align}
		\bigl| e^{\pm f} - 1 \bigr| \leq \frac{C_\pm \delta}{1 - e^{-D N \delta}}
	\end{align}
	with some other constants $C_\pm(C,D)> 0$.
\end{lemma}

\begin{proof}
	Since $N \geq 1$ and $\delta/(1 - e^{-D \delta})$ is continuous for $0 \leq \delta \leq \dm$, the function $f$ is uniformly bounded: $|f| \leq \tilde{C}(C,D)$. Besides, for any $R > 0$ and $|w| \leq R$ we have $|e^w - 1| \leq C(R) |w|$. This implies the lemma statement.
\end{proof}
\vspace{-0.1cm}

\begin{lemma} \label{lem:F-prod}
	Suppose functions $F_1(N, \delta)$, $F_2(N, \delta)$ satisfy the bounds
	\begin{align}
		|F_j - 1| \leq \frac{C_j \delta}{1 - e^{-D_j N \delta}}
	\end{align}
	uniformly for all $N \geq 1$, $0 < \delta \leq \dm$ with some $C_j \in (0, \infty)$, $D_j \in (0, \infty]$ and $\dm > 0$. Then their product is uniformly bounded in a similar way
	\begin{align}
		|F_1 F_2 - 1| \leq \frac{C \delta}{1 - e^{-D N \delta}}
	\end{align}
	with some constants $C,D > 0$ depending on $C_j, D_j$.
\end{lemma}
\begin{proof}
	The claim follows from the triangle inequalities
	\begin{align}
		| F_1 F_2 - 1 |  \leq | F_1 -1 | \, |F_2 - 1| + |F_1 - 1| + |F_2 - 1|
	\end{align}
	and the assumed bounds on $F_j$.
\end{proof}

Using Notation~\ref{not:ozy}, define
\begin{align}\label{F-factor}
	\begin{aligned}
		f(N, \beta; m, u, \epsilon) &= e^{- \pi \imath Nm - \frac{\pi \imath}{2} m^2} \; \frac{ \gamma^{(2)}(z + y) }{ \gamma^{(2)}(z) } \\[6pt]
		& \times \bigl(2\sh \pi(N \delta + \imath \beta) \bigr)^{-\frac{m + \imath u}{2}} \, \bigl(2\sh \pi(N \delta - \imath \beta) \bigr)^{\frac{m - \imath u}{2}}.
	\end{aligned}
\end{align}
Then the limit~\eqref{ratio-lim} reads
\begin{align} \label{flim}
	f(N, \beta; m, u, \epsilon) \underset{\substack{\delta \to 0^+\\[3pt] N \delta \to \alpha \; \,}}{=} 1.
\end{align}
Since $\epsilon(\delta) = O(1)$, the function
\begin{align}\label{F-f}
	F(N, \beta; m, u, \epsilon) = f(N, \beta; m, u, \epsilon) \, e^{ -\pi\imath |N| \epsilon \delta^2 }
\end{align}
also tends to one
\begin{align} \label{F-lim}
	F(N, \beta; m, u, \epsilon) \underset{\substack{\delta \to 0^+ \\[3pt] N \delta \to \alpha \; \,}}{=} 1.
\end{align}
From the above lemmas we have the following error bounds for both functions. 

\begin{proposition} \label{prop:g-ratio}
	Set $N_0 =  |m| + \Bigl\lfloor | \Im u | \Bigr\rfloor + 2$ and take the functions $F$, $f$ given by \eqref{F-f},~\eqref{F-factor}. Then,
	\begin{itemize}
		\item[\it (i)] there exist constants $C_1, C_2, \dm > 0$ such that
		\begin{align}\label{cor2-ineq}
			\bigl| F^{\pm 1} - 1 \bigr| \leq \frac{C_1 \delta}{1 - e^{-C_2 | N | \delta}}
		\end{align}
		for all $N, \delta, \beta$ satisfying
		\begin{align}
			| N| \geq N_0 + 1, \qquad 0 < \delta \leq \dm, \qquad | \beta | \leq 1;
		\end{align}
		
		\item[\it (ii)] for any fixed $M > 0$ there exist constants $C_1, C_2, \dm > 0$ such that
		\begin{align} \label{cor2-ineq2}
			\bigl| f^{\pm 1} - 1 \bigr| \leq \frac{C_1 \delta}{1 - e^{-C_2 | N | \delta}}
		\end{align}
		for all $N, \delta, \beta$ satisfying
		\begin{align}
			N_0 + 1 \leq |N| \leq \left\lfloor \frac{M}{\delta} \right\rfloor, \qquad 0 < \delta \leq \dm, \qquad | \beta | \leq 1.
		\end{align}
	\end{itemize}	
\end{proposition}

\begin{remark}
	Note that in the first part we do not assume any upper bound on $N$ (although in the limit~\eqref{F-lim} we take $N\delta \to \alpha$), so that the estimate holds in particular for $|N| \geq M/\delta$ with any fixed $M$ (and sufficiently small $\dm$). This is essential for bounding tails of sums containing ratios of gamma functions, see Sections~\ref{sec:beta-lim-switch},~\ref{sec:con-lim-switch}.
\end{remark}

\begin{proof}
	We begin by proving the estimate for the function $F$~\eqref{cor2-ineq}. 
	It is sufficient to consider only the case $N > 0$. Indeed, we can consider the
	bounds for the complex-conjugate function~$\bar F$. Since
	$\overline{\gamma^{(2)}(z)}=\gamma^{(2)}(\bar z)$ one can repeat the analysis given below 
	after replacing $z$ and $y$ by their complex conjugates. However, according
	to  \eqref{com_conj} this is equivalent to the replacement $N\to -N$ together with
	$\beta\to - \beta,\, m\to -m,\, u\to -\bar{u}$, and the proof given below does not depend
	on the latter three sign changes.

	Moreover, let us first assume $|\beta| \leq 1/2$. Denote $n_0 = N_0 - 1$ and split function $F$ given by~\eqref{F-factor},~\eqref{F-f} into four parts
	\begin{align}\label{Fsplit}
		\begin{aligned}
			& F(N, \beta; m, u, \epsilon) = \frac{F(N - n_0, \beta; m + n_0, u, \epsilon)}{F(N - n_0,\beta; n_0, 0, 0)} \; e^{-\pi \imath n_0 \epsilon \delta^2} \\[6pt]
			& \quad \times \biggl( \frac{ \sh \pi (N \delta + \imath \beta) }{ \sh \pi ([N - n_0] \delta + \imath \beta) } \biggr)^{ -\frac{m + \imath u}{2} } \,  \biggl( \frac{ \sh \pi (N \delta - \imath \beta) }{ \sh \pi ([N - n_0] \delta - \imath \beta) } \biggr)^{ \frac{m - \imath u}{2} }.
		\end{aligned}
	\end{align}
	By our assumptions,
	\begin{align} \label{NN0mu-assump}
		N - n_0 \geq 1, \qquad n_0 > 0, \qquad m + n_0 \geq -|m| + n_0 > | \Im u |.
	\end{align}
	Therefore, we can apply Lemma~\ref{lem:gamma-ratio} for two functions in the first line of~\eqref{Fsplit}
	\begin{align}
		& \bigl| \, \ln F(N - n_0, \beta; m + n_0, u, \epsilon) \bigr| \leq \frac{C_1 \delta}{1 - e^{-D_1 (N - n_0) \delta}} , \\[6pt]
		& \bigl| \, \ln F(N - n_0, \beta; n_0, 0, 0) \bigr| \leq \frac{C_2 \delta}{1 - e^{-D_2 (N - n_0) \delta}}.
	\end{align}
	Besides, the logarithm of exponent from the first line is bounded in obvious way
	\begin{align}
		\Bigl| \, \ln e^{-\pi \imath n_0 \epsilon \delta^2} \Bigr| = \pi n_0 |\epsilon| \delta^2 \leq C_3 \delta,
	\end{align}
	since $\epsilon(\delta) = O(1)$. 	Next rewrite the first ratio of hyperbolic sines
	\begin{align}\label{shsh-ratio}
		\biggl( \frac{ \sh \pi (N \delta + \imath \beta) }{ \sh \pi ([N - n_0] \delta + \imath \beta) } \biggr)^{ -\frac{ m + \imath u}{2} } = e^{-\frac{m + \imath u}{2} \, \pi n_0 \delta } \;  \biggl( \frac{ 1 - e^{-2\pi(N \delta + \imath \beta)} }{ 1 - e^{-2\pi([N - n_0]\delta + \imath \beta)}} \biggr)^{ -\frac{ m + \imath u}{2} },
	\end{align}
	so that for the function in brackets we can use Lemma~\ref{lem:ln} (with $\mu = n_0$)
	\begin{align}
		\Bigl| \, \ln \Bigl( 1 - e^{ -2\pi (N \delta + \imath \beta)} \Bigr) - \ln \Bigl( 1 - e^{ -2\pi ([N - n_0] \delta + \imath \beta)} \Bigr) \Bigr| \leq \frac{2\pi n_0 \delta }{1 - e^{-2\pi (N - n_0)\delta}}.
	\end{align}
	In addition, for the logarithm of exponential function in~\eqref{shsh-ratio} we have
	\begin{align}
		\Bigl| \, \ln e^{-\frac{m + \imath u}{2} \, \pi n_0 \delta } \Bigr| = \frac{|m + \imath u|}{2} \, \pi n_0 \delta \leq C_4 \delta.
	\end{align}
	The second ratio of hyperbolic sines in~\eqref{Fsplit} is estimated in the same way.
	
	Now, since $N \geq n_0 + 1$, we can use inequality
	\begin{align}\label{N0-ineq}
		1 - e^{-D (N - n_0) \delta} \geq 1 - e^{- \frac{D}{n_0 + 1} \, N \delta }
	\end{align}
	for any $D > 0$. As a result, the function $F(N, \beta; m, u, \epsilon)$ \eqref{Fsplit} is factorised into the product of functions $F_j$, whose logarithms are bounded~as
	\begin{align}\label{Fineq}
		| \ln F_j | \leq \frac{C_j \delta}{1 - e^{-D_j N \delta}}
	\end{align}
	with some constants $C_j \in (0, \infty)$, $D_j \in (0, \infty]$. By Lemma~\ref{lem:exp}, we therefore have bounds for the functions $F_j$ themselves
	\begin{align}
		\bigl| F_j^{\pm 1} - 1 \bigr|  \leq \frac{C_{\pm, j} \, \delta}{1 - e^{- D_j N \delta}}.
	\end{align}
	Finally, due to Lemma~\ref{lem:F-prod}, we have the same type of bound for the product of $F_j$, which 
equals $F(N, \beta; m, u, \epsilon)$~\eqref{Fsplit}, and analogously for its reciprocal $F^{-1}$.
	
	Thus, we proved the first statement of the proposition for $|\beta| \leq 1/2$.
If $|\beta| \in (1/2,1]$, we use the variables
	\begin{align}
		\beta' = \beta - \beta/|\beta|, \qquad  N' = N + \beta/|\beta|
	\end{align}
	instead of $\beta$ and $N$. Then we have $|\beta'|\leq 1/2$ and $N' \geq n_0 + 1$, which coincides with the conditions we used before. So, this case reduces to the previous one.
	
	Next we consider the second statement of proposition~\eqref{cor2-ineq2} with $N > 0$ (as before, the opposite case $N< 0$ follows by complex conjugation). The function $f$ is related $F$ in a very simple way \eqref{F-f}.
In this case we assume the upper bound $N \leq \lfloor M/\delta \rfloor$, hence,
	\begin{align}
		\Bigl| \, \ln e^{ \pm \pi\imath N \epsilon \delta^2 } \Bigr| = \pi N \delta \, | \epsilon | \delta \leq C \delta.
	\end{align}
	Therefore, one can use already proven bound on $F^{\pm 1}$ and again invoke Lemmas~\ref{lem:exp},~\ref{lem:F-prod} to obtain the same type of estimate for $f^{\pm 1}$.
\end{proof}

For large $|N|$ the above proposition implies the following useful corollary.

\begin{notation} \label{not:ozy2}
	Parametrize
	\begin{align} \label{param-2}
		\begin{aligned}
			& \quad \omega_1 = \bar{\omega}_2 = \imath + \delta, && \qquad\quad  \delta > 0, \\[6pt]
			& \quad z = \imath \sqrt{\omega_1 \omega_2} (N + \beta), && \qquad\quad N \in \mathbb{Z}, &&  \quad \beta \in \mathbb{R}, \\[6pt]
			& \quad y_j = \imath \sqrt{\omega_1 \omega_2} (m_j + u_j \delta + \epsilon_j(\delta) \delta^2), && \qquad\quad m_j \in \mathbb{Z}, &&  \quad u_j \in \mathbb{C}, && \quad j = 1,2,
		\end{aligned}
	\end{align}
	where $\epsilon_j(\delta) \in C^1[0, \dm]$ with some $\dm > 0$.
\end{notation}

\begin{corollary} \label{cor:g-ratio-bigN}
	With Notation~\ref{not:ozy2}, for any $\nu > 0$ there exist $C, \dm > 0$ such that
	\begin{align}
		\biggl| \frac{ \gamma^{(2)}(z + y_1) }{ \gamma^{(2)}(z + y_2) } \biggr| \leq C \, e^{\pi |N| \Im \bigl(\epsilon_2(\delta) - \epsilon_1(\delta)\bigr) \delta^2 } \, \Bigl| 2\sh \pi ( N\delta + \imath \beta) \Bigr|^{ \Im (u_2 -  u_1)}
	\end{align}
	for all $N, \delta, \beta$ satisfying
	\begin{align}
		| N | \geq \biggl\lfloor \frac{\nu}{\delta} \biggr\rfloor, \qquad 0 < \delta \leq \dm, \qquad | \beta | \leq 1.
	\end{align}
\end{corollary}

\begin{proof}
	By definitions~\eqref{F-factor},~\eqref{F-f},
	\begin{align}\label{gg-r}
		\biggl| \frac{ \gamma^{(2)}(z + y_1) }{ \gamma^{(2)}(z + y_2) } \biggr| = \biggl| \frac{F(N, \beta; m_1, u_1, \epsilon_1)}{F(N, \beta; m_2, u_2, \epsilon_2)} \biggr|  \; e^{\pi |N| \Im \bigl(\epsilon_2(\delta) - \epsilon_1(\delta)\bigr) \delta^2 } \, \Bigl| 2\sh \pi ( N\delta + \imath \beta) \Bigr|^{ \Im (u_2 -  u_1)}.
	\end{align}
	Function $F$ and its reciprocal are uniformly bounded due to Proposition~\ref{prop:g-ratio}
	\begin{align}
		\bigl| F^{\pm 1} \bigr| \leq 1 + \bigl| F^{\pm1} -1 \bigr| \leq 1 + \frac{C_1 \delta}{1 - e^{-C_2 |N| \delta}} \leq 1 + \frac{C_1 \delta}{1 - e^{-C_2 \delta}} \leq C,
	\end{align}
	where we use assumption $|N| \geq 1$ (choose $\dm \leq 1$) and the fact that function $C_1 \delta/(1 - e^{-C_2 \delta})$ is continuous for $\delta \in [0, \dm]$. Hence, the ratio of functions $F$ in \eqref{gg-r} is also uniformly bounded, which leads to the claimed statement. Notice that by assumption of this corollary
	\begin{align}
		|N| \geq \biggl\lfloor \frac{\nu}{\delta} \biggr\rfloor \geq \left\lfloor \frac{\nu}{\dm} \right\rfloor.
	\end{align}
	Therefore, one can choose $\dm$, so that the lower bound for $|N|$ from Proposition~\ref{prop:g-ratio} is satisfied, that is
	\begin{align}
		\left\lfloor \frac{\nu}{\dm} \right\rfloor \geq  |m_j| + \Bigl\lfloor | \Im u_j | \Bigr\rfloor + 3
	\end{align}
	for both $j = 1,2$.
\end{proof}

As one can see from Proposition~\ref{prop:g-ratio}, the drawback of removing restrictions on the parameters $m,u$ is the exclusion of small $|N|$. However, we need some uniform bound for these values too (including $N = 0$) to deal with the limits of hyperbolic integrals. For this we prove the following statement. Denote
\begin{align}
	\sign(N) = \left\{
	\begin{aligned}
		& + 1 && \;\; N \geq 0,\\
		& -1 && \;\; N < 0.
	\end{aligned} \right.
\end{align}

\begin{corollary} \label{cor:g-ratio-smallN}
	Using Notation~\ref{not:ozy2}, assume that $\Im u_1, \Im u_2 \not\in \mathbb{Z}$.
	Then for any $N_0 \in \mathbb{Z}_{>0}$ there exist some constants \mbox{$C, \dm > 0$} 
and $k \in \mathbb{Z}_{>0}$ such that
	\begin{align}\label{cor3-ineq}
		\biggl| \frac{ \gamma^{(2)}(z + y_1) }{ \gamma^{(2)}(z + y_2) } \biggr| \leq C \, \Bigl| 2\sh \pi ( N\delta + 2k \sign(N) \delta + \imath \beta) \Bigr|^{ \Im (u_2 -  u_1)}
	\end{align}
	for all $N, \delta, \beta$ satisfying
	\begin{align}
		| N | \leq N_0, \qquad 0 < \delta \leq \dm, \qquad | \beta | \leq 1.
	\end{align}
\end{corollary}

\begin{proof}
	As in the proof of Proposition~\ref{prop:g-ratio}, it is sufficient to consider $N \geq 0$, the opposite case 
follows from the  complex conjugation. Besides, it is sufficient to consider the case $|\beta| \leq 1/2$, since 
in the case $|\beta| \in (1/2, 1]$ we can shift the parameters 
$\beta' = \beta - \beta/|\beta|$, $m_j' = m_j + \beta/|\beta|$.
	
	Take the smallest $k \in \mathbb{Z}_{>0}$ such that
	\begin{align}
		2k \geq  |m_j| + \Bigl\lfloor | \Im u_j | \Bigr\rfloor + 3
	\end{align}
	for both $j = 1,2$ and factorise function in question
	\begin{align} \label{g-ratio2}
		\begin{aligned}
			\frac{\gamma^{(2)}(z + y_1)}{\gamma^{(2)}(z + y_2)} & = \frac{\gamma^{(2)}(z + y_1 + k\omega_1 - k \omega_2)}{\gamma^{(2)}(z + 2\imath \sqrt{\omega_1 \omega_2} k)} \; \frac{\gamma^{(2)}(z + 2\imath \sqrt{\omega_1 \omega_2} k)}{\gamma^{(2)}(z + y_2 + k\omega_1 - k \omega_2)} \\[6pt]
			& \times \frac{\gamma^{(2)}(z + y_1)}{\gamma^{(2)}(z + y_1 + k\omega_1 - k \omega_2)} \; \frac{\gamma^{(2)}(z + y_2 +  k\omega_1 - k \omega_2)}{\gamma^{(2)}(z + y_2)}.
		\end{aligned}
	\end{align}
	Notice that
	\begin{align}
		k\omega_1 - k\omega_2 = 2\imath k = \imath \sqrt{\omega_1 \omega_2} \bigl(2k + O(\delta^2)\bigr).
	\end{align}
	Since $N + 2k \geq |m_j| + \bigl\lfloor | \Im u_j | \bigr\rfloor + 3$, we can use Proposition~\ref{prop:g-ratio} for the first two ratios in~\eqref{g-ratio2}. Indeed,
	\begin{align}\label{g-F}
		\begin{aligned}
			\biggl| \frac{\gamma^{(2)}(z + y_j + k\omega_1 - k \omega_2)}{\gamma^{(2)}(z + 2\imath \sqrt{\omega_1 \omega_2} k)} \biggr| & = \bigl| f(N + 2k, \beta; m_j, u_j, \tilde{\epsilon}_j) \bigr| \\[6pt]
			& \times \Bigl| 2\sh \pi ( N\delta + 2k\delta + \imath \beta) \Bigr|^{ -\Im u_j},
		\end{aligned}
	\end{align}
	and from Proposition~\ref{prop:g-ratio} we deduce that both $f$ and $1/f$ are uniformly bounded.
	Hence,
	\begin{multline}\label{g-bound}
		\biggl| \frac{\gamma^{(2)}(z + y_1 + k\omega_1 - k \omega_2)}{\gamma^{(2)}(z + 2\imath \sqrt{\omega_1 \omega_2} k)} \; \frac{\gamma^{(2)}(z + 2\imath \sqrt{\omega_1 \omega_2} k)}{\gamma^{(2)}(z + y_2 + k\omega_1 - k \omega_2)} \biggr| \\[6pt]
		\leq C \, \Bigl| 2\sh \pi ( N\delta + 2k\delta + \imath \beta) \Bigr|^{ \Im u_2 - \Im u_1}.
	\end{multline}
	It is left to analyse two last ratios in~\eqref{g-ratio2}. First, using difference equations for gamma functions, we rewrite them as
	\begin{align}
		\frac{\gamma^{(2)}(z + y_j +  k\omega_1 - k \omega_2)}{\gamma^{(2)}(z + y_j)} = (-1)^k \, \prod_{s = 0}^{k - 1} \frac{ \sin \frac{\pi}{\omega_1} (z + y_j - \omega_2 - s \omega_2) }{ \sin \frac{\pi}{\omega_2} (z + y_j + s \omega_1) }.
	\end{align}
	Let us argue that under the assumption $\Im u_j \not\in \mathbb{Z}$ each factor in this product is bounded from both sides
	\begin{align}\label{s-ratio}
		C_1 \leq \Biggl|  \frac{ \sin \frac{\pi}{\omega_1} (z + y_j - \omega_2 - s \omega_2) }{ \sin \frac{\pi}{\omega_2} (z + y_j + s \omega_1) } \Biggr| \leq C_2, \qquad C_1, C_2 > 0.
	\end{align}
	Notice that together with the formula~\eqref{g-bound} this implies the statement~\eqref{cor3-ineq}.
	
	To prove~\eqref{s-ratio} we use the fact that
$		|\sin\, x|^2 =  \sin^2(\Re x) + \sh^2(\Im x).$
	In our case
	\begin{align}
		& \biggl| \sin \frac{\pi}{\omega_1} (z + y_j - \omega_2 - s \omega_2) \biggr|^2 =  \sin^2 \bigl( \pi \bigl[\beta + O(\delta) \bigr]  \bigr) + \sh^2 \bigl( \pi \delta \bigl[ \beta - a_j + O(\delta) \bigr] \bigr), \\[6pt]
		& \biggl| \sin \frac{\pi}{\omega_2} (z + y_j + s \omega_1) \biggr|^2 =  \sin^2 \bigl( \pi \bigl[\beta + O(\delta) \bigr]  \bigr) + \sh^2 \bigl( \pi \delta \bigl[ \beta - b _j + O(\delta) \bigr] \bigr),
	\end{align}
	where we denoted
	\begin{align}		& 
a_j: = - N - m_j - 2s + \Im u_j , 
\quad b_j: = -N - m_j - 2s - 2 - \Im u_j.
	\end{align}
	By the assumption $\Im u_j \not\in \mathbb{Z}$, hence, 
$a_j \not= 0$, $ b_j \not= 0$.
	In the double inequality~\eqref{s-ratio}, which now can be rewritten as
	\begin{align}\label{sinsin-b}
		C_1^2 \leq \frac{ \sin^2 \bigl( \pi \bigl[\beta + O(\delta) \bigr]  \bigr) + \sh^2 \bigl( \pi \delta \bigl[ \beta - a _j + O(\delta) \bigr] \bigr) }{ \sin^2 \bigl( \pi \bigl[\beta + O(\delta) \bigr]  \bigr) + \sh^2 \bigl( \pi \delta \bigl[ \beta - b _j + O(\delta) \bigr] \bigr) } \leq C_2^2,
	\end{align}
	it is sufficient to estimate the upper bound, since the lower one is of the same type.
	
	For the upper bound we need to verify 
that the denominator is bounded from below. This is plausible since $| \beta | \leq 1/2$ and $b_j \not= 0$. For simplicity, first consider the case \mbox{$|b_j| > 1/2$}. The opposite case is analogous, but requires additional small tricks.
	
For $|b_j| > 1/2$ and small enough $\delta \in [0, \dm]$ the combination of parameters in 
the denominator is bounded from below
	\begin{align}\label{beta-b}
		| \beta - b_j + O(\delta) | \geq | b_j | - | \beta | - C\dm \geq | b_j | - \frac{1}{2} - C \dm
	\end{align}
	uniformly in $\beta, \delta$. For the rest of arguments we have
	\begin{align}
		|  \beta - a _j + O(\delta) | \leq A, \qquad |\beta| - B_1\delta \leq | \beta + O(\delta) | \leq |\beta| + B_2 \delta.
	\end{align}
	Besides, for small enough $x > 0$ we have the inequalities
	\begin{align}\label{sinsh}
		\frac{x}{2} \leq \sin x \leq x, \qquad x \leq \sh x \leq 2x,
	\end{align}
	using which we consequently obtain
	\begin{align}
		\frac{ \sin^2 \bigl( \pi \bigl[\beta + O(\delta) \bigr]  \bigr) + \sh^2 \bigl( \pi \delta \bigl[ \beta - a _j + O(\delta) \bigr] \bigr) }{ \sin^2 \bigl( \pi \bigl[\beta + O(\delta) \bigr] \bigr) + \sh^2 \bigl( \pi\delta \bigl[ \beta - b _j + O(\delta) \bigr] \bigr) } \leq C \, \frac{ (|\beta| + B_2\delta)^2 + \delta^2 }{ (|\beta| - B_1\delta)^2 + \delta^2 }
	\end{align}
with some constant $C$ uniform in $\beta, \delta$. Writing the last ratio in terms of the variable 
$ \tilde{\beta} = \beta / \delta$, we obtain the function bounded on the whole real line $\tilde{\beta} \in \mathbb{R}$
	\begin{align}
		\frac{ ( |\tilde{\beta}| + B_2)^2 + 1 }{ (|\tilde{\beta}| - B_1)^2 + 1 } \leq D.
	\end{align}
	This proves the desired bound in the case $|b_j| > 1/2$.
	
	Now suppose $| b_j | \leq 1/2$. As we noted earlier, $b_j \not=0$. Besides,
	\begin{align}
		| \beta - b_j + O(\delta) | \leq | \beta - b_j | + C \dm.
	\end{align}
	Clearly, we can choose small enough $\dm$, so that
	\begin{align}
		0 \not\in [b_j - 2C\dm, b_j + 2C\dm].
	\end{align}
	For $\beta \not\in [b_j - 2C\dm, b_j + 2C\dm]$ the following combination
	\begin{align}
		| \beta - b_j + O(\delta) | \geq | \beta - b_j | - C \dm \geq C \dm > 0
	\end{align}
	is bounded from below, as in the previous case~\eqref{beta-b}. Hence, for such $\beta$ the ratio in question is uniformly bounded by the same arguments.
	
	It is left to analyse the values $\beta \in [b_j - 2C\dm, b_j + 2C\dm]$. Notice that the numerator 
is clearly bounded uniformly in $\beta, \delta$
	\begin{align}
		\sin^2 \bigl( \pi \bigl[\beta + O(\delta) \bigr]  \bigr) + \sh^2 \bigl( \pi \delta \bigl[ \beta - a _j + O(\delta) \bigr] \bigr) \leq A.
	\end{align}
	For the denominator let us simply use the fact
$		\sh^2 \bigl( \pi\delta \bigl[ \beta - b _j + O(\delta) \bigr] \bigr) \geq 0,
$
	and, as before, the bound for small enough $\delta$
	\begin{align}
		\sin^2 \bigl( \pi \bigl[\beta + O(\delta) \bigr] \bigr) \geq \frac{\pi^2}{4} \bigl[\beta + O(\delta) \bigr]^2 \geq \frac{\pi^2}{4} (|\beta| - B_1\delta)^2.
	\end{align}
	Consequently,
	\begin{align}\label{sinsin-r}
		\frac{ \sin^2 \bigl( \pi \bigl[\beta + O(\delta) \bigr]  \bigr) + \sh^2 \bigl( \pi \delta \bigl[ \beta - a _j + O(\delta) \bigr] \bigr) }{ \sin^2 \bigl( \pi \bigl[\beta + O(\delta) \bigr] \bigr) + \sh^2 \bigl( \pi\delta \bigl[ \beta - b _j + O(\delta) \bigr] \bigr) } \leq \frac{ \tilde{A} }{ (|\beta| - B_1\delta)^2 }.
	\end{align}
	Since $\beta$ belongs to the compact interval $[b_j - 2C\dm, b_j + 2C\dm]$, which does not contain $\beta = 0$, 
the denominator is bounded from below for small enough $\delta$:
$		|\beta| - B_1 \delta \geq \tilde{B} > 0.
$ 
	Thus, the ratio~\eqref{sinsin-r} is also bounded uniformly in $\beta, \delta$. This concludes the proof 
of the upper bound~\eqref{sinsin-b}.
\end{proof}

\section{Riemann sums and improper integrals} \label{sec:riem}

Approximations of improper Riemann integrals by the Riemann sums can fail in general, since two involved limits may not commute. In this section we check such approximations for particular functions of interest.
We start by considering the integral
\begin{align}
	I(\delta, a) = \delta \int_{- \frac{1}{2}}^{\frac{1}{2}} (\delta^2 + \beta^2)^{-a} \, d\beta
\end{align}
and proving the following auxiliary lemma.

\begin{lemma} \label{lem:I-bound}
	Let $a \in (0,1)$ and $\delta \in (0, 1]$. There exist constants $C_1(a), C_2(a)$ such that
	\begin{align}
		I(\delta, a) \leq C_1(a) \, \delta^{2(1 - a)} + C_2(a) \, \delta \ln \frac{1}{\delta}.
	\end{align}
\end{lemma}
\begin{remark}
	Depending on $a$ each of the terms can dominate as $\delta \to 0^+$.
\end{remark}
\begin{proof}
	Change the integration variable $\beta = x \delta$
	\begin{align}
		I(\delta, a)  = 2\delta^{2(1- a)} \int_0^{\frac{1}{2\delta}}  (1 + x^2)^{-a} \, dx.
	\end{align}
	Since $a \in (0, 1)$, we have
	\begin{align}
		\begin{aligned}
			\int_0^{\frac{1}{2\delta}}  (1 + x^2)^{-a} \, dx & \leq \int_0^{\frac{1}{2}} (1 + x^2)^{-a} \, dx + \int_{\frac{1}{2}}^{\frac{1}{2\delta}} x^{-2a} \, dx \\[8pt]
			& = \left\{ \begin{aligned}
				& C_1(a) + C_2(a) \, \delta^{2a - 1}, && \quad a \ne \frac{1}{2}, \\[6pt]
				& C_3 + \ln \frac{1}{\delta}, &&\quad a = \frac{1}{2}.
			\end{aligned} \right.
		\end{aligned}
	\end{align}
	From this estimate we obtain the claim.
\end{proof}

\subsection{Beta integral} \label{sec:riem-sum-beta}
Define
\begin{align}
	\Gd(\alpha) = \int_{- \frac{1}{2}}^{\frac{1}{2}} \frac{e^{-2\pi \imath (\alpha v + \beta k)}}{ (2\sh \pi (\alpha + \imath \beta))^{m + \imath u} \, (2\sh \pi (\alpha - \imath \beta))^{-m + \imath u}} \, d\beta,
\end{align}
where $m, k \in \mathbb{Z}$, $v \in \mathbb{R}$ and $\Im u \in (-1, -1/2]$. This integral is absolutely convergent for $\alpha \in \mathbb{R} \setminus \{0\}$, since
\begin{align} \label{G-bound}
	| \Gd(\alpha) | \leq \int_{- \frac{1}{2}}^{\frac{1}{2}} \bigl| 2 \sh \pi (\alpha 
+ \imath \beta) \bigr|^{2 \Im u} d\beta \leq \bigl| 2 \sh (\pi \alpha ) \bigr|^{2 \Im u} .
\end{align}
In Section~\ref{sec:beta-lim-calc} we use the following statement.

\begin{lemma} \label{lem:riem-sum-beta}
	Let $M, 1/\delta \in \mathbb{Z}_{>0}$. Then
	\begin{align}
		\lim_{\delta \to 0^+} \delta \sum_{\substack{ |N| \leq M/\delta,\; N \neq 0}} 
\Gd(N\delta) = \int_{-M}^M \Gd(\alpha) \, d\alpha.
	\end{align}
\end{lemma}

\begin{proof}
	Consider the half of this sum and integral for $N, \alpha > 0$, the proof for another half is analogous. 
By definition of the improper Riemann integral
	\begin{align}
		\int_{0}^M \Gd(\alpha) \, d\alpha= \lim_{\epsilon \to 0^+} \int_\epsilon^M \Gd(\alpha) \, d\alpha,
	\end{align}
	where the integral on the right is already a proper one, so it can be approximated in the standard way
	\begin{align}
		\int_\epsilon^M \Gd(\alpha) \, d\alpha =  \lim_{\delta \to 0^+} \Biggl( \Gd(\lceil \epsilon/\delta \rceil \delta) \, \bigl( (\lceil \epsilon/\delta \rceil + 1) \delta - \epsilon \bigr) +  \sum_{N = \lceil \epsilon/\delta \rceil + 1}^{M/\delta} \Gd(N\delta) \, \delta \Biggr).
	\end{align}
	Here we have partitioned the interval $[\epsilon, M]$ by the points
	\begin{align}
		\epsilon < (\lceil \epsilon/\delta \rceil + 1) \delta < (\lceil \epsilon/\delta \rceil + 2) \delta 
< \ldots < (M/\delta - 1) \delta < M.
	\end{align}
	First, let us show that for our function
	\begin{align}\label{f-lim}
		\lim_{\delta \to 0^+} \Gd(\lceil \epsilon/\delta \rceil \delta) \, \bigl( \lceil \epsilon/\delta \rceil  \delta - \epsilon \bigr) = 0,
	\end{align}
	which simplifies the above approximation to
	\begin{align}
		\int_\epsilon^M \Gd(\alpha) \, d\alpha =  \lim_{\delta \to 0^+} \delta \sum_{N = \lceil \epsilon/\delta \rceil }^{M/\delta} \Gd(N\delta).
	\end{align}
	Indeed, $\lceil \epsilon/\delta \rceil  \delta \geq \epsilon$, so from~\eqref{G-bound} we have
$		\bigl| \Gd(\lceil \epsilon/\delta \rceil \delta) \bigr| \leq ( 2\sh \pi \epsilon)^{2 \Im u},
$ 
	which implies the limit~\eqref{f-lim}.
	
	To prove the claim it is left to show that
	\begin{align}
		\lim_{\epsilon \to 0^+} \lim_{\delta \to 0^+} \delta \sum_{N = \lceil \epsilon/\delta \rceil }^{M/\delta} \Gd(N\delta) =  \lim_{\delta \to 0^+} \delta \sum_{N = 1}^{M/\delta} \Gd(N\delta),
	\end{align}
	or equivalently,
	\begin{align} \label{f-de-lim}
		\lim_{\epsilon \to 0^+} \lim_{\delta \to 0^+} \delta \sum_{N = 1}^{\lceil \epsilon/\delta \rceil} \Gd(N \delta) = 0.
	\end{align}
	Using the inequality
	\begin{align}
		|\sh \pi (N \delta + \imath \beta)|^2 = \sh^2(\pi N \delta) + \sin^2(\pi \beta) \geq C( (N\delta)^2 + \beta^2),
	\end{align}
	which is valid for $|\beta| \leq 1/2$, together with the bound~\eqref{G-bound} we have
	\begin{multline}
		\Biggl| \delta \sum_{N = 1}^{\lceil \epsilon/\delta \rceil} \Gd(N \delta) \Biggr| \leq C \delta \sum_{N = 1}^{\lceil \epsilon/\delta \rceil} \int_{-\frac{1}{2}}^{\frac{1}{2}} d\beta \; ( (N\delta)^2 + \beta^2)^{\Im u} \\
		\leq \sum_{N = 1}^{\lceil \epsilon/\delta \rceil} \frac{1}{N} \biggl( C_1 \, (N\delta)^{2(1 + \Im u)} + C_2 \, (N\delta) \ln \frac{1}{N\delta} \biggr),
	\end{multline}
	where on the last step we use Lemma~\ref{lem:I-bound}. Consider the first term in the brackets on 
the far right
	\begin{multline}
		\sum_{N = 1}^{\lceil \epsilon/\delta \rceil} \frac{1}{N}  (N\delta)^{2(1 + \Im u)} = \delta^{2(1 + \Im u)} \sum_{N = 1}^{\lceil \epsilon/\delta \rceil} N^{1 + 2 \Im u} \\
		\leq \delta^{2(1 + \Im u)} \int_{0}^{\lceil \epsilon/\delta \rceil} x^{1 + 2\Im u} \, dx = \frac{(\lceil \epsilon/\delta \rceil \delta)^{2(1 + \Im u)}}{2(1 + \Im u)},
	\end{multline}
	where we used the fact that $\Im u \in (-1, -1/2]$. The last expression tends to zero if we take 
the limits $\delta \to 0^+$ and then $\epsilon \to 0^+$. The remaining part
	\begin{align}
		\delta \sum_{N = 1}^{\lceil \epsilon/\delta \rceil} \ln \frac{1}{N\delta} = \delta \lceil \epsilon/\delta \rceil \ln \frac{1}{\delta} - \delta \ln \bigl( \lceil \epsilon/\delta \rceil! \bigr).
	\end{align}
	also tends to zero in the consecutive limits $\delta \to 0^+$, $\epsilon \to 0^+$ (use the bound $n! \geq n^n/e^{n - 1}$). This finishes the proof of~\eqref{f-de-lim} and, consequently, of this lemma.
\end{proof}

\subsection{Conical function}  \label{sec:riem-sum-con}

Define the integral
\begin{multline}
	\Gdc(\alpha) = \int_{-\frac{1}{2}}^{\frac{1}{2}} e^{-2\pi \imath (\alpha v + \beta k)} \, 	\bigl( 2\sh \pi (\alpha + \imath \beta) \; 2\sh \pi (\alpha - \rho + \imath \beta - \imath \sigma) \bigr)^{-m - \imath u} \\
	\times \bigl(2\sh \pi (\alpha - \imath \beta)  \; 2\sh \pi (\alpha - \rho - \imath \beta + \imath \sigma) \bigr)^{m - \imath u} \, d\beta,
\end{multline}
where $m, k \in \mathbb{Z}$, $v \in \mathbb{R}$, $\Im u \in (-1, 0)$, $\rho > 0$ and $|\sigma| \leq 1/2$. 
This integral is absolutely convergent for $\alpha \in \mathbb{R} \setminus \{0, \rho\}$ due to the bound
\begin{multline} \label{Gdc-bound}
	| \Gdc(\alpha) | \leq \int_{-\frac{1}{2}}^{\frac{1}{2}} \bigl| 2\sh \pi (\alpha + \imath \beta) \; 2\sh \pi (\alpha - \rho + \imath \beta - \imath \sigma) \bigr|^{2 \Im u} \, d\beta \\
	\leq \bigl| 2\sh (\pi \alpha) \; 2\sh \pi (\alpha - \rho) \bigr|^{2 \Im u}.
\end{multline}
In Section~\ref{sec:con-lim-calc} we use the following statement.

\begin{lemma}
	Let $M, 1/\delta \in \mathbb{Z}_{>0}$ such that $M \geq \rho + 2$. Then
	\begin{align}
		\lim_{\delta \to 0^+} \delta \Biggl( \sum_{N = -M/\delta}^{-1} + \sum_{N = 1}^{\lfloor \rho/\delta \rfloor - 1} + \sum_{N = \lfloor \rho/\delta \rfloor + 2}^{M/\delta} \Biggr) \, \Gdc(N\delta) = \int_{-M}^M \Gdc(\alpha) \, d\alpha.
	\end{align}
\end{lemma}

\begin{proof}
	Let us show that the limit of the last sum has the form
	\begin{align} \label{Gdc-int-riem}
		\lim_{\delta \to 0^+} \delta \sum_{N = \lfloor \rho/\delta \rfloor + 2}^{M/\delta} \Gdc(N\delta)
 = \int_{\rho}^M \Gdc(\alpha) \, d\alpha.
	\end{align}
The	treatment of the other two sums is similar. The integral over $\alpha$ is improper for $\Im u \leq -1/2$, so we write
	\begin{align}
		 \int_{\rho}^M \Gdc(\alpha) \, d\alpha = \lim_{\epsilon \to 0^+} \int_{\rho + \epsilon}^M \Gdc(\alpha) \, d\alpha.
	\end{align}
	Now the integral on the right can be approximated by the Riemann sums
	\begin{multline}
		\int_{\rho + \epsilon}^M \Gdc(\alpha) \, d\alpha = \lim_{\delta \to 0^+} \biggl( \Gdc\bigl(\lceil (\rho + \epsilon)/\delta \rceil \delta \bigr) \, \bigl( \bigl[\lceil (\rho + \epsilon)/\delta \rceil + 1\bigr]\delta - (\rho + \epsilon) \bigr) \\
		+ \sum_{N = \lceil (\rho + \epsilon)/\delta \rceil + 1}^{M/\delta} \Gdc(N\delta) \, \delta \biggr).
	\end{multline}
	Here we partitioned the interval $[\rho + \epsilon, M]$ by the points
	\begin{align}
		\rho + \epsilon < (\lceil (\rho + \epsilon)/\delta \rceil + 1) \delta < (\lceil (\rho + \epsilon)/\delta \rceil + 2) \delta < \ldots < (M/\delta - 1) \delta < M.
	\end{align}
	Since $\lceil (\rho + \epsilon)/\delta \rceil \delta \geq \rho + \epsilon$, from~\eqref{Gdc-bound} we have
$ 
		\bigl| \Gdc\bigl(\lceil (\rho + \epsilon)/\delta \rceil \delta \bigr) \bigr| \leq C(\rho, \epsilon),
$ 
	so that
	\begin{align}
		\lim_{\delta \to 0^+} \Gdc\bigl(\lceil (\rho + \epsilon)/\delta \rceil \delta \bigr) \, \bigl( \lceil (\rho + \epsilon)/\delta \rceil \delta - (\rho + \epsilon)  \bigr) = 0.
	\end{align}
	This simplifies the above approximation to
	\begin{align}
		\int_{\rho + \epsilon}^M \Gdc(\alpha) \, d\alpha = \delta \sum_{N = \lceil (\rho + \epsilon)/\delta \rceil}^{M/\delta} \Gdc(N\delta)  .
	\end{align}
	Hence, as in the previous section, to prove~\eqref{Gdc-int-riem} it is only left to show that
	\begin{align}
		\lim_{\epsilon \to 0^+} \lim_{\delta \to 0^+} \delta \sum_{N = \lfloor \rho/\delta \rfloor + 2}^{\lceil (\rho + \epsilon)/\delta \rceil} \Gdc(N\delta) = 0 .
	\end{align}
	For $N \geq \lfloor \rho/\delta \rfloor + 2 > \rho/\delta$ from the estimate~\eqref{Gdc-bound} we have
	\begin{align}
		| \Gdc(N\delta) | \leq \bigl| 2\sh (\pi \rho) \bigr|^{2\Im u} \int_{-\frac{1}{2}}^{\frac{1}{2}} \bigl| 2\sh \pi (N\delta - \rho + \imath \beta - \imath \sigma) \bigr|^{2 \Im u} \, d\beta.
	\end{align}
	Furthermore, in the last estimate we can take $\sigma = 0$, since the integral of periodic function (in~$\beta$) gives the same value over any period. For $\sigma = 0$ and $|\beta| \leq 1/2$ we have
	\begin{align}
		\bigl| 2\sh \pi (N\delta - \rho + \imath \beta) \bigr|^{2} \geq C ( (N\delta - \rho)^2 + \beta^2 ) \geq C ( (N - \lfloor \rho/\delta \rfloor - 1)^2 \delta^2 + \beta^2 ).
	\end{align}
	The rest of the proof is the same as at the end of the previous section.
\end{proof}


\begin{thebibliography}{99}
	

\bibitem{BCDK} N. Belousov, L. Cherepanov, S. Derkachov, S. Khoroshkin, \textit{Calogero--Sutherland hyperbolic system and Heckman--Opdam $\mathfrak{gl}_n$ hypergeometric function}, \href{https://doi.org/10.1007/s00029-026-01148-8}{Selecta Mathematica, New Series} \textbf{32} (2026), \href{https://doi.org/10.48550/arXiv.2508.18864}{\tt [2508.18864]}.


\bibitem{BSS2} N.~M.~Belousov, G.~A.~Sarkissian, V.~P.~Spiridonov, \textit{Complex rational Ruijsenaars model. The two-particle case}, \href{https://doi.org/10.54546/NaturalSciRev.100503}{Natural Science Review} \textbf{2} (2025) 100503, \href{https://doi.org/10.48550/arXiv.2508.12449}{\tt [2508.12449]}.


\bibitem{BSS} N.~M.~Belousov, G.~A.~Sarkissian, V.~P.~Spiridonov, \textit{Complex binomial theorem and pentagon identities}, \href{https://doi.org/10.1134/S0040577926010010}{Theoretical and Mathematical Physics} \textbf{226} (2026) 1--20, \href{https://doi.org/10.48550/arXiv.2412.07562}{\tt [2412.07562]}.

\bibitem{B} F. J. van de Bult, \textit{Ruijsenaars’ hypergeometric function and the modular double of $U_q(\mathfrak{sl}(2,\mathbb{C}))$}, \href{https://doi.org/10.1016/j.aim.2005.05.023}{Advances in Mathematics} \textbf{204} (2006) 539--571, \href{https://doi.org/10.48550/arXiv.math/0501405}{\tt [math/0501405]}.

\bibitem{BRS} F. J. van de Bult, E. M. Rains, J. V. Stokman, \textit{Properties of generalized univariate hypergeometric functions}, \href{https://doi.org/10.1007/s00220-007-0289-0}{Communications in Mathematical Physics} \textbf{275} (2007) 37--95, \href{https://doi.org/10.48550/arXiv.math/0607250}{\tt [math/0607250]}.

\bibitem{DLMF} \textit{NIST Digital Library of Mathematical Functions}. \href{https://dlmf.nist.gov/}{https://dlmf.nist.gov/}, Release 1.2.4 of 2025-03-15.
    F. W. J. Olver, A. B. Olde Daalhuis, D. W. Lozier, B. I. Schneider, R. F. Boisvert, C. W. Clark, B. R. Miller, B. V. Saunders, H. S. Cohl, and M. A. McClain, eds.
    
\bibitem{DSS} S. E. Derkachov, G. A. Sarkissian, V. P. Spiridonov, \textit{Elliptic hypergeometric function and $6j$-symbols for the $SL(2,\mathbb{C})$ group}, \href{https://doi.org/10.1134/S0040577922100087}{Theoretical and Mathematical Physics} \textbf{213} (2022) 1406--1422, \href{https://doi.org/10.48550/arXiv.2111.06873}{\tt [2111.06873]}.

\bibitem{F} L. D. Faddeev, \textit{Discrete Heisenberg-Weyl Group and modular group}, \href{https://doi.org/10.1007/BF01872779}{Letters in Mathematical Physics} \textbf{34} (1995) 249--254, \href{https://doi.org/10.48550/arXiv.hep-th/9504111}{\tt [hep-th/9504111]}.

\bibitem{GR} G. Gasper, M. Rahman, \textit{Basic hypergeometric series}, Cambridge University Press (2nd ed.), 2004.

\bibitem{GGR} I. M. Gelfand, M. I. Graev, V. S. Retakh, \textit{Hypergeometric functions over an arbitrary field}, \href{https://doi.org/10.1070/RM2004v059n05ABEH000771}{Russian Mathematical Surveys} \textbf{59} (2004) 831--905.

\bibitem{HR} M. Halln\"as, S. Ruijsenaars, \textit{Joint eigenfunctions for the relativistic Calogero--Moser Hamiltonians of hyperbolic type. I. First steps}, \href{https://doi.org/10.1093/imrn/rnt076}{International Mathematics Research Notices} \textbf{2014}:16 (2014) 4400--4456, \href{https://doi.org/10.48550/arXiv.1206.3787}{\tt [1206.3787]}.

\bibitem{I} I. Ip, \textit{Representation of the quantum plane, its quantum double and harmonic analysis on $GL^+_q(2,\mathbb{R})$}, \href{https://doi.org/10.1007/s00029-012-0112-4}{Selecta Mathematica, New Series} \textbf{19} (2013) 987--1082, \href{https://doi.org/10.48550/arXiv.1108.5365}{\tt [1108.5365]}.

\bibitem{K} T. H. Koornwinder, \textit{Jacobi functions as limit cases of $q$-ultraspherical polynomials}, \href{https://doi.org/10.1016/0022-247X(90)90026-C}{Journal of Mathematical Analysis and Applications} \textbf{148} (1990) 44--54.

\bibitem{MN} V. F. Molchanov, Yu. A. Neretin, \textit{A pair of commuting hypergeometric operators on the complex plane and bispectrality}, \href{https://doi.org/10.4171/jst/349}{Journal of Spectral Theory} \textbf{11}:2 (2021) 509--586, \href{https://doi.org/10.48550/arXiv.1812.06766}{\tt [1812.06766]}.

\bibitem{N} Y. A. Neretin, \textit{Barnes-Ismagilov integrals and hypergeometric functions of the complex field}, \href{https://doi.org/10.3842/SIGMA.2020.072}{SIGMA} \textbf{16} (2020) 072, \href{https://doi.org/10.48550/arXiv.1910.10686}{\tt [1910.10686]}.

\bibitem{PT} B. Ponsot, J. Teschner, \textit{Clebsch-Gordan and Racah-Wigner coefficients for a continuous series of representations of $U_q(\mathfrak{sl}(2,\mathbb{R}))$}, \href{https://doi.org/10.1007/PL00005590}{Communications in Mathematical Physics} \textbf{224} (2001) 613--655, \href{https://doi.org/10.48550/arXiv.math/0007097}{\tt [math/0007097]}.
	
\bibitem{Ra} E. M. Rains, \textit{Limits of elliptic hypergeometric integrals}, \href{https://doi.org/10.1007/s11139-007-9055-3}{Ramanujan Journal} {\bf 18} (2009) 257--306, \href{https://doi.org/10.48550/arXiv.math/0607093}{\tt [math/0607093]}.

\bibitem{R} S. N. M. Ruijsenaars, \textit{First order analytic difference equations and integrable quantum systems},
 \href{https://doi.org/10.1063/1.531809}{Journal of Mathematical Physics} \textbf{38} (1997) 1069--1146.

\bibitem{R3} S. N. M. Ruijsenaars, \textit{A relativistic hypergeometric function}, \href{https://doi.org/10.1016/j.cam.2004.05.024}{Journal of Computational and Applied Mathematics} \textbf{178} (2005) 393--417.

\bibitem{R2} S. Ruijsenaars, \textit{A relativistic conical function and its Whittaker limits}, \href{https://doi.org/10.3842/SIGMA.2011.101}{SIGMA} \textbf{7} (2011) 101, \href{https://doi.org/10.48550/arXiv.1111.0115}{\tt [1111.0115]}.
	
\bibitem{SS} G.~A.~Sarkissian, V.~P.~Spiridonov, \textit{The endless beta integrals}, \href{https://doi.org/10.3842/SIGMA.2020.074}{SIGMA} \textbf{16} (2020) 074, \href{https://doi.org/10.48550/arXiv.2005.01059}{\tt [2005.01059]}.

\bibitem{SSh} G. Schrader, A. Shapiro, \textit{On $b$-Whittaker functions}, arXiv preprint \href{https://doi.org/10.48550/arXiv.1806.00747}{\tt [1806.00747]}.

\bibitem{Sh} T. Shintani, \textit{On a Kronecker limit formula for real quadratic fields}, \href{https://irma.math.unistra.fr/~yalkinog/Shintani.pdf}{Journal of the Faculty of Science, the University of Tokyo} \textbf{24} (1977) 167--199.

\bibitem{essays} V. P. Spiridonov, \textit{Essays on the theory of elliptic
hypergeometric functions}, \href{https://doi.org/10.1070/RM2008v063n03ABEH004533}{Russian Mathematical Surveys} {\bf 63}:3 (2008) 405--472, \href{https://doi.org/10.48550/arXiv.0805.3135}{\tt [0805.3135]}.

\end{thebibliography}
\end{document}